\documentclass[11pt,a4paper]{article}   
\usepackage[utf8]{inputenc}
\usepackage[left=2.5cm, right=2.5cm, top = 2.54cm]{geometry}
\usepackage{graphicx}   
\usepackage{tikz}   
\usepackage{amsmath,amssymb,amsthm}  
\usepackage{mathtools}
\usepackage{hyperref}
\usepackage[numbers,sort&compress]{natbib}
\hypersetup{colorlinks=true, citecolor = cyan}   

\renewcommand{\cite}{\citet}

\usepackage{enumerate}
\usepackage{appendix}
\newtheorem{theorem}{Theorem}   

\newtheorem{lemma}[theorem]{Lemma}
\newtheorem{remark}[theorem]{Remark}

\newtheorem{definition}[theorem]{Definition}
\newtheorem{example}[theorem]{Example}
\newtheorem{assumption}[theorem]{Assumption}
\usepackage{amsmath}
\usepackage{algorithm}
\usepackage{algpseudocode}

\DeclarePairedDelimiter\abs{\lvert}{\rvert}
\DeclarePairedDelimiter\norm{\lVert}{\rVert}
\DeclareMathOperator*{\argmin}{arg\,min}

\newcounter{samCounter}

\title{Mean-Field Generalisation Bounds for Learning Controls in Stochastic Environments}
\author{Boris Baros\thanks{Mathematical Institute, University of Oxford, Oxford OX2 6GG, UK,
 {\tt \{boris.baros, samuel.cohen, christoph.reisinger\}@maths.ox.ac.uk}} \and Samuel Cohen\footnotemark[1] \and Christoph Reisinger\footnotemark[1]}
 \date{}

\begin{document} 

\maketitle
\begin{abstract}
    We consider a data-driven formulation of the classical discrete-time stochastic control problem. Our approach exploits the natural structure of many such problems, in which significant portions of the system are uncontrolled. Employing the dynamic programming principle and the mean-field interpretation of single-hidden layer neural networks, we formulate the control problem as a series of infinite-dimensional minimisation problems. When regularised carefully, we provide practically verifiable assumptions for non-asymptotic bounds on the generalisation error achieved by the minimisers to this problem, thus ensuring stability in overparametrised settings, for controls learned using finitely many observations. We explore connections to the traditional noisy stochastic gradient descent algorithm, and subsequently show promising numerical results for some classic control problems.
\end{abstract}

\section{Introduction}
\subsection{Motivation}
When solving stochastic control problems, one is often limited by the challenge of specifying realistic model dynamics of the involved processes. Parametric approaches to estimating dynamics introduce model error, while `model-free' approaches typically suffer from extreme curse of dimensionality constraints. The development of reliable machine-learning based methods for stochastic control is therefore of significant practical interest.

In this paper, we focus on problems where a decision maker faces a \emph{stochastic environment}, that is, where they interact with a system with unknown and uncontrolled stochastic dynamics, which, together with their control, induce a controlled state process and costs. Examples of this include optimal investment for a small investor -- here the stochastic dynamics of assets are uncontrolled and unknown, the investor chooses a strategy based on past observations, and together these generate a wealth process which must be optimised. A second example is aerial navigation in the presence of uncertain weather -- the weather is unaffected by the navigation policy chosen, while the navigator must account for uncertainties in their planning, and the resulting flight-plan needs to be optimised. In both these cases, the stochastic environment is naturally high-dimensional and may not be Markovian, and so is challenging to model statistically using finitely many observations.

We consider the setting where we have access to a finite number of i.i.d.~samples of trajectories of the stochastic environment, that is, historical paths which the environment has taken and which can be used to model future behaviour. Rather than using these to build an explicit statistical model of the environment, we investigate learning controls by direct optimisation against these historical scenarios, recasting the problem as empirical risk minimisation. As our actions can depend on the environment in a complex manner, it is natural to describe them using overparametrised models, such as neural networks. This overparametrised setting has been shown in \cite{overlearn} to suffer from unavoidably poor generalisation, due to information leakage from using full trajectories of the stochastic environment for training. In particular, overparametrised models learn to anticipate the training data, thus leading to poor out-of-sample performance. Consequently, the in-sample error does not serve as a good proxy for the out-of-sample error.

These statistical learning issues are the focus of this work. We consider a regularised variant of the empirical risk minimisation approach, modified to consider a data-driven stochastic control setting. Exploiting the benefits of dynamic programming and the mean-field interpretation of one-hidden layer neural networks, we are able to recast the problem as a backwards inductive series of infinite-dimensional minimisation problem. This allows gradient-based training to be viewed as a dynamical system in the space of measures, which is amenable to mathematical analysis. Based on this perspective, we demonstrate that entropy regularisation induces stability of generalisation of provably unique minimisers to the problem. 

Furthermore, we demonstrate that this formulation is also practical, by presenting an algorithm for its minimisation in the case of one-hidden layer neural networks. We also link this problem to the unregularised control problem, by deriving a scaling of the regularisation term which leads to a balance of bias and generalisation error for moderately large samples.

\subsection{Current Literature}
Sequential (discrete-time) decision-making problems, in particular those under uncertainty, arise ubiquitously across sectors, emerging in problems spanning engineering, biology, finance, transport, and beyond. As such, this class of problems has been studied extensively by a variety of disciplines with differing motivations, leading to significant advances both theoretically and computationally. For this reason, any literature review will necessarily be incomplete.

On the theoretical side, a rich theory exists surrounding existence and uniqueness of solutions (\cite{bertsekasShreve}), as well as attempts to weaken standard assumptions such as time consistency and Markovian dynamics (see, for example, the recent works \cite{me_myself_I,control_pham}). Whilst this direction provides tools necessary for computational methods, it fails to answer questions concerning modelling. Moreover, solutions are often intractable (\cite{stoch_control_survey}). These issues naturally lead to considering statistical techniques, which may shed light on real applications. 

Many \emph{simulation-based iterative methods} have since emerged to combine these fields, under the various names of reinforcement learning, approximate dynamic programming, neurodynamic programming and Monte-Carlo style algorithms; a selection of references includes \cite{RL_book, DL-SC, ML_SC, ent_christoph}. Whether based on machine learning or more statistical in nature, these estimation procedures typically depend on interactive access to stationary real-world systems or pre-calibrated simulators, in order to provide sufficient data for experimentation and learning. In addition, their performance often deteriorates with dimension, due to Bellman's `curse of dimensionality'.

In practice, \emph{synthetic data generation for large time-series models} may be difficult (\cite{synthetic}), particularly as time-series data are rarely stationary. Therefore, we often operate in settings with relatively small training sets, and training generators with good general performance is challenging. This motivates the main goal of this paper, which is to learn high-quality decision rules in high-dimensional settings with little training data. 

In dealing with the high-dimensional and non-linear aspects of such problems, it is common to \emph{parametrise controls with neural networks}, due to their desirable function approximation properties and dimensionality reduction capabilities (\cite{DL-SC,DL-finance, buehler2019deep}). The algorithm we propose is in this broad class, and is similar to the NNcontPI method in \cite{algo_paper_1, algo_paper_2}.

However, overparametrised sequential decision-making problems exhibit \emph{overlearning} (\cite{overlearn}), whose effects are parallel to what occurs with overfitting in classical supervised learning problems -- out-of-sample behaviour is unstable, and, in particular, we cannot use in-sample error as a good estimator for the out-of-sample error. This connects to the increasingly well-understood statistical properties of large neural networks: various models exist to analyse the training and out of sample performance of these methods, from more classical techniques such as Rademacher complexity (\cite{rad_gauss,overlearn, vapnik}) to highly specialised neural network models of learning, such as the neural tangent kernel (\cite{NTK}), random feature models (\cite{rand_features}), and mean-field analysis.

We focus on a \emph{mean-field formalism}, as in \cite{langevin, mean_field_mei, mean_field_OG,mean_field_games,mean-field-depth}. These  exhibit surprising connections to traditional noisy stochastic gradient descent algorithms via McKean--Vlasov SDEs and propagation-of-chaos results. We regularise this learning system using relative entropy, to formulate a minimisation problem in the space of probability measures over neural network parameters. In some sense, this is similar to the relaxed control approach considered in \cite{Wang2020RLContinuousTime, policy_entropy,ent_christoph, ent_christoph_2}, however our approach is fundamentally nonlinear, and ultimately yields classical feedback controls.
 
Building on these insights, we study the \emph{generalisation properties} of learned controls in a regularised overparametrised regime. This is related to unexpected contradictions of the classical bias-variance tradeoff (\cite{bias_variance, double_descent, empirical_ge, leave_one_out}).  In particular, the mean-field methods we consider have desirable generalisation bounds (\cite{sam_ge, NitandaWuSuzuki2021ParticleDualAveraging}), which we can apply in a control context. 
This is motivated by the view that, in the context of large control problems, stability of actions at a small cost in performance is often more desirable than an unstable algorithm which performs optimally in-sample.

\subsection{Main Contributions}
\par
Our contributions are as follows:
\begin{itemize}
    \item 
    We formalise a new data-driven framework for computationally solving finite-horizon stochastic control problems, inspired by recent results from mean-field neural networks. By lifting into an infinite-dimensional measure space and regularising with relative entropy, we derive a well-posed problem with a unique solution. 
    \item 
    We demonstrate general conditions (importantly, encompassing  nonlinearities in state dynamics and standard neural network activation functions) under which there exists an upper bound on the generalisation error of order $n^{-1}$, where $n$ is the size of the training dataset. This extends recent results from the mean-field approach in supervised learning (\cite{sam_ge}) to that of stochastic control by incorporating dynamic programming. 
    
    Importantly, to the best of our knowledge, this presents the first set of results to guarantee out-of-sample performance in data-driven stochastic control settings with low sample size. 
    \item 
    Using our $n^{-1}$ upper bound on generalisation error, we demonstrate a suitable scaling for the regularisation strength under which both bias and generalisation are bounded from above by terms of order $n^{-\frac{1}{2}}$. 
    \item 
    Leveraging results regarding Mean-Field Langevin Dynamics (MFLD) and Propagation of Chaos (PC), we are able to explicitly provide a training mechanism for the problem. Noting that the minimisers to our problem are unique, this demonstrates a guarantee on both bias and generalisation of our formulation. This is in contrast to standard applications of neural networks to data-driven stochastic control problems, where the possibilities of local minima ensure no such guarantee. 
\end{itemize}

The remaining paper is structured as follows. In Section \ref{problem} we will present the problem under consideration from both a classical stochastic control perspective and a data-driven empirical risk minimisation perspective. We briefly outline the overlearning phenomenon (Section \ref{overlrn}) associated with such an approach. We consider the parameterisation of controls using neural networks (Section \ref{NN parametrisation section}), and then reformulate the problem in terms of learning optimal distributions in the space of probability measures, rather than the standard finite-dimensional setting of learning optimal weights (Section \ref{Markov definition}). This alternative formulation involves exploiting the dynamic programming principle, invoking the mean-field interpretation of the one-hidden layer neural network parametrisation of actions, and finally adding an entropy regularisation term to the objective function. We finish this section by presenting key assumptions on the inputs to the problem in Assumption \ref{verifiable assumptions}.

In Section \ref{minimisers exist section}, we demonstrate a series of first-order conditions which characterise the unique minimisers to the problem (Theorem \ref{minimisers}). This characterises learning as evaluating a specific map from the empirical measure of the training data, whose generalisation error we can then analyse. In Section \ref{generalisation as stability section} we begin by noting (Theorem \ref{bousquet ge result}) that the generalisation error can be reformulated in terms of the stability of the expected empirical loss under resampling one point of the training data. We exploit this formulation to write the generalisation error in terms of linear functional derivatives of running costs and the minimising parameter measures, detailed in Theorem \ref{generalisation error}

We next demonstrate (Section \ref{bounding generalisation error section}) that our assumptions lead to an upper bound on the generalisation error of the minimising parameter measures. This upper bound is of order $n^{-1}$ (Theorem \ref{1/n theoretical}), where $n$ is the size of the training dataset, and is finite under reasonable assumptions regarding the moments of the data-generating distribution.  

Moving on to the computational aspects of the problem, in Section \ref{computation} we outline results which justify an extension of the traditional noisy stochastic gradient descent algorithm as a suitable algorithm for approximating the minimising parameter measure (Algorithm \ref{gibbs vector algo}). Section \ref{numerics} concludes the work by considering two concrete applications of our method. We demonstrate that our algorithm is feasible, empirically exhibits the theoretical behaviours demonstrated in earlier sections, and retains good in-sample performance (thus managing a balance of bias and stability, as discussed in Section \ref{balancing bias + stability}). 
\section{Problem Formulation}\label{problem}
\subsection{An Empirical Risk Minimisation Problem}
We focus our attention on a common instance of a classical control problem, where components of the state are uncontrolled, as considered in \cite[Example 2.2]{ndp} and \cite{overlearn}.
We call these components of the process the \textit{stochastic environment}.
For motivation and concreteness, we begin by giving the following case:

\begin{example}\label{flights}
    \par
    Consider the problem of navigating a plane from a start point to a destination. The controller attempts to specify an optimal sequence of velocities, where optimality is described by a combination of objectives, such as avoiding obstacles, minimising fuel usage, and reducing travel time. However, since the weather is random -- in particular the wind speed and direction --
    so will be the fuel consumption or travel time, and we are led to minimising an objective in expectation -- a typical stochastic control problem. We formalise this problem below in \eqref{full problem}.
    \par
    The approach we investigate exploits the fact that our control (the chosen velocity sequence) negligibly affects the weather. Therefore, we may view this part of the state vector as an uncontrolled process, the stochastic environment.
    \par 
    Supposing we have access to i.i.d.~realisations of the stochastic environment, and a model for the fuel consumption of the flight, we could evaluate controls offline and use the resulting performance as an estimator of the expected performance under the real distribution of the stochastic environment. 
    \par
    Traditional analytical methods would require explicit modelling of the weather, which is clearly difficult and subject to model uncertainty. Our approach circumvents this issue, the new challenges being how we might learn from data in such a context, and assess the performance out of sample. 
    The modeling requirement is reduced to the fuel consumption -- a considerably easier prospect (\cite{fuel_models}).
\end{example}
\par
Below we formalise such a setting. We consider a state process $X$, controlled by a process $U$, and dependent on an (uncontrolled) stochastic environment vector\footnote{In some contexts, the stochastic environment may be called the innovations process, and is often assumed to be a white noise process. We will instead allow $Z$ to have general (unknown) dynamics $\nu_{\mathrm{pop}}$.} $Z = \{Z_s\}_{s > 0}$. These take values in corresponding sets $\mathcal{X}, \mathcal{U},\mathcal{Z}^{T}$, which we assume to be subsets of (finite dimensional) Euclidean spaces. In this paper we will focus on the discrete-time case with a finite horizon $T$. The aim of the controller is then to specify some sequence of feedback controls $u_t:\mathcal{X}\to \mathcal{U}$, $t\in\mathbb{T} := \{0,\ldots, T\}$ and where measurability is implicit, to minimise the expectation of some cost functional. 
\par
We suppose that the stochastic environment $Z$ is distributed according to some unknown law $\nu_{\mathrm{pop}}$ on $\mathcal{Z}^T$, and generates a filtration $\mathbb{F} := (\mathcal{F}_t)_{t\in\mathbb{T}}$, where $\mathcal{F}_0 := \{\Omega, \emptyset\}$, and for $t \geq 1$ we define $\mathcal{F}_t := \sigma(\{Z_{s}\}_{s\leq t})$. Given a fixed initial state $x_0\in\mathcal{X}$, the state process $X^{u}(Z) = (X_t^{u}(Z))_{t\in\mathbb{T}}$ is defined recursively via a one-step transition function,
\begin{align} X_0^{u} := x_0, \quad X_{t+1}^u := h_t(X_t^u, u_t(X_t^u), Z_{t+1}), \quad t=0,\ldots, T-1.\label{transition function} \end{align}
We observe that the resulting state process $X^{u}(Z)$ is then $\mathbb{F}$-adapted. 
\begin{remark}
    Without much additional consideration it is simple to consider the case where $x_0$ is a random variable, independent of the stochastic environment $Z$, under the assumption that the distribution of $x_0$ has polynomial moments of sufficient order\footnote{Alternatively, we can simply start our problem at time $t=1$,  set $x_1=Z_0$, and formally initialise the system at $X_0=0$.}.
\end{remark}
\par
Denote by $\mathcal{C}$ the space of (sequences of) feedback controls $u = \{u_t\}_{t=0}^{T-1}$. The controller wishes to solve 
\begin{align} \mathrm{minimise}\ u\in\mathcal{C}\mapsto \mathbb{E}_{Z\sim\nu_{\mathrm{pop}}}\Big[\sum_{t=0}^{T-1} c_t(X_t^u(Z), u_t( X_t^u)) + \Phi(X_T^{u}(Z))\Big] =: \mathbb{E}_{Z\sim\nu_{\mathrm{pop}}}[\ell(X^u(Z), u)]. \label{full problem}
\end{align}
We make the following assumptions on the current framework for the scope of this work:
\begin{assumption}\label{formulation assumptions}
    \begin{enumerate}[i.]
    \item
    For every feedback control $u\in\mathcal{C}$, $X^u(Z)$ is a Markov process.
    \item 
    Transition functions $\{h_t\}_{t=0}^{T-1}:\mathcal{X}\times\mathcal{U}\times\mathcal{Z}\to \mathcal{X}$, and costs $\Phi:\mathcal{X} \to \mathbb{R}, \{c_t\}_{t=0}^{T-1}:\mathcal{X}\times\mathcal{U}\to\mathbb{R}$ are known and continuous in all their arguments.
    \end{enumerate}
\end{assumption}
\begin{remark}
    When we view $X$ as a feature vector, Assumption \ref{formulation assumptions}(i) is essentially an assumption regarding $X$ being rich enough that it is a sufficient statistic for describing the ``state'' of the problem at a given time. Albeit at the price of adding dimensions, this can always be achieved by taking $X_t$ to include all past observations and actions as components. 
\end{remark}
\par
The data-driven aspect of our approach arises from the fact that we do not have access to the distribution of the stochastic environment a priori. Instead, we have a set of sample paths of the stochastic environment $\{Z^{(i)}\}_{i=1}^{n}$, which generates a training distribution $\nu_n := \frac{1}{n}\sum_{i=1}^{n}\delta_{Z^{(i)}}$. The controller can then recast the problem as an empirical risk minimisation (ERM) problem, minimising an unbiased estimate of the expected loss; namely, 
\begin{align} \mathrm{minimise}\ u\in\mathcal{C}\mapsto \mathbb{E}_{Z\sim\nu_n}[\ell(X^u(Z), u)] = \frac{1}{n}\sum_{i=1}^n\ell(X^u(Z^{(i)}), u).\label{ERM problem}\end{align}
We emphasise that it is the paths $Z^{(i)} = \{Z_t^{(i)}\}_{t\in\mathbb{T}}$ which are i.i.d.~(as $i$ varies) and each path is \emph{not} assumed to be an i.i.d.~sequence (as $t$ varies).
\subsection{Overlearning}\label{overlrn}
The optimisation over the training distribution in \eqref{ERM problem} is a classical problem, with many standard statistical or machine learning methods available (\cite{ndp, DL-SC, ML_SC}). Suppose we represent $u$ using some parameters $\theta \in \Theta$ -- for example, using a neural network. One concern that arises is overfitting\footnote{In \cite{overlearn} this is described in stochastic control settings as \textit{overlearning}.}, where model parameters are fitted too closely to optimise the in-sample loss \eqref{ERM problem}, leading to parameters that generalise poorly when applied to \eqref{full problem}. If we consider an increasing sequence of sufficiently rich parameter spaces $\{\Theta_k\}_{k\in\mathbb{N}}$, under fairly mild conditions on $Z$,  \cite{overlearn} prove the asymptotic result
\begin{align} \limsup_{n\to\infty}\lim_{k\to\infty} \inf_{\theta\in\Theta_k} \frac{1}{n}\sum_{i=1}^{n}\ell(X^{u_\theta}(Z^{(i)}), u_\theta) \leq \inf_{u\in\mathcal{A}_{\mathrm{ant}}} \mathbb{E}_{Z}[\ell(X^u(Z), u)] < \inf_{u\in\mathcal{C}} \mathbb{E}_{Z}[\ell(X^u(Z), u)], \label{overlearning result}
\end{align}
where $u_\theta$ is the feedback control parametrised by $\theta\in\Theta_k$, and $\mathcal{A}_{\mathrm{ant}}$ denotes the space of anticipative controls, that is, controls which depend on the whole path of $Z$ (including future values), rather than just the current state.
\par
This result highlights that, in the overparametrised setting, the minimiser of the empirical risk outperforms (in-sample) not only the feedback control minimising the expected loss, but also the anticipative control minimising the empirical risk. This behaviour arises from the fact that, despite the resulting feedback control $u_\theta$ being $\mathbb{F}$-adapted, there is information leakage from the future when training the parameters. 
\begin{example}
    \par
    If we train an investment strategy using one observation of a stock process over two periods, the trained system can identify which initial stock price changes corresponded (in our training data) with increases of stock values in the second period. In these scenarios, the strategy will aim to invest as much as possible in the stock over the second period to obtain the highest wealth possible at the terminal time. If we then test the resulting strategy out-of-sample, on a stock with the same initial price change, and the value instead decreases overall, we will catastrophically fail, and incur the worst possible loss. 
    \par
    This is an example of a control problem whose empirical risk has a degenerate minimiser, that is, the optimal strategy would be an infinite initial investment. Importantly, it illustrates the instability of training with empirical risk in the overparametrised regime. We will demonstrate this empirically in Section \ref{numerics}. 
    \par
\end{example}
\begin{remark}
    As discussed in \cite{overlearn}, synthetic data generation is a means of curbing overlearning, demonstrating that in the underparametrised setting we see a convergence to the optimal control rather than the overlearned one. However, as discussed in Example \ref{flights}, this would require explicitly modelling the stochastic environment, which is the very issue we wish to avoid. In the following section, we introduce an adjusted version of the minimisation problem which we will demonstrate to have more desirable properties than the above minimisers. 
\end{remark}

\subsection{Control Parametrisation via Mean-Field Neural Networks}\label{NN parametrisation section}
In order to specify a training mechanism, we typically need to parametrise feedback controls. Due to their desirable function approximation properties, clear training procedure, and the mean-field analytic tools available for investigating their learning, we employ scaled one-hidden layer neural networks (also known as mean-field neural networks) to parametrise the feedback control at different time-points (\cite{mean_field_OG,mean_field_mei}).

Concretely, we fix a one-hidden layer neural network with $r$ hidden neurons to give a feedback control $u_{\theta^{(r)}}: \mathcal{X}\to\mathcal{U}$, via 
\begin{align} u_{\theta^{(r)}}(X) = \frac{1}{r}\sum_{j=1}^{r} a_{j}\sigma(w_{j}\cdot {X} + b_{j}) = \frac{1}{r}\sum_{j=1}^{r}\phi(\theta_{j}, X), \label{NN parametrisation}\end{align}
where $\theta_{j} := (a_{j}, w_{j}, b_{j})\in\Theta$, $\phi(\theta_{j}, X) := a_{j}\sigma(w_{j} \cdot X + b_{j}),$ and $\theta^{(r)}$ denotes the collection $\{\theta_{j}\}_{j=1}^{r}$.
Noting that we can write 
\[ \frac{1}{r}\sum_{j=1}^{r} \phi(\theta_{j}, X) = \int_{\Theta} \phi(\theta, X)\mathfrak{m}^r(\mathrm{d}\theta),\]
where $\mathfrak{m}^r := \frac{1}{r}\sum_{j=1}^{r} \delta_{\theta_{j}}$ is the empirical distribution of the parameters, we can instead proceed with a set of measure-parametrised feedback controls $u_{m}:\mathcal{X}\to \mathcal{U}$, where
\[ u_{m}(X) := \int_{\Theta} \phi(\theta,X)m(\mathrm{d}\theta) = \mathbb{E}_{\theta\sim m}[\phi(\theta,X)]. \]
From mean-field theory, the behaviour of these measure-parametrised controls is approximated well by the neural network parametrised case, with $r$ sufficiently large (\cite{mean_field_OG}). We will discuss this connection further in Section \ref{computation}, but for now we proceed with this infinite-dimensional parametrisation, affording us a rich hypothesis space. This allows us to view gradient-based training methods as dynamical systems in the space of probability measures over $\Theta$, which we may analyse using infinite-dimensional calculus. 
\par
As we will discuss further in Remark \ref{global vs sliced}, it will be convenient to separate the controls at different times, leading us to use a length $T$ vector $\mathbf{m} = (m_t)_{t=0}^{T-1},$ corresponding to the control at each time $t$.
For the sake of notational simplicity, from here onwards we often omit writing the control functions $u_m$, instead just writing $m$. For instance, the running cost $c_t(x, u_{m_t}(x))$ becomes $c_t(x, m_t)$. 
\subsection{Approximate Dynamic Programming and Entropy Regularisation}\label{Markov definition}
Recalling that the state process $X$ is a (potentially time-inhomogeneous) controlled Markov process for every measure-vector $\mathbf{m}$, we introduce some standard notation for the resulting Markov Decision Process (MDP). By $\{ P_t^{\mathbf{m}}(x, \mathrm{d}x')\}_{u\in\mathcal{C}, x\in\mathcal{X}, t\in\mathbb{T}}$ we denote the family of transition probabilities associated to the Markov process $X$, explicitly defined as $P_t^\mathbf{m}(x, \mathrm{d}x') := \mathbb{P}[h_t(x, u_{m_t}(x), Z_{t+1})\in \mathrm{d}x'].$ For any measurable function $F$ we have the pushforward notation
\[P_t^{\mathbf{m}}F(x) = \int F(x')P_t^{\mathbf{m}}(x, \mathrm{d}x') = \mathbb{E}[F(h_t(x, u_{m_t}(x), Z_{t+1}))] = \mathbb{E}[F(X_{t+1}^{\mathbf{m}}(Z))|X_t^{\mathbf{m}}(Z) = x],\]
where we observe that the $P_t^{\mathbf{m}}$ implicitly depend on the (unknown) law $\nu_{\mathrm{pop}}$ of $Z$.
Note the similarity of our MDP to that of \cite{algo_paper_1}, who instead assume the stochastic environment at each time is identical and independently distributed, and the transition function is time-constant. 
\par
Denoting the optimal value by 
\[ V_0(x_0) := \inf_{\mathbf{m}}\mathbb{E}_{Z}[\ell(X^\mathbf{m}(Z), \mathbf{m})],\]
the dynamic programming principle ensures that we may solve for $V_0(x_0)$ via backwards induction with terminal condition $V_T(x) = \Phi(x)$, followed by the system 
\[ \begin{cases}
      Q_t(x, m_t) := c_{t}(x, u_{m_t}(x)) + P_t^{\mathbf{m}} V_{t+1}(x), & x\in\mathcal{X} \\
      V_t(x) := \inf_{m_t} Q_t(x,m_t).
    \end{cases} \]
We call $Q_t$ the optimal state-measure value function and $V_t$ the optimal value function, similarly to \cite{RL_book}. 
It is a standard result that solving this recursion (assuming the infimum is attainable) provides an optimal control in feedback form $u^*_{\mathbf{m}}$, that is,
    \[ V_0(x_0) = \mathbb{E}_{Z}[\ell(X^{u^*_{\mathbf{m}}}(Z), u^*_{\mathbf{m}})].\]
\par
\begin{remark}\label{global vs sliced}
    It is worth noting that we have transitioned from aiming to learn a single control for all time points to learning separate controls at each time point -- specifically, we fit separate neural networks at each time $t$. 
    \par
    Others have proposed using one global function to represent controls (\cite{DL-SC, em}), followed by a single minimisation problem, rather than the dynamic sequence we propose. Whilst this may seem more desirable for large $T$, achieving a global minimiser with sufficient flexibility requires a highly complex model. In the case of neural network parametrisations, this can lead to exploding gradient issues (see \cite{exploding_gradients}). 
    \par
    In addition to this, we will demonstrate that unique minimisers in the dynamic approach can be guaranteed under more general convexity assumptions than when considering the global problem.
\end{remark}

One key issue arises from the fact that we do not have explicit access to the transition probabilities, rendering the problem unsolvable. Since the stochastic environment $Z$ is unaffected by the state function $X$ and the chosen control $u$, we can instead evaluate state trajectories for given controls over some i.i.d.\ training set $\{Z^{(i)}\}_{i=1}^n$. Taking the cost-to-go from some initial $x$ at time $t$ along these trajectories, and some chosen control $u$, provides an approximation of the Q-functions as follows. 
\par 
Given control measures $\mathbf{m} = (m_t)_{t=0}^{T-1}$, a training sample path $Z = (Z_t)_{t=1}^{T}$,  and state $x\in\mathcal{X}$, we define the \textit{empirical Q-functions} as
\begin{align} \widehat{Q}_t(x, m_t,m_{t+1},\ldots, m_{T-1},  Z) := \sum_{s\geq t} c_s^*(X_s^{t,x, \mathbf{m}}(Z), m_s),\label{ADP problem}\end{align}
where $X_s^{t,x,\mathbf{m}}(Z)$ denotes the state process defined by \eqref{transition function}, with initial condition $X_t^{t,x,\mathbf{m}}(Z) = x$, and following controls from $\mathbf{m}$ thereafter, and we use running costs $c_t^*$, defined by 
\[ c_t^*(X_t^{\mathbf{m}}(Z), m_t) := \begin{cases}
        c_t(X_t^{\mathbf{m}}(Z), m_t) & t < T-1, \\
        c_{T-1}(X_{T-1}^{\mathbf{m}}(Z), m_{T-1}) + \Phi(h_{T-1}(X_{T-1}^{\mathbf{m}}(Z), u_{m_{T-1}}(X_{T-1}^{\mathbf{m}}(Z)), Z_{T})) & t = T-1,
    \end{cases}
    \]
which will allow us to simplify notation when needed. In \cite{algo_paper_1} an approximate dynamic programming approach is analysed, where -- in our measure-controlled formulation -- the controller solves the upper-triangular series of minimisation problems
\[ m_t \mapsto \frac{1}{n}\sum_{i=1}^{n} \widehat{Q}_t(X_t^{\mathrm{ref}}(Z^{(i)}), m_t, m_{t+1:T-1}^*, Z^{(i)}), \quad t = T-1, \ldots, 0,\]
where the $m_{t+1:T-1}^* := \{m_s^*\}_{s=t+1}^{T-1}$ are the minimising measures for future steps (previously determined by backwards induction), and $X_t^{\mathrm{ref}}$ denotes the state controlled up to time $t$ by some pre-specified `reference control'. We adopt the reference control in order to eliminate dependence on the chosen measures for earlier timesteps, thus introducing the aforementioned upper-triangular structure of the problem. 
\begin{remark}
    It is important to note that we focus our attention on the generalisation error, that is, how much worse our control will perform out-of-sample than in-sample. This is important for our understanding of the approximate dynamic programming, as is highlighted by the following scenario: 
    \par
    Suppose the reference control results in $X^{\mathrm{ref}}_2 = x_2$, for some fixed value $x_2$. When training the control measure $m_2$, we cannot expect to have good performance for other values of $X_2$ as these are not explored during training. If we now consider the next step, where we train a control at time $t = 1$, which implicitly depends on the control $m_2$, we may not obtain a near-optimal solution to the overall MDP. 
    \par
    However, what we will show is that, with the appropriate regularisation, the generalisation error remains small -- the controls $m_1, m_2$ which we construct will continue to perform comparably in and out-of-sample, despite not being optimised at time $t = 2$ for the state $X_2^{m_1, x_1}$ which we obtain by following $m_1$ from state $x_1$. That is, our fitted control continues to generalise well, but approximate dynamic programming will not (without further assumptions) ensure convergence to an optimal control. 
    \par
    This highlights the importance of using a reference control (which may be randomised) that causes $X^{\mathrm{ref}}$ to explore the space well (as visualised in Appendix \ref{zermelo} for a navigation problem from Section \ref{numerics}), as this encourages controls to be learned which perform well when started from a variety of states. 
\end{remark}
What we have suggested so far minimises the original loss $\ell$ over some training set $\{Z^{(i)}\}_{i=1}^{n}$ with a rich action space parametrised by $\mathcal{P}(\Theta)^T$. Left like this, any minimisation will lead to overlearning, now at $T$ separate time-points. Inspired by results in the setting of supervised learning (\cite{sam_ge}), we add an entropy regularisation term to our loss function to combat the overlearning effect from \eqref{overlearning result}. 
\par

\par
Fixing some initial reference control, which generates state process $X^{\mathrm{ref}}(Z)$, we solve a lower-triangular series of backwards inductive minimisation problems, aiming to minimise, for each $t$, the map
\begin{align}\label{entropy_minimisation_problem} m_t \in \mathcal{P}_2(\Theta) \mapsto \mathbb{E}_{Z\sim\nu_n}\big[\widehat{Q}_t(X_t^{\mathrm{ref}}(Z), m, Z)\big] + \frac{\sigma^2}{2\beta^2}\mathrm{KL}(m_t||\gamma^\sigma),
\end{align}
where:
\begin{itemize}
    \item 
    $\mathcal{P}_2(\Theta)$ denotes the space of finite-variance probability measures over parameter space $\Theta$. See Appendix \ref{prereq} for more information regarding such spaces. 
    \item 
    For notational simplicity, we have omitted future control measures in $\widehat{Q}_t$.
    Where all measures are important we will sometimes use the notation $\widehat{Q}_t(x, m_t, \ldots, m_{T-1}, z)$ as in \eqref{ADP problem}, but the definition stays the same.  
    \item 
    $\sigma, \beta > 0$ are regularisation hyperparameters. We will see in Secton \ref{computation} that these values decouple and control different aspects of our eventual algorithm.
    \item 
    The Kullback--Leibler divergence is defined as 
    \begin{align*}
        \mathrm{KL}(m'||\gamma^\sigma) := 
        \begin{cases}
            \int_{\Theta} \log\Big(\frac{m'(\theta)}{\gamma^\sigma(\theta)}\Big)m'(\theta)\mathrm{d}\theta, \\
            \infty \quad \mathrm{otherwise}.
        \end{cases}
    \end{align*}
    Note that we will only work with finite-entropy measures, so will somewhat abuse notation and simply write $m(\theta)$ for the density of measure $m$ at $\theta$.
    \item 
    The Gibbs measure $\gamma^{\sigma}$ has density $\gamma^{\sigma}(\theta) = \frac{1}{F}\exp\Big\{-\frac{1}{\sigma^2} \Gamma(\theta)\Big\},$ where $\Gamma:\Theta\to\mathbb{R}$ is a regularisation potential, and $F$ is a normalisation constant.
\end{itemize}

\par
We will demonstrate useful properties of the minimisers of \eqref{entropy_minimisation_problem} under the following assumptions. Aside from Assumption \ref{verifiable assumptions}(i) (which we discuss in Remark \ref{rem:verifiable assumptions discussion}), these assumptions are easily verifiable in practice, being assumptions only regarding the inputs to the problem.
\begin{assumption}\label{verifiable assumptions}
    Concerning the costs and the state dynamics, for all $x\in\mathcal{X}, u\in\mathcal{U}, z\in\mathcal{Z}, t \in \mathbb{T}$, we assume there exists some $C > 0 $ such that:
    \begin{enumerate}[i.]
        \item 
        The empirical Q-functions $\widehat{Q}_t:\mathcal{X}\times\mathcal{P}_2(\Theta) \to \mathbb{R}$ are nonnegative, and for each $x\in \mathcal{X}$ are convex and $\mathcal{C}^2$ with respect to $m\in \mathcal{P}_2(\Theta)$ (see Appendix \ref{prereq} for a precise definition);
        \item
        The running costs $\{c_t\}_{t=0}^{T-1}$ and terminal cost $\Phi$ satisfy a quadratic growth condition,
        \[ \abs{c_t(x, u)} \leq C(1+\norm{x}^2 + \norm{u}^2), \quad \abs{\Phi(x)} \leq C(1+\norm{x}^2);\]
        \item 
        The derivatives of the running costs and terminal cost exist and satisfy linear growth conditions. That is, 
        \[ \norm{\nabla_x c_t(x, u)},\norm{\nabla_u c_t(x, u)} \leq C(1+\norm{x}+\norm{u}), \quad \norm{\nabla_x \Phi(x)} \leq C(1+\norm{x}); \]
        \item 
        The state transition functions $\{h_t\}_{t}$ satisfy a linear growth condition  
        \[ \norm{h_t(x, u, z)} \leq C(1+\norm{x}+\norm{u}+\norm{z});\]
        \item 
        The state transition functions $\{h_t\}_t$ are differentiable with respect to $x$ and $u$, and are Lipschitz continuous with Lipschitz constant $C$. 
    \end{enumerate}
    At the level of the neural network and the regularising potential, we assume that:
    \begin{enumerate}[i.]\setcounter{enumi}{5}
        \item 
        The activation function $\phi$ appearing in \eqref{NN parametrisation} satisfies $\norm{\phi(x, \theta)} \leq C(1+\norm{x})(1+\norm{\theta}^2)$;
        \item 
        For some $p \geq 4$, the regularising potential $\Gamma$ satisfies $\lim_{\norm{\theta}\to\infty}\frac{\Gamma(\theta)}{\norm{\theta}^p} = \infty$. 
    \end{enumerate}
\end{assumption}
\begin{remark}\label{rem:verifiable assumptions discussion}
We make a few comments on Assumption \ref{verifiable assumptions}:
    \begin{itemize}
        \item Assumption \ref{verifiable assumptions}(i) is somewhat restrictive, as it corresponds to the convexity of the $\widehat{Q}_t$ functions (or discrete Hamiltonian) for our problem, and is based on the interaction between the costs $c^*_t$ and the dynamics $h_t$. In the case where the $h_t$ are linear and $c^*_t(x,m)$ are convex for all $t$, it is easy to verify that this assumption is satisfied.
        
        This special case naturally occurs when we consider relaxed control problems (where the control $u$ is replaced by a probability measure over the control space, and hence the costs and dynamics are linear in $u$, and hence in its mean-field parametrisation $m$). The usual regularization methods used in these cases (see, for example, discussion in \cite{ent_christoph_2} or \cite{Wang2020RLContinuousTime}) are often designed to ensure the required convexity.
        
        In practice, this assumption does not appear critical, as mean-field training performs well in cases where the potential function is mildly not convex (see, for example, \cite{nonconvex}). 
        \item 
        Other assumptions on the growth conditions are possible, and will simply lead to differing powers in the upper bounds that we demonstrate later. 
        \item 
        Assumption \ref{verifiable assumptions}(iv, v) ensures that, under the reference control, the state process $X^{\mathrm{ref}}$ is well-defined, and that the state process is continuous with respect to changing the control. 
        \item 
        Regarding the growth conditions on the neural network (Assumption \ref{verifiable assumptions}(vi)), this assumption includes the ReLU activation function -- this is often missed when one makes smoothness assumptions instead. 
        \item 
        For the sake of proofs going forward, we assume, without loss of generality, that $C \geq 1$, and we allow it change line by line (however, $C$ will not depend on $Z$ or the chosen controls).
        \item 
        These assumptions are sufficient to ensure that \eqref{entropy_minimisation_problem} is weakly continuous when restricted to measures $m_t\in\mathcal{P}_2(\Theta)$, which are absolutely continuous with respect to $\gamma^\sigma$. 
    \end{itemize}
\end{remark}
\section{Existence and Uniqueness of Minimisers}\label{minimisers exist section}

We now characterise the minimisers of \eqref{entropy_minimisation_problem}. We denote the $\ell$-th moment of a measure $m$ by 
\[ E_m^{(\ell)} := \mathbb{E}_{\theta\sim m}[\norm{\theta}^\ell],\]
and will repeatedly make use of the fact (a consequence of H\"{o}lder's inequality) that
$\Big\Vert\sum_{i=1}^J x_i\Big\Vert^k \leq J^{k-1}\sum_{i=1}^J \norm{x_i}^k$
for $k > 1$ in order to make simplifications such as $(1+E_{m}^{(2)})^2 \leq 2(1+E_m^{(4)})$, and so on for higher powers and larger sums. This is somewhat crude, but allows significant algebraic simplification. 
\begin{theorem}\label{minimisers}
    Under Assumption \ref{verifiable assumptions}, there exists a unique vector of measures $\mathfrak{m}(\nu) = (\mathfrak{m}_t(\nu))_{t=0}^{T-1}$ simultaneously minimising the series of approximate dynamic programming problems \eqref{entropy_minimisation_problem}, which we call the Gibbs vector.
    \par
    Moreover, when $\nu \in \mathcal{P}_q(\mathcal{Z}^{T})$ for $q \geq T$, the $t$-th element of the Gibbs vector is the unique fixed point of the map $m \mapsto M_t( m, \nu)$, where $M_t: \mathcal{P}_p(\Theta) \times \mathcal{P}_q(\mathcal{Z}^T) \to \mathcal{P}_p(\Theta)$ is defined in terms of the density of its output, given by
    \begin{align} M_t(m,\nu;\mathrm{d}\theta) = \frac{1}{F_{\beta,\sigma, t}}\exp\left\{-\frac{2\beta^2}{\sigma^2}\left[\int_{\mathcal{Z}^{T}}\frac{\delta}{\delta m} \widehat{Q}_t(X_t^{\mathrm{ref}}(Z), m,Z;\theta)\nu(\mathrm{d}Z)+\frac{1}{2\beta^2} \Gamma(\theta)\right]\right\}\mathrm{d}\theta,\label{Szpruch first order condition}\end{align}
    for $t = T-1, \ldots, 0$, in which $F_{\beta,\sigma, t}$ is a normalisation constant and $p$ is the value described in Assumption \ref{verifiable assumptions}(vii).
\end{theorem}
\begin{proof}
    We proceed with proving the claim assuming $\mathcal{X}, \mathcal{U} \subset \mathbb{R}$. The multivariate case follows analogously with increasingly involved notation. 
    \par 
    At any time $t$, by admissibility of the Gibbs measure $\gamma^\sigma$ we first note that any potential minimisers will occupy the set 
    \[ \Big\{m_{t}\in\mathcal{P}_p(\Theta): \frac{\sigma^2}{2\beta^2}\mathrm{KL}(m_{t}||\gamma^{\sigma}) \leq \mathbb{E}_{Z\sim\nu}\big[\widehat{Q}_t(X_t^{\mathrm{ref}}(Z),\gamma^{\sigma}, Z)\big]\Big\}.\] 
    From \cite[Lemma 1.4.3]{weak_compactness} we note that this set is relatively compact in the weak topology, guaranteeing existence of minimisers to \eqref{entropy_minimisation_problem}.
    \par
    By Assumption \ref{verifiable assumptions}, the empirical Q-functions $\widehat{Q}_t$ are convex with respect to $m_t$, and the relative entropy is strictly convex with respect to $m_t$, so the regularised empirical Q-function from \eqref{entropy_minimisation_problem} is strictly convex with respect to $m_t$ and therefore we can guarantee existence of a unique minimiser. 
    \par
    From the first-order condition in \cite{langevin}, we may conclude that, for each $t$, $\mathfrak{m}_t^{\beta,\sigma}(\nu)$ satisfies
    \[ \frac{\delta}{\delta m}\int_{\mathcal{Z}^{T}} \widehat{Q}_t(X_t^{\mathrm{ref}}(Z), m, Z; \theta)\nu(\mathrm{d}Z) + \frac{\sigma^2}{2\beta^2}\log(m) + \frac{1}{2\beta^2}\Gamma(\theta) = F_t,\]
    where the $\{F_t\}_{t=0}^{T-1}$ are constants, and $\frac{\delta}{\delta m}$ denotes the linear functional derivative with respect to $m$, as defined in Definition \ref{frechet definition}, Appendix \ref{prereq}. Rearrangement yields the representation \ref{Szpruch first order condition}.
    \par
    We are just left to show that the given map represents a genuinely valid map into $\mathcal{P}_p(\Theta)$. Given the exponential form of $M_t$, it is sufficient to show that, for each $t$ and each $m\in \mathcal{P}_p(\Theta)$, we have 
    \[ -\frac{2\beta^2}{\sigma^2}\left[\frac{\delta}{\delta m}\int_{\mathcal{Z}^{T}} \widehat{Q}_t(X_t^{\mathrm{ref}}(Z), m,Z;\theta)\nu(\mathrm{d}Z)+\frac{1}{2\beta^2} \Gamma(\theta)\right] \leq -c\norm{\theta}^p,\]
    for some constant $c > 0$, for all $\theta$ sufficiently large.
    \par
    
    Computing directly,
    \begin{align*}
        \Big|\frac{\delta \widehat{Q}_t}{\delta m}(X_t^{\mathrm{ref}}(Z), m ,Z; \theta) \Big| & = \Big|\partial_u c_t^*\Big(\phi(X_t^{\mathrm{ref}}(Z), \theta) - \mathbb{E}_{\theta\sim m}[\phi(X_t^{\mathrm{ref}}(Z), \theta)]\Big) \\
        & \quad + \sum_{s > t} \Big(\partial_x c_s^* + (\partial_u c_s^*)( \partial_x u_{\mathfrak{m}_s(\nu)})\Big)\frac{\delta}{\delta m}X_s^{t, m, \mathfrak{m}(\nu)}(Z)\Big|,
    \end{align*}
    where $X_s^{t, m, \mathfrak{m}(\nu)}(Z)$ denotes the state process with initial condition $X_t^{t, m, \mathfrak{m}(\nu)}(Z) = X_t^{\mathrm{ref}}(Z)$, followed by control measure $m$ at time $t$, then the $(\mathfrak{m}_s(\nu))_{s > t}$ obtained from the prior minimisations in the backwards induction \eqref{entropy_minimisation_problem}. Writing $P(Z) := \prod_{s=0}^{T-1}(1+\norm{Z_{s+1}})$ for notational clarity, and applying the inequalities of Assumption \ref{verifiable assumptions}, Lemma \ref{frechet_state}, and Lemma \ref{norm x}, we see 
    \begin{align*}
        \Big|\frac{\delta \widehat{Q}_t}{\delta m}(X_t^{\mathrm{ref}}(Z), m ,Z; \theta) \Big| & \leq C(1+\norm{X_t^{\mathrm{ref}}(Z)})(1+E^{(2)}_{m})(1+\norm{\theta}^2 + E_m^{(2)})\prod_{s=t+1}^{T-1}(1+E_{\mathfrak{m}_s(\nu)}^{(2)}) \\
        & \leq C (1+\norm{x_0})P(Z)(1+E^{(2)}_{m})(1+\norm{\theta}^2 + E_m^{(2)})\prod_{s=t+1}^{T-1}(1+E_{\mathfrak{m}_s(\nu)}^{(2)}) \\
        & \leq C(1+\norm{x_0})P(Z) b(m,\theta)\prod_{s=t+1}^{T-1}(1+E^{(2)}_{\mathfrak{m}_s(\nu)}),
    \end{align*}
    where we have absorbed moments of the reference controls into $C$, and written
    \[ b(m, \theta) := (1+\norm{\theta}^2 + E_m^{(2)})^2.\]
    Here $E_m^{(2)} < \infty$ as $m \in \mathcal{P}_p(\Theta)$. 
    Since $p \geq 4$, and we are taking an inductive approach in assuming that the previously found measures $(\mathfrak{m}_s(\nu))_{s>t}$ are in $\mathcal{P}_p(\Theta)$, we know that 
    \[ \prod_{s=t+1}^{T-1} (1+E^{(2)}_{\mathfrak{m}_s(\nu)}) < \infty.\]
    Since $q \geq T$, by H\"{o}lder's inequality we may conclude that 
    \begin{align*}
    \mathbb{E}_{Z\sim\nu}[P(Z)] & \leq C\mathbb{E}_{Z\sim\nu}\Big[\prod_{s=0}^{T-1}\big(1+\norm{Z_{s+1}}\big)\Big] \\ & \quad \leq C\prod_{s=0}^{T-1}\mathbb{E}_{Z\sim\nu}\big[\big(1+\norm{Z_{s+1}}\big)^T\big]^{\frac{1}{T}} \leq C\mathbb{E}_{Z\sim\nu}\big[\big(1+\norm{Z}\big)^T\big] < \infty,  
    \end{align*}
    and so may write 
    \[ \Bigg|\int_{\mathcal{Z}^T} \frac{\delta \widehat{Q}_t}{\delta m}(X_t^{\mathrm{ref}}(Z), m, Z;\theta)\nu(\mathrm{d}Z)\Bigg| \leq C b(m,\theta),\]
    which also gives 
    \[ \int_{\mathcal{Z}^T}\frac{\delta \widehat{Q}_t}{\delta m}(X_t^{\mathrm{ref}}(Z), m, Z;\theta)\nu(\mathrm{d}Z) \geq -C b(m,\theta),\]
    where we have omitted $(1+\norm{x_0})$ since $\norm{x_0} < \infty$ by definition. 
    Finally, from Assumption \ref{verifiable assumptions}, since $p \geq 4$, for any $A > 0$ there exists some $M > 0$ such that $\norm{\theta} \geq M$ ensures
    \[ \Gamma(\theta) \geq A(\norm{\theta}^p + b(m,\theta)).\]
    For any $t$ we have 
    \[ -\frac{2\beta^2}
    {\sigma^2}\left[\frac{\delta}{\delta m}\int_{\mathcal{Z}^{T}} \widehat{Q}_t(X_t^{\mathrm{ref}}(Z), m,Z;\theta)\nu(\mathrm{d}Z)+\frac{1}{2\beta^2} \Gamma(\theta)\right] \leq \frac{2\beta^2}{\sigma^2}C b(m,\theta) - \frac{A}{\sigma^2}\norm{\theta}^p - \frac{A}{\sigma^2}b(m,\theta),\]
    so taking $A = 2\beta^2 C$ provides the required bound. 
    
\end{proof}
\section{Generalisation Error as Stability}\label{generalisation as stability section}

Returning to the overlearning result \eqref{overlearning result}, it is of interest to bound the difference in performance of the Gibbs vector $\mathfrak{m}(\nu_n)$ in and out-of-sample. This is characterised by the \textit{generalisation error}, which we denote by 
\[\mathrm{gen}(\mathfrak{m}(\nu_n), \nu_{\mathrm{pop}}) := \mathbb{E}_{\mathbf{Z}_n}\big[\mathbb{E}_{Z\sim\nu_{\mathrm{pop}}}[\ell(X^{\mathfrak{m}(\nu_n)}(Z), \mathfrak{m}(\nu_n))] - \mathbb{E}_{Z\sim\nu_n}[\ell(X^{\mathfrak{m}(\nu_n)}(Z), \mathfrak{m}(\nu_n))]\big].\]

The first important step in analysing the generalisation error involves noting a result from \cite{leave_one_out}, which allows us to characterise the generalisation error of the Gibbs vector measures as their stability under resampling. 
\begin{theorem}\label{bousquet ge result}
    Given a training set of i.i.d.~paths $\mathbf{Z}_n = \{Z^{(i)}\}_{i=1}^{n}$ with each $Z^{(i)}\sim\nu_{\mathrm{pop}}$, and a single resampled path $\widetilde{Z}^{(1)}\sim\nu_{\mathrm{pop}}$ independent of the training paths, we may rewrite the generalisation error as 
    \[ \mathrm{gen}(\mathfrak{m}(\nu_n), \nu_{\mathrm{pop}}) = \mathbb{E}_{\mathbf{Z}_n, \widetilde{Z}^{(1)}}\Big[\ell(X^{\mathfrak{m}(\nu_n)}(\widetilde{Z}^{(1)}), \mathfrak{m}(\nu_n)) - \ell(X^{\mathfrak{m}(\nu_{n,(1)})}(\widetilde{Z}^{(1)}), \mathfrak{m}(\nu_{n,(1)}))\Big],\]
    where $\nu_{n,(1)} := \nu_n + \frac{1}{n}\Big(\delta_{\widetilde{Z}^{(1)}} - \delta_{Z^{(1)}}\Big)$ denotes the resampled empirical distribution. 
\end{theorem}
Writing the generalisation error in this way is useful, since we may make repeated use of the fundamental theorem of calculus (both in standard terms and for linear functional derivatives) in order to write this object in terms of derivatives of known quantities. In particular, we observe a $1/n$ scaling from the fact that, for a general $\mathcal{C}^1$ function $F(m_t)$, from the definition of linear functional derivative we may write 
\begin{align*}
    & F(\mathfrak{m}_t(\nu_{n})) - F(\mathfrak{m}_t(\nu_{n,(1)})) \\ 
    & = \int_0^1 \int_\Theta \frac{\delta F}{\delta m_t}\Big(\mathfrak{m}_t(\nu_{n,(1)}) + \lambda\big(\mathfrak{m}_t(\nu_n) - \mathfrak{m}_t(\nu_{n,(1)})\big); \theta\Big)\big(\mathfrak{m}_t(\nu_n) - \mathfrak{m}_t(\nu_{n,(1)})\big)(\mathrm{d}\theta)\mathrm{d}\lambda \\
    & = \frac{1}{n}\int_0^1 \int_0^1 \int_\Theta \int_{\mathcal{Z}^{T}} \frac{\delta F}{\delta m_t}\Big(\mathfrak{m}_t(\nu_{n,(1)}) + \lambda\big(\mathfrak{m}_t(\nu_n) - \mathfrak{m}_t(\nu_{n,(1)})\big); \theta\Big) \\
    & \quad \times \frac{\delta \mathfrak{m}_t}{\delta \nu}\Big(\nu_{n,(1)} + \tilde\lambda\big(\nu_n - \nu_{n,(1)}\big); Z\Big)\Big(\delta_{Z^{(1)}} - \delta_{\widetilde{Z}^{(1)}}\Big)(\mathrm{d}Z)\mathrm{d}\theta\mathrm{d}\tilde\lambda\mathrm{d}\lambda,
\end{align*}
where we use the results from Appendix \ref{appB}, which guarantee that each $\mathfrak{m}_t$ is $\mathcal{C}^1$ when viewed as a map of a general measure $\nu\in\mathcal{P}_q(\mathcal{Z}^T)$, where $q$ is as described in Theorem \ref{minimisers}. The difficulty in our context arises from the interactions of controls at different times via the state variable, leading to the following, more involved, representation. 
\begin{theorem}\label{generalisation error}
    The generalisation error of the Gibbs vector $\mathfrak{m}(\nu_n)$ can be written as 
    \begin{align}
        &\mathrm{gen}(\mathfrak{m}(\nu_n), \nu_{\mathrm{pop}}) \nonumber \\& = \frac{1}{n}\mathbb{E}_{\mathbf{Z}_n, \widetilde{Z}^{(1)}}\Bigg[\int_0^1\int_0^1\int_\Theta \Bigg(\sum_{t=0}^{T-1}\Bigg\{\frac{\delta c_t^*}{\delta m_t}\Big(X_t^{\mathrm{ref}}(\widetilde{Z}^{(1)}), \mathfrak{m}_t^{\tilde{\lambda}};\theta\Big)\frac{\delta \mathfrak{m}_t}{\delta \nu}\Big(\nu_n^{\tilde{\lambda}_2}; Z\Big)\Bigg\} \nonumber \\
        & \quad +  \int_0^1 \sum_{t=1}^{T-1}\Bigg\{f_t^{\lambda}\big(\mathfrak{m}_{t-1:T-1}(\nu_{n,(1)})\big)\frac{\delta g_t}{\delta m_{t-1}}(\mathfrak{m}_{t-1}^{\tilde\lambda}; \theta)\frac{\delta \mathfrak{m}_{t-1}}{\delta \nu}\Big(\nu_n^{\tilde\lambda_2}; Z\Big) \nonumber \\
        & \quad + \sum_{s=t-1}^{T-1} \frac{\delta f_t^\lambda}{\delta m_s}\Big(\mathfrak{m}_{t-1:s-1}(\nu_{n,(1)}), \mathfrak{m}_s^{\tilde\lambda}, \mathfrak{m}_{s+1:T-1}(\nu_n); \theta\Big) \nonumber \\
        & \quad \quad \times g_t(\mathfrak{m}_{t-1}(\nu_n))\frac{\delta \mathfrak{m}_s}{\delta \nu}\Big(\nu_n^{\tilde\lambda_2}; Z\Big)\Bigg\}\mathrm{d}\lambda\Bigg)\Bigg|_{Z=Z^{(1)}}^{Z=\widetilde{Z}^{(1)}}(\mathrm{d}\theta)\mathrm{d}\tilde\lambda_2 \mathrm{d}\tilde\lambda\Bigg] \nonumber 
    \end{align}
    where for notational clarity we define, for $\mathbf{m}:=(m_l)_{l=0}^{T-1}$,
    \begin{align*}
        f_t^\lambda(m_{t-1}, m_t, \ldots, m_{T-1}) & := \frac{\partial \widehat{Q}_t}{\partial x}\Big(X_{t}^{\mathrm{ref}}(\widetilde{Z}^{(1)}) + \lambda\big(X_{t}^{t-1, \mathbf{m}}(\widetilde{Z}^{(1)}) - X_{t}^{\mathrm{ref}}(\widetilde{Z}^{(1)})\big), m_{t},\widetilde{Z}^{(1)}\Big), \\
        g_t(m_{t-1})& := X_{t}^{t-1, \mathbf{m}}(\widetilde{Z}^{(1)}) - X_{t}^{\mathrm{ref}}(\widetilde{Z}^{(1)}) \\
        & = h_{t-1}(X_{t-1}^{\mathrm{ref}}(\widetilde{Z}^{(1)}), u_{m_{t-1}}(X_{t-1}^{\mathrm{ref}}(\widetilde{Z}^{(1)})), \widetilde{Z}^{(1)}_{t}) - X_{t}^{\mathrm{ref}}(\widetilde{Z}^{(1)}), \\
        \mathfrak{m}_t^{\tilde\lambda} & := \mathfrak{m}_t(\nu_{n,(1)}) + \tilde\lambda\big(\mathfrak{m}_t(\nu_{n}) - \mathfrak{m}_t(\nu_{n,(1)})\big) \\
        \nu_n^{\tilde{\lambda}_2} & := \nu_{n,(1)} + \tilde{\lambda}_2(\nu_n - \nu_{n,(1)}).
    \end{align*}
\end{theorem}
\begin{proof}
    We begin by considering the generalisation error of each of the minimisation problems in \eqref{entropy_minimisation_problem}. That is, we consider
    \begin{align*} & \mathrm{gen}_t(\mathfrak{m}_t(\nu_n), \nu_{\mathrm{pop}}) := \mathbb{E}_{\mathbf{Z}_n, \widetilde{Z}^{(1)}}\Big[\widehat{Q}_t(X_t^{\mathrm{ref}}(\widetilde{Z}^{(1)}), \mathfrak{m}_t(\nu_n),\widetilde{Z}^{(1)}) - \widehat{Q}_t(X_t^{\mathrm{ref}}(\widetilde{Z}^{(1)}), \mathfrak{m}_t(\nu_{n,(1)}), \widetilde{Z}^{(1)})\Big].
    \end{align*}
    Expanding using \eqref{ADP problem}, we see 
    \begin{align*}
    & \mathrm{gen}_t(\mathfrak{m}_t(\nu_n), \nu_{\mathrm{pop}}) \\   
    & = \mathbb{E}_{\mathbf{Z}_n, \widetilde{Z}^{(1)}}\Big[c_t(X_t^{\mathrm{ref}}(\widetilde{Z}^{(1)}), \mathfrak{m}_t(\nu_n)) - c_t(X_t^{\mathrm{ref}}(\widetilde{Z}^{(1)}), \mathfrak{m}_{t}(\nu_{n, (1)}))\Big] \\
    & \quad + \mathbb{E}_{\mathbf{Z}_n, \widetilde{Z}^{(1)}}\Big[\widehat{Q}_{t+1}(X_{t+1}^{t, \mathfrak{m}(\nu_n)}(\widetilde{Z}^{(1)}), \mathfrak{m}_{t+1}(\nu_n), \widetilde{Z}^{(1)})- \widehat{Q}_{t+1}(X_{t+1}^{t, \mathfrak{m}(\nu_{n,(1)})}(\widetilde{Z}^{(1)}), \mathfrak{m}_{t+1}(\nu_{n,(1)}),\widetilde{Z}^{(1)})\Big] \\
    & = \mathbb{E}_{\mathbf{Z}_n, \widetilde{Z}^{(1)}}\Big[c_t(X_t^{\mathrm{ref}}(\widetilde{Z}^{(1)}), \mathfrak{m}_t(\nu_n)) - c_t(X_t^{\mathrm{ref}}(\widetilde{Z}^{(1)}), \mathfrak{m}_{t}(\nu_{n, (1)}))\Big] \\
    & \quad + \mathbb{E}_{\mathbf{Z}_n, \widetilde{Z}^{(1)}}\Big[\widehat{Q}_{t+1}(X_{t+1}^{t, \mathfrak{m}(\nu_n)}(\widetilde{Z}^{(1)}), \mathfrak{m}_{t+1}(\nu_n), \widetilde{Z}^{(1)})- \widehat{Q}_{t+1}(X_{t+1}^{\mathrm{ref}}(\widetilde{Z}^{(1)}), \mathfrak{m}_{t+1}(\nu_n), \widetilde{Z}^{(1)})\Big]\\ & \quad - \mathbb{E}_{\mathbf{Z}_n, \widetilde{Z}^{(1)}}\Big[\widehat{Q}_{t+1}(X_{t+1}^{t, \mathfrak{m}(\nu_{n,(1)})}(\widetilde{Z}^{(1)}), \mathfrak{m}_{t+1}(\nu_{n,(1)}), \widetilde{Z}^{(1)}) -\widehat{Q}_{t+1}(X_{t+1}^{\mathrm{ref}}(\widetilde{Z}^{(1)}), \mathfrak{m}_{t+1}(\nu_{n,(1)}), \widetilde{Z}^{(1)})\Big] \\
    & \quad + \mathbb{E}_{\mathbf{Z}_n, \widetilde{Z}^{(1)}}\Big[\widehat{Q}_{t+1}(X_{t+1}^{\mathrm{ref}}(\widetilde{Z}^{(1)}), \mathfrak{m}_{t+1}(\nu_n), \widetilde{Z}^{(1)}) - \widehat{Q}_{t+1}(X_{t+1}^{\mathrm{ref}}(\widetilde{Z}^{(1)}), \mathfrak{m}_{t+1}(\nu_{n,(1)}), \widetilde{Z}^{(1)})\Big] \\
    & = \mathbb{E}_{\mathbf{Z}_n, \widetilde{Z}^{(1)}}\Big[c_t(X_t^{\mathrm{ref}}(\widetilde{Z}^{(1)}), \mathfrak{m}_t(\nu_n)) - c_t(X_t^{\mathrm{ref}}(\widetilde{Z}^{(1)}), \mathfrak{m}_{t}(\nu_{n, (1)}))\Big] \\ 
    & \quad + \mathbb{E}_{\mathbf{Z}_n, \widetilde{Z}^{(1)}}\Bigg[\int_0^1 \frac{\partial \widehat{Q}_{t+1}}{\partial x}\Big(X_{t+1}^{\mathrm{ref}}(\widetilde{Z}^{(1)}) + \lambda(X_{t+1}^{t, \mathfrak{m}(\nu_n)}(\widetilde{Z}^{(1)}) - X_{t+1}^{\mathrm{ref}}(\widetilde{Z}^{(1)})), \mathfrak{m}_{t+1}(\nu_n), \widetilde{Z}^{(1)}\Big) \\
    & \quad \quad \quad \times \Big(X_{t+1}^{t, \mathfrak{m}(\nu_n)}(\widetilde{Z}^{(1)}) - X_{t+1}^{\mathrm{ref}}(\widetilde{Z}^{(1)})\Big)\mathrm{d}\lambda \\
    & \quad \quad - \int_0^1 \frac{\partial \widehat{Q}_{t+1}}{\partial x}\Big(X_{t+1}^{\mathrm{ref}}(\widetilde{Z}^{(1)}) + \lambda(X_{t+1}^{t, \mathfrak{m}(\nu_{n,(1)})}(\widetilde{Z}^{(1)}) - X_{t+1}^{\mathrm{ref}}(\widetilde{Z}^{(1)})), \mathfrak{m}_{t+1}(\nu_{n,(1)}), \widetilde{Z}^{(1)}\Big) \\
    & \quad \quad \quad \times \Big(X_{t+1}^{t, \mathfrak{m}(\nu_{n,(1)})}(\widetilde{Z}^{(1)}) - X_{t+1}^{\mathrm{ref}}(\widetilde{Z}^{(1)})\Big)\mathrm{d}\lambda \Bigg] + \mathrm{gen}_{t+1}(\mathfrak{m}_{t+1}(\nu_n), \nu_{\mathrm{pop}}),
    \end{align*}
    where we have used the fundamental theorem of calculus in order to write the $\frac{\partial \widehat{Q}_{t+1}}{\partial x}$ terms. 
    Using this recursion, we conclude that 
    \begin{align*}
        & \mathrm{gen}(\mathfrak{m}(\nu_n), \nu_{\mathrm{pop}}) = \mathrm{gen}_0(\mathfrak{m}_0(\nu_n), \nu_{\mathrm{pop}}) \\ & \quad = \mathbb{E}_{\mathbf{Z}_n, \widetilde{Z}^{(1)}}\Bigg[\sum_{t=0}^{T-1}\Big\{c_t^*(X_t^{\mathrm{ref}}(\widetilde{Z}^{(1)}), \mathfrak{m}_t(\nu_n)) - c_t^*(X_t^{\mathrm{ref}}(\widetilde{Z}^{(1)}), \mathfrak{m}_{t}(\nu_{n, (1)}))\Big\} \\
        & \quad \quad + \sum_{t=0}^{T-2}\Bigg\{\int_0^1 \frac{\partial \widehat{Q}_{t+1}}{\partial x}\Big(X_{t+1}^{\mathrm{ref}}(\widetilde{Z}^{(1)}) + \lambda(X_{t+1}^{t, \mathfrak{m}(\nu_n)}(\widetilde{Z}^{(1)}) - X_{t+1}^{\mathrm{ref}}(\widetilde{Z}^{(1)})), \mathfrak{m}_{t+1}(\nu_n), \widetilde{Z}^{(1)}\Big) \\
        & \quad \quad \quad \times \Big(X_{t+1}^{t, \mathfrak{m}(\nu_n)}(\widetilde{Z}^{(1)}) - X_{t+1}^{\mathrm{ref}}(\widetilde{Z}^{(1)})\Big)\mathrm{d}\lambda \\
        & \quad \quad - \int_0^1 \frac{\partial \widehat{Q}_{t+1}}{\partial x}\Big(X_{t+1}^{\mathrm{ref}}(\widetilde{Z}^{(1)}) + \lambda(X_{t+1}^{t, \mathfrak{m}(\nu_{n,(1)})}(\widetilde{Z}^{(1)}) - X_{t+1}^{\mathrm{ref}}(\widetilde{Z}^{(1)})), \mathfrak{m}_{t+1}(\nu_{n,(1)}), \widetilde{Z}^{(1)}\Big) \\
        & \quad \quad \quad \times \Big(X_{t+1}^{t, \mathfrak{m}(\nu_{n,(1)})}(\widetilde{Z}^{(1)}) - X_{t+1}^{\mathrm{ref}}(\widetilde{Z}^{(1)})\Big)\mathrm{d}\lambda\Bigg\}\Bigg]. 
    \end{align*}
    We now proceed to simplify using linear functional derivatives. For the $c_t^*$ terms this will be simple, as measures differ in only one argument. On the other hand, the $\frac{\partial\widehat{Q}_t}{\partial x}$ terms differ in a number of arguments, so we decompose these into further terms which differ by only a single argument -- this is cumbersome but conceptually simple. 
    \par
    
    Writing using $f_t^\lambda, g_t$ as defined above, we can more clearly perform the decomposition,
    \begin{align*}
        & \int_0^1 \Big(f_t^\lambda\big(\mathfrak{m}_{{t-1}:{T-1}}(\nu_n)\big)g_t(\mathfrak{m}_{t-1}(\nu_n))-f_t^\lambda\big(\mathfrak{m}_{t-1:T-1}(\nu_{n,(1)})\big)g_t(\mathfrak{m}_{t-1}(\nu_{n,(1)}))\Big)\mathrm{d}\lambda \\
        & = \int_0^1 \Big(f_t^\lambda\big(\mathfrak{m}_{t-1:T-1}(\nu_n)\big)g_t(\mathfrak{m}_{t-1}(\nu_n)) - f_t^\lambda\big(\mathfrak{m}_{t-1}(\nu_{n,(1)}),\mathfrak{m}_{t:T-1}(\nu_n)\big)g_t(\mathfrak{m}_{t-1}(\nu_n)) \\
        & \quad + f_t^\lambda(\mathfrak{m}_{t-1}(\nu_{n,(1)}),\mathfrak{m}_{t:T-1}(\nu_n))g_t(\mathfrak{m}_{t-1}(\nu_n)) - f_t^\lambda(\mathfrak{m}_{t-1:t}(\nu_{n,(1)}),\mathfrak{m}_{t+1:T-1}(\nu_n))g_t(\mathfrak{m}_{t-1}(\nu_n)) \\
        & \quad +\cdots \\
        & \quad + f_t^\lambda(\mathfrak{m}_{t-1:T-2}(\nu_{n,(1)}),\mathfrak{m}_{T-1}(\nu_n))g_t(\mathfrak{m}_{t-1}(\nu_n)) - f_t^\lambda(\mathfrak{m}_{t-1:T-1}(\nu_{n,(1)}))g_t(\mathfrak{m}_{t-1}(\nu_n)) \\
        & \quad + f_t^\lambda(\mathfrak{m}_{t-1:T-1}(\nu_{n,(1)}))g_t(\mathfrak{m}_{t-1}(\nu_n))- f_t^\lambda(\mathfrak{m}_{t-1:T-1}(\nu_{n,(1)}))g_t(\mathfrak{m}_{t-1}(\nu_{n,(1)}))
        \Big)\mathrm{d}\lambda \\
        & = \int_0^1\int_0^1\int_\Theta \Big(\frac{\delta f_t^\lambda}{\delta m_{t-1}}\big(\mathfrak{m}_{t-1}^{\tilde{\lambda}}, \mathfrak{m}_{t:T-1}(\nu_n);\theta\big)g_t(\mathfrak{m}_{t-1}(\nu_n))\big(\mathfrak{m}_{t-1}(\nu_n) - \mathfrak{m}_{t-1}(\nu_{n,(1)})\big) \\
        & \quad + \frac{\delta f_t^\lambda}{\delta m_t}\big(\mathfrak{m}_{t-1}(\nu_{n,(1)}), \mathfrak{m}_t^{\tilde{\lambda}},\mathfrak{m}_{t+1:T-1}(\nu_n);\theta\big)g_t(\mathfrak{m}_{t-1}(\nu_n))\big(\mathfrak{m}_t(\nu_n) - \mathfrak{m}_t(\nu_{n,(1)})\big) \\
        & \quad + \cdots \\
        & \quad + \frac{\delta f_t^\lambda}{\delta m_{T-1}}\big(\mathfrak{m}_{t-1:T-2}(\nu_{n,(1)}),\mathfrak{m}_{T-1}^{\tilde\lambda};\theta\big)g_t(\mathfrak{m}_{t-1}(\nu_n))\big(\mathfrak{m}_{T-1}(\nu_n) - \mathfrak{m}_{T-1}(\nu_{n,(1)}))\big) \\
        & \quad + f_t^\lambda(\mathfrak{m}_{t-1:T-1}(\nu_{n,(1)}))\frac{\delta g_t}{\delta m_{t-1}}(\mathfrak{m}_{t-1}^{\tilde\lambda};\theta)\big(\mathfrak{m}_{t-1}(\nu_n) - \mathfrak{m}_{t-1}(\nu_{n,(1)})\big)\Big)(\mathrm{d}\theta)\mathrm{d}\tilde\lambda\mathrm{d}\lambda.
    \end{align*}
    We can similarly simplify, for general $s$, 
    \begin{align*}
        \mathfrak{m}_{s}(\nu_n) - \mathfrak{m}_{s}(\nu_{n,(1)}) & = \int_0^1\int_{\mathcal{Z}^{T}}\frac{\delta \mathfrak{m}_s}{\delta \nu}\Big(\nu_n^{\tilde{\lambda}_2}; Z\Big)\Big(\nu_n - \nu_{n,(1)}\Big)(\mathrm{d}Z)\mathrm{d}\tilde{\lambda}_2 \\
        & = \frac{1}{n}\int_0^1 \int_{\mathcal{Z}^{T}}\frac{\delta \mathfrak{m}_s}{\delta \nu}\Big(\nu_n^{\tilde{\lambda}_2}; Z\Big)\Big(\delta_{Z^{(1)}} - \delta_{\widetilde{Z}^{(1)}}\Big)(\mathrm{d}Z)\mathrm{d}\tilde{\lambda}_2 \\
        & = \frac{1}{n}\int_0^1 \frac{\delta \mathfrak{m}_s}{\delta \nu}\Big(\nu_n^{\tilde{\lambda}_2}; Z\Big)\Big|^{Z = Z^{(1)}}_{Z = \widetilde{Z}^{(1)}}\mathrm{d}\tilde{\lambda}_2.
    \end{align*}
    We are left to rewrite 
    \[ \mathbb{E}_{\mathbf{Z}_n, \widetilde{Z}^{(1)}}\Bigg[\sum_{t=0}^{T-1}\Big\{c_t^*(X_t^{\mathrm{ref}}(\widetilde{Z}^{(1)}), \mathfrak{m}_t(\nu_n)) - c_t^*(X_t^{\mathrm{ref}}(\widetilde{Z}^{(1)}), \mathfrak{m}_{t}(\nu_{n, (1)}))\Big\}\Bigg],\]
    which follows similarly, first taking the linear functional derivative in $m$, then the linear functional derivative in $\nu$. 
\end{proof}
\begin{remark}\label{L2 generalisation remark}
    A similar reformulation as in Theorem \ref{generalisation error} is possible for the $L_2$ generalisation error. With $q$ as in Theorem \ref{minimisers}, for $\nu \in \mathcal{P}_q(\mathcal{Z}^{T})$ we denote
    \begin{align}
        R_t(\nu, m_{t:T-1}) := \mathbb{E}_{Z\sim\nu}\big[\widehat{Q}_t(X_t^{\mathrm{ref}}(Z), m_{t:T-1}, Z)\big],
    \label{risk definition}\end{align}
    and define the $L_2$ generalisation error to be 
    \begin{align}
        & \mathbb{E}_{\mathbf{Z}_n}\Big[\big(R_0(\nu_{\mathrm{pop}}, \mathfrak{m}(\nu_n)) - R_0(\nu_n, \mathfrak{m}(\nu_n))\big)^2\Big] \nonumber \\
        & \quad = \mathbb{E}_{\mathbf{Z}_n}\Big[\Big(\mathbb{E}_{Z\sim\nu_{\mathrm{pop}}}[\ell(X^{\mathfrak{m}(\nu_n)}(Z), \mathfrak{m}(\nu_n))] - \mathbb{E}_{Z\sim\nu_n}[\ell(X^{\mathfrak{m}(\nu_n)}, \mathfrak{m}(\nu_n))]\Big)^2\Big]. \label{L2 error}
    \end{align} 
    We further discuss possible bounds on such an object in Remark \ref{Monte Carlo comparisons}.
\end{remark}
We now prove a result relating the moments of $\mathfrak{m}_t(\nu)$, which are stochastic if $\nu$ is allowed to be random, to a deterministic upper bound in terms of the moments of $Z$ and the measure $\tilde{\gamma}_p^\sigma$, which we define by its density
\[ \tilde{\gamma}_p^\sigma(\mathrm{d}\theta) := \frac{1}{\tilde{F}^\sigma}\exp\Big\{-\frac{1}{\sigma^2}\Gamma(\theta) + \norm{\theta}^p\Big\}\mathrm{d}\theta.\]
Due to Assumption \ref{verifiable assumptions}(vii), we know that $\tilde\gamma_p^\sigma \in \mathcal{P}_p(\Theta)$ and we observe that, for all $m \in \mathcal{P}_p(\Theta)$, 
\[ \mathrm{KL}(m||\gamma^\sigma) = \mathrm{KL}(m||\tilde{\gamma}_p^\sigma) + E_m^{(p)} + \mathrm{constant},\]
where the constant is independent of $m$. 
Therefore, considering the minimisers of the adjusted upper-triangular minimisation problems
\begin{align}\label{tilde_entropy_minimisation_problem} m_t \in \mathcal{P}_2(\Theta) \mapsto \mathbb{E}_{Z\sim\nu_n}\big[\widehat{Q}_t(X_t^{\mathrm{ref}}(Z), m, Z)\big] + E_{m_t}^{(p)}+\frac{\sigma^2}{2\beta^2}\mathrm{KL}(m_t||\tilde\gamma_p^\sigma), \quad t= T-1, \ldots, 0,
\end{align}
we see that the Gibbs vector $\mathfrak{m}(\nu_n)$ is also a minimiser here. In Lemma $\ref{tilde gamma lemma}$ we exploit the suboptimality of $\tilde\gamma_p^\sigma$ in solving \eqref{tilde_entropy_minimisation_problem} to derive the required upper bounds. 
\begin{lemma}\label{tilde gamma lemma}
    For $p$ as in Assumption \ref{verifiable assumptions},  there exists $C > 0$ such that for each $t \in \mathbb{T}$,
    \[ E_{\mathfrak{m}_t(\nu)}^{(p)} \leq TE_{\tilde{\gamma}_p^\sigma}^{(p)} + C\mathbb{E}_{Z\sim\nu}[P(Z)^2],\]
    where $\nu\in\mathcal{P}_q(\mathcal{Z}^{T})$ with $q$ as in Theorem \ref{minimisers}.
    
    In particular, there exists some $C > 0$  such that 
    \[ \prod_{t=0}^{T-1} (1+E_{\mathfrak{m}_t(\nu)}^{(p)}) \leq C\mathbb{E}_{Z\sim\nu}[P(Z)^{2T}],\]
    with $P(Z)$ as defined in Theorem \ref{minimisers}.
\end{lemma}
\begin{proof}
We begin by noting, from dynamic programming, that 
\[ \mathfrak{m}(\nu) = \argmin_{(m_t)_{t=0}^{T-1}\subset \mathcal{P}_p(\Theta)} \Bigg\{ \mathbb{E}_{Z\sim\nu}\big[\widehat{Q}_0(x_0, \mathfrak{m}(\nu), Z)\big] + \sum_{t=0}^{T-1} \Big(\frac{\sigma^2}{2\beta^2}\mathrm{KL}(\mathfrak{m}_t(\nu)||\tilde\gamma_p^\sigma) + E_{\mathfrak{m}_t(\nu)}^{(p)}\Big)\Bigg\}.\]
From admissibility of $(\tilde\gamma_p^\sigma, \ldots, \tilde\gamma_p^\sigma)$ for the above objective, we see that
\begin{align*}  & \mathbb{E}_{Z\sim\nu}\big[\widehat{Q}_0(x_0, \mathfrak{m}(\nu), Z)\big] + \sum_{t=0}^{T-1} \Big(\frac{\sigma^2}{2\beta^2}\mathrm{KL}(\mathfrak{m}_t(\nu)||\tilde\gamma_p^\sigma) + E_{\mathfrak{m}_t(\nu)}^{(p)}\Big) \\ & \quad \leq \mathbb{E}_{Z\sim\nu}\big[\widehat{Q}_0(x_0, \tilde{\gamma}_p^\sigma, \ldots, \tilde{\gamma}_p^\sigma, Z)\big] + T E_{\tilde\gamma_p^\sigma}^{(p)}. \end{align*}
In particular, from nonnegativity of all left-hand terms, for any $t$,
\[ E_{\mathfrak{m}_t(\nu)}^{(p)} \leq \mathbb{E}_{Z\sim\nu}\big[\widehat{Q}_0(x_0, \tilde\gamma_p^\sigma, \ldots, \tilde\gamma_p^\sigma, Z)\big] + T E_{\tilde\gamma_p^\sigma}^{(p)}.\]
Noting now the quadratic growth conditions from Assumption \ref{verifiable assumptions}(ii, vi), we may bound
\begin{align*}
    \widehat{Q}_0(x_0, \tilde\gamma_p^\sigma, \ldots, \tilde\gamma_p^\sigma, Z) & = \sum_{t=0}^{T-1}c_t^*(X_t^{\tilde\gamma_p^\sigma}(Z), \tilde\gamma_p^\sigma) \\ 
    & \leq C\sum_{t=0}^{T-1}(1+\norm{X_t^{\tilde\gamma_p^\sigma}(Z)}^2)(1+E_{\tilde\gamma_p^\sigma}^{(2)})^2 \\
    & \leq C(1+\norm{x_0}^2)\sum_{t=0}^{T-1}\Bigg((1+E_{\tilde{\gamma_p^\sigma}}^{(2)})^2\prod_{s=0}^{t-1}(1+E_{\tilde\gamma_p^\sigma}^{(2)})^2 (1+\norm{Z_{s+1}})^2\Bigg) \\
    & \leq C(1+\norm{x_0}^2)P(Z)^2 (1+E_{\tilde\gamma_p^\sigma}^{(2)})^{2T}.
\end{align*}
Taking expectations over $\nu$ and absorbing $(1+\norm{x_0}^2)(1+E_{\tilde{\gamma}_p^\sigma}^{(2)})^{2T}$ into $C$ provides the claim.
\end{proof}

Note that, as discussed in the proof of \cite[Lemma 5.3]{sam_ge}, Jensen's inequality implies that the above bound then holds for all moments of $\mathfrak{m}_t(\nu)$, up to the $p$-th moment. 
\section{Bounding the Generalisation Error}\label{bounding generalisation error section}
\begin{theorem}\label{1/n theoretical}
    Suppose that Assumptions \ref{verifiable assumptions} holds. Then the unique minimiser to the stochastic control problem \eqref{entropy_minimisation_problem} described in Theorem \ref{minimisers} exhibits generalisation error with upper bound
    \[ \mathrm{gen}(\mathfrak{m}^{\beta,\sigma}(\nu_n),\nu_{\mathrm{pop}}) \leq \frac{c}{n}\frac{2\beta^2}{\sigma^2} \mathbb{E}_{Z^{(1)}\sim\nu_{\mathrm{pop}}}\Big[(1+\norm{Z^{(1)}})^A\Big]\]
    for some $A \leq 4T + 14$.
    In particular, when $\nu\in\mathcal{P}_{q}(\mathcal{Z}^{T})$ for $q \geq A, p \geq 8$, then the generalisation error for the Gibbs vector is of the scale $n^{-1}$.  
\end{theorem}
\begin{proof}
    Recalling the form of the generalisation error found in Theorem \ref{generalisation error}, we begin by finding bounds on the linear functional derivatives of the running costs $c_t^*$, and the functions $f_t^\lambda, g_t$, for all $t$. 
    \par
    As a pair of prerequisite results we begin by bounding $f_t^\lambda$ and $g_t$. Using Lemmas \ref{tilde gamma lemma}, \ref{norm x}, \ref{Q deriv x}, for general $\nu\in\mathcal{P}_q(\mathcal{Z}^T)$, we find that 
    \begin{align}
        f_{t}^\lambda(\mathfrak{m}_{t-1:T-1}(\nu)) & = \frac{\delta \widehat{Q}_{t}}{\partial x}\Big(X_{t}^{\mathrm{ref}}(\widetilde{Z}^{(1)}) + \lambda\big(X_{t}^{t-1,\mathfrak{m}_{t-1}(\nu)}(\widetilde{Z}^{(1)}) + X_{t}^{\mathrm{ref}}(\widetilde{Z}^{(1)})\big), \mathfrak{m}_{t:T-1}(\nu)\Big) \nonumber\\
        & \leq C(1+\lambda\norm{X_{t}^{t-1,\mathfrak{m}_{t-1}(\nu)}(\widetilde{Z}^{(1)})}+(1-\lambda)\norm{X_{t}^{\mathrm{ref}}(\widetilde{Z}^{(1)})})\nonumber\\
        & \quad \times \prod_{s=t}^{T-1}(1+E_{\mathfrak{m}_{s}(\nu)}^{(4)})(1+\norm{\widetilde{Z}^{(1)}_{s+1}}) \nonumber\\
        & \leq C(1+\norm{x_0})(1+E_{\mathfrak{m}_{t-1}(\nu)}^{(2)})\prod_{s=0}^{T-1}(1+\norm{\widetilde{Z}^{(1)}_{s+1}})\prod_{s=t}^{T-1}(1+E_{\mathfrak{m}_{s}(\nu)}^{(4)}) \nonumber \\
        & \leq C P(\widetilde{Z}^{(1)}) \mathbb{E}_{Z\sim\nu}[P(Z)^{2(T-t+1)}],\label{f bound}
    \end{align}
    and more simply
    \begin{align}
        g_{t}(\mathfrak{m}_{t-1}(\nu)) & \leq \norm{X_{t}^{t-1,\mathfrak{m}_{t-1}(\nu)}(\widetilde{Z}^{(1)})}+\norm{X_{t}^{\mathrm{ref}}(\widetilde{Z}^{(1)})} \nonumber \\
        & \leq C(1+\norm{x_0})(1+E_{\mathfrak{m}_{t-1}(\nu)}^{(2)})\prod_{s=0}^{t-1}(1+\norm{\widetilde{Z}^{(1)}_{s+1}}) \nonumber\\
        & \leq C P(\widetilde{Z}^{(1)})\mathbb{E}_{Z\sim\nu}[P(Z)^2],\label{g bound}
    \end{align}
    where we absorb moments of $\tilde\gamma_p^\sigma$ and powers of $\norm{x_0}$ into $C$ as in Lemma \ref{tilde gamma lemma}.
    \par
    Returning to bounding linear functionals, we start with the running costs $c_t^*$, using Assumption \ref{verifiable assumptions}(vi) and Lemma \ref{norm x}, 

    \begin{align}
        \frac{\delta c_t^*}{\delta m_t}\Big(X_t^{\mathrm{ref}}(\widetilde{Z}^{(1)}), \mathfrak{m}_t^{\tilde{\lambda}}; \theta\Big) & = \partial_u c_t^* (X_t^{\mathrm{ref}}(\widetilde{Z}^{(1)}), \mathfrak{m}_t^{\tilde\lambda})\big(\phi(X_t^{\mathrm{ref}}(\widetilde{Z}^{(1)}), \theta) - \mathbb{E}_{\theta\sim\mathfrak{m}_t^{\tilde\lambda}}[\phi(X_t^{\mathrm{ref}}(\widetilde{Z}^{(1)})]\big) \nonumber \\
        & \leq C(1+\norm{X_t^{\mathrm{ref}}(\widetilde{Z}^{(1)})}^2)(1+\norm{\theta}^2+E_{\mathfrak{m}_t^{\tilde\lambda}}^{(2)})(1+E_{\mathfrak{m}_t^{\tilde{\lambda}}}^{(2)}) \nonumber \\
        & \leq C(1+\norm{\theta}^2+E_{\mathfrak{m}_t^{\tilde\lambda}}^{(2)})(1+E_{\mathfrak{m}_t^{\tilde\lambda}}^{(2)})\prod_{s=0}^{t-1}(1+\norm{\widetilde{Z}^{(1)}_{s+1}})^2 \nonumber \\
        & \leq C(1+\norm{\theta}^2+E_{\mathfrak{m}_t^{\tilde\lambda}}^{(2)})(1+E_{\mathfrak{m}_t^{\tilde\lambda}}^{(2)})P(\widetilde{Z}^{(1)})^2. \label{first}
    \end{align}
    Recalling that $\mathfrak{m}_t^{\tilde\lambda} := (1-\tilde\lambda)\mathfrak{m}_t(\nu_{n,(1)}) + \tilde\lambda \mathfrak{m}_t(\nu_n),$
    using Lemma \ref{tilde gamma lemma} we write
    \begin{align}
        E_{\mathfrak{m}_t^{\tilde\lambda}}^{(2)} & = (1-\tilde\lambda) E_{\mathfrak{m}_t(\nu_{n,(1)})}^{(2)} + \tilde\lambda E_{\mathfrak{m}_t(\nu_n)}^{(2)}\nonumber \\
        & \leq C(1-\tilde\lambda)\mathbb{E}_{Z\sim\nu_{n,(1)}}[P(Z)^2] + C\tilde\lambda \mathbb{E}_{Z\sim\nu_n}[P(Z)^2]\nonumber \\
        & = C\Big(\mathbb{E}_{Z\sim\nu_n}[P(Z)^2] + \frac{1-\tilde\lambda}{n}\big(P(\widetilde{Z}^{(1)})^2 - P(Z^{(1)})^2\big)\Big) \nonumber \\
        & \leq C\Big(\frac{1}{n} \sum_{i=1}^{n} P(Z^{(i)})^2 + P(\widetilde{Z}^{(1)})^2 + P(Z^{(1)})^2\Big). \label{perturbed moment bound}
    \end{align}
    Note then from Lemma \ref{cov} that
    \begin{align*}
        & \int_\Theta \sum_{t=0}^{T-1} \Bigg\{\frac{\delta c_t^*}{\delta m_t}\Big(X_t^{\mathrm{ref}}(\widetilde{Z}^{(1)}), \mathfrak{m}_t^{\tilde{\lambda}};\theta\Big)\frac{\delta \mathfrak{m}_t}{\delta \nu}\Big(\nu_n^{\tilde{\lambda}_2}; Z\Big)\Bigg\}\Bigg|^{Z=Z^{(1)}}_{Z=\widetilde{Z}^{(1)}}(\mathrm{d}\theta) \\
        & \quad = - \frac{2\beta^2}{\sigma^2}\sum_{t=0}^{T-1}\mathbb{C}\mathrm{ov}_{\theta\sim\mathfrak{m}_t(\nu_n^{\tilde\lambda_2})}\Bigg[\frac{\delta c_t^*}{\delta m_t}\Big(X_t^{\mathrm{ref}}(\widetilde{Z}^{(1)}), \mathfrak{m}_t^{\tilde{\lambda}};\theta\Big), \frac{\delta S_t}{\delta \nu}\Big(\nu_n^{\tilde\lambda_2}, \theta; Z\Big)\Big|^{Z=Z^{(1)}}_{Z=\widetilde{Z}^{(1)}}\Bigg] \\
        & \quad \leq \frac{2\beta^2}{\sigma^2}\sum_{t=0}^{T-1}\mathbb{E}_{\theta\sim \mathfrak{m}_t(\nu_n^{\tilde\lambda_2})}\Bigg[\Bigg(\frac{\delta c_t^*}{\delta m_t}\Big(X_t^{\mathrm{ref}}(\widetilde{Z}^{(1)}), \mathfrak{m}_t^{\tilde{\lambda}};\theta\Big)\Bigg)^2\Bigg]^{\frac{1}{2}}\mathbb{E}_{\theta\sim\mathfrak{m}_t(\nu_n^{\tilde\lambda_2})}\Bigg[\Bigg(\frac{\delta S_t}{\delta \nu}\Big(\nu_n^{\tilde\lambda_2}, \theta; Z\Big)\Bigg|^{Z=Z^{(1)}}_{Z=\widetilde{Z}^{(1)}}\Bigg)^2\Bigg]^{\frac{1}{2}},
    \end{align*}
    so we need to bound 
    \begin{align*}
    & \mathbb{E}_{\theta\sim \mathfrak{m}_t(\nu_n^{\tilde\lambda_2})}\Bigg[\Bigg(\frac{\delta c_t^*}{\delta m_t}\Big(X_t^{\mathrm{ref}}(\widetilde{Z}^{(1)}), \mathfrak{m}_t^{\tilde{\lambda}};\theta\Big)\Bigg)^2\Bigg]^{\frac{1}{2}} \\
    & \quad \leq C\mathbb{E}_{\theta\sim \mathfrak{m}_t(\nu_n^{\tilde\lambda_2})}\big[(1+\norm{\theta}^2+E_{\mathfrak{m}_t^{\tilde\lambda}}^{(2)})^2\big]^{\frac{1}{2}}(1+E_{\mathfrak{m}_t^{\tilde\lambda}}^{(2)})P(\widetilde{Z}^{(1)})^2 \\
    & \quad \leq C(1+E_{\mathfrak{m}_t(\nu_n^{\tilde\lambda_2})}^{(4)}+E_{\mathfrak{m}_t^{\tilde\lambda}}^{(4)})(1+E_{\mathfrak{m}_t^{\tilde\lambda}}^{(2)})P(\widetilde{Z}^{(1)})^2 \\
    & \quad \leq CP(\widetilde{Z}^{(1)})^2\Big(\frac{1}{n} \sum_{i=1}^{n} P(Z^{(i)})^2 + P(\widetilde{Z}^{(1)})^2 + P(Z^{(1)})^2\Big)^2,
    \end{align*}
    where we simplified $E_{\mathfrak{m}_t(\nu_n^{\tilde\lambda_2})}^{(4)}$ using the same procedure as for \eqref{perturbed moment bound}. 
    \par
    Finally, applying Lemma \ref{S_t expectation inequality} and \eqref{perturbed moment bound} we bound
    \begin{align}
        & \mathbb{E}_{\theta\sim \mathfrak{m}_t(\nu_n^{\tilde\lambda_2})}\Bigg[\Bigg(\frac{\delta c_t^*}{\delta m_t}\Big(X_t^{\mathrm{ref}}(\widetilde{Z}^{(1)}), \mathfrak{m}_t^{\tilde{\lambda}};\theta\Big)\Bigg)^2\Bigg]^{\frac{1}{2}}\mathbb{E}_{\theta\sim\mathfrak{m}_t(\nu_n^{\tilde\lambda_2})}\Bigg[\Bigg(\frac{\delta S_t}{\delta \nu}\Big(\nu_n^{\tilde\lambda_2}, \theta; Z\Big)\Bigg|^{Z=Z^{(1)}}_{Z=\widetilde{Z}^{(1)}}\Bigg)^2\Bigg]^{\frac{1}{2}} \nonumber \\
        & \leq C P(\widetilde{Z}^{(1)})^2 \Big( P(Z^{(1)})^2 + P(\widetilde{Z}^{(1)})^2 \Big)\Big(\frac{1}{n} \sum_{i=1}^{n} P(Z^{(i)})^2 + P(\widetilde{Z}^{(1)})^2 + P(Z^{(1)})^2\Big)^2 \nonumber \\& \quad \times (1+E_{\mathfrak{m}_t(\nu_n^{\tilde\lambda_2})}^{(2)})\prod_{s=t}^{T-1}(1+E_{\mathfrak{m}_s(\nu_n^{\tilde\lambda_2})}^{(4)}) \nonumber \\
        & \leq C P(\widetilde{Z}^{(1)})^2 \Big(P(Z^{(1)})^2 + P(\widetilde{Z}^{(1)})^2 \Big)\Big(\frac{1}{n} \sum_{i=1}^{n} P(Z^{(i)})^2 + P(\widetilde{Z}^{(1)})^2 + P(Z^{(1)})^2\Big)^{T-t+3}. \nonumber
    \end{align}
    Over the full sum, we bound uniformly over $t$ to find that 
    \begin{align}
        & \sum_{t=0}^{T-1}\mathbb{E}_{\theta\sim \mathfrak{m}_t(\nu_n^{\tilde\lambda_2})}\Bigg[\Bigg(\frac{\delta c_t^*}{\delta m_t}\Big(X_t^{\mathrm{ref}}(\widetilde{Z}^{(1)}), \mathfrak{m}_t^{\tilde{\lambda}};\theta\Big)\Bigg)^2\Bigg]^{\frac{1}{2}}\mathbb{E}_{\theta\sim\mathfrak{m}_t(\nu_n^{\tilde\lambda_2})}\Bigg[\Bigg(\frac{\delta S_t}{\delta \nu}\Big(\nu_n^{\tilde\lambda_2}, \theta; Z\Big)\Bigg|^{Z=Z^{(1)}}_{Z=\widetilde{Z}^{(1)}}\Bigg)^2\Bigg]^{\frac{1}{2}} \nonumber \\
        & \leq C P(\widetilde{Z}^{(1)})^2 \Big(P(Z^{(1)})^2 + P(\widetilde{Z}^{(1)})^2 \Big)\Big(\frac{1}{n} \sum_{i=1}^{n} P(Z^{(i)})^2 + P(\widetilde{Z}^{(1)})^2 + P(Z^{(1)})^2\Big)^{T+3}. \label{first final}
    \end{align}
    For the remaining linear functional derivatives we aim to bound in a similar fashion, but will simplify notation for clarity. Defining $\widetilde{X}_t^{\lambda, m_{t-1}}(Z):= X_t^{\mathrm{ref}}(Z) + \lambda\big(X_t^{t-1, m_{t-1}}(Z) - X_t^{\mathrm{ref}}(Z)\big)$ we begin with 
    \begin{align*}
        & \frac{\delta f_t^\lambda}{\delta m_{t-1}}\big(m_{t-1:T-1}; \theta)  \\ & = \frac{\delta}{\delta m_{t-1}}\frac{\partial \widehat{Q}_t}{\partial x}\big(\widetilde{X}_t^{\lambda, m_{t-1}}(\widetilde{Z}^{(1)}), m_t, \widetilde{Z}^{(1)}; \theta\big) \\ 
        & = \frac{\partial^2 \widehat{Q}_t}{\partial x^2}\big(\widetilde{X}_t^{\lambda, m_{t-1}}(\widetilde{Z}^{(1)}), m_t, \widetilde{Z}^{(1)}\big) \frac{\delta }{\delta m_{t-1}}\widetilde{X}_t^{\lambda, m_{t-1}}(\widetilde{Z}^{(1)}) \\
        & = \lambda\frac{\partial^2 \widehat{Q}_t}{\partial x^2}\big(\widetilde{X}_t^{\lambda, m_{t-1}}(\widetilde{Z}^{(1)}), m_t, \widetilde{Z}^{(1)}\big)\partial_u h_{t-1}(X_{t-1}^{\mathrm{ref}}(\widetilde{Z}^{(1)}), u_{m_{t-1}}(X_{t-1}^{\mathrm{ref}}(\widetilde{Z}^{(1)})), \widetilde{Z}^{(1)}_{t}) \\
        & \quad \times \big(\phi(X_{t-1}^{\mathrm{ref}}(\widetilde{Z}^{(1)}), \theta) - \mathbb{E}_{\theta\sim m_{t-1}}[\phi(X_{t-1}^{\mathrm{ref}}(\widetilde{Z}^{(1)}), \theta)]\big) \\
        & \leq C\lambda(1+\norm{\widetilde{X}_t^{\lambda, m_{t-1}}(\widetilde{Z}^{(1)})})(1+\norm{X_{t-1}^{\mathrm{ref}}(\widetilde{Z}^{(1)})})(1+\norm{\theta}^2 + E_{m_{t-1}}^{(2)})\prod_{s=t}^{T-1}(1+E_{m_s}^{(8)})(1+\norm{\widetilde{Z}^{(1)}_{s+1}}) \\
        & \leq C\lambda (1+E_{m_{t-1}}^{(2)})(1+\norm{\theta}^2 + E_{m_{t-1}}^{(2)})P(\widetilde{Z}^{(1)})^2\prod_{s=t}^{T-1}(1+E_{m_s}^{(8)}),
    \end{align*}
    where we used Assumption \ref{verifiable assumptions} and Lemmas \ref{norm x}, \ref{Q double deriv x} to simplify. 
    \par
    Multiplying by the bound for $g_t$ found in \eqref{g bound}, we find that 
    \begin{align} & 
    \frac{\delta f_t^{\lambda}}{\delta m_{t-1}}\big(\mathfrak{m}_{t-1}^{\tilde\lambda}, \mathfrak{m}_{t:T-1}(\nu_n); \theta)g_t(\mathfrak{m}_{t-1}(\nu_n)) \nonumber \\
    & \quad \leq C(1+E_{\mathfrak{m}_{t-1}^{\tilde\lambda}}^{(2)})(1+\norm{\theta}^2 + E_{\mathfrak{m}_{t-1}^{\tilde\lambda}}^{(2)})P(\widetilde{Z}^{(1)})^3\mathbb{E}_{Z\sim\nu_n}[P(Z)^2] \prod_{s=t}^{T-1}(1+E_{\mathfrak{m}_s(\nu_n)}^{(8)}). \nonumber
    \end{align}
    Again anticipating the use of Lemma \ref{cov}, we apply Lemma \ref{tilde gamma lemma} and \eqref{perturbed moment bound} to find 
    \begin{align}
        & \mathbb{E}_{\theta\sim \mathfrak{m}_{t-1}(\nu_n^{\tilde\lambda_2})}\Bigg[\Bigg(\frac{\delta f_t^\lambda}{\delta m_{t-1}}\big(\mathfrak{m}_{t-1}^{\tilde{\lambda}}, \mathfrak{m}_{t:T-1}(\nu_n);\theta\big)g_t(\mathfrak{m}_{t-1}(\nu_n))\Bigg)^2\Bigg]^{\frac{1}{2}} \nonumber \\
        & \leq CP(\widetilde{Z}^{(1)})^3\Big(\frac{1}{n}\sum_{i=1}^{n}P(Z^{(i)})^2\Big)(1+E_{\mathfrak{m}_{t-1}^{\tilde\lambda}}^{(2)})(1+E_{\mathfrak{m}_{t-1}(\nu_n^{\tilde\lambda_2})}^{(4)}+E_{\mathfrak{m}_{t-1}^{\tilde\lambda}}^{(4)})\prod_{s=t}^{T-1}(1+E_{\mathfrak{m}_s(\nu_n)}^{(8)}) \nonumber \\
        & \leq CP(\widetilde{Z}^{(1)})^3\Big(\frac{1}{n}\sum_{i=1}^{n}P(Z^{(i)})^2\Big)^{T-t+1}\Big(\frac{1}{n} \sum_{i=1}^{n} P(Z^{(i)})^2 + P(\widetilde{Z}^{(1)})^2 + P(Z^{(1)})^2\Big)^2. \nonumber
    \end{align}
    Bounding uniformly over all $t$ and applying Lemmas \ref{cov}, \ref{S_t expectation inequality}, we find
    \begin{align}
        & \int_\Theta\sum_{t=1}^{T-1}\Bigg\{\frac{\delta f_t^\lambda}{\delta m_{t-1}}\big(\mathfrak{m}_{t-1}^{\tilde{\lambda}}, \mathfrak{m}_{t:T-1}(\nu_n);\theta\big)g_t(\mathfrak{m}_{t-1}(\nu_n))\frac{\delta \mathfrak{m}_{t-1}}{\delta \nu}\Big(\nu_n^{\tilde{\lambda}_2}; Z\Big)\Bigg\}\Bigg|^{Z=Z^{(1)}}_{Z=\widetilde{Z}^{(1)}}(\mathrm{d}\theta) \nonumber \\
        & \quad \leq \frac{2\beta^2}{\sigma^2}\sum_{t=1}^{T-1}\mathbb{E}_{\theta\sim \mathfrak{m}_{t-1}(\nu_n^{\tilde\lambda_2})}\Bigg[\Bigg(\frac{\delta f_t^\lambda}{\delta m_{t-1}}\big(\mathfrak{m}_{t-1}^{\tilde{\lambda}}, \mathfrak{m}_{t:T-1}(\nu_n);\theta\big)g_t(\mathfrak{m}_{t-1}(\nu_n))\Bigg)^2\Bigg]^{\frac{1}{2}} \nonumber \\
        & \quad \quad \times \mathbb{E}_{\theta\sim\mathfrak{m}_{t-1}(\nu_n^{\tilde\lambda_2})}\Bigg[\Bigg(\frac{\delta S_{t-1}}{\delta \nu}\Big(\nu_n^{\tilde\lambda_2}, \theta; Z\Big)\Bigg|_{Z=\widetilde{Z}^{(1)}}^{Z=Z^{(1)}}\Bigg)^2\Bigg]^{\frac{1}{2}} \nonumber\\
        & \quad \leq \frac{2\beta^2}{\sigma^2}C P(\widetilde{Z}^{(1)})^3 \Big(P(Z^{(1)})^2 + P(\widetilde{Z}^{(1)})^2\Big)\Big(\frac{1}{n}\sum_{i=1}^{n}P(Z^{(i)})^2\Big)^{T} \nonumber\\
        & \quad \quad \times \Big(\frac{1}{n} \sum_{i=1}^{n} P(Z^{(i)})^2 + P(\widetilde{Z}^{(1)})^2 + P(Z^{(1)})^2\Big)^{2}\sum_{t=1}^{T-1}(1+E_{\mathfrak{m}_{t-1}(\nu_n^{\tilde\lambda_2})}^{(2)})\prod_{s=t-1}^{T-1}(1+E_{\mathfrak{m}_s(\nu_n^{\tilde\lambda_2})}^{(4)}) \nonumber\\
        & \quad \leq \frac{2\beta^2}{\sigma^2}C P(\widetilde{Z}^{(1)})^3 \Big(P(Z^{(1)})^2 + P(\widetilde{Z}^{(1)})^2\Big)\Big(\frac{1}{n}\sum_{i=1}^{n}P(Z^{(i)})^2\Big)^T \nonumber \\
        & \quad \quad \times \Big(\frac{1}{n} \sum_{i=1}^{n} P(Z^{(i)})^2 + P(\widetilde{Z}^{(1)})^2 + P(Z^{(1)})^2\Big)^{T+3}. \label{second final}
    \end{align}
    Moving on, for $s > t-1$, we find from Lemmas \ref{norm x}, \ref{Q frechet deriv},
    \begin{align*}
        & \frac{\delta f_t^\lambda}{\delta m_s}\big(\mathfrak{m}_{t-1:s-1}(\nu_{n,(1)}), \mathfrak{m}_s^{\tilde\lambda}, \mathfrak{m}_{s+1:T-1}(\nu_n); \theta\big)\\
        & \quad = \frac{\delta}{\delta m_s}\frac{\partial \widehat{Q}_t}{\partial x}\big(\widetilde{X}_t^{\lambda, \mathfrak{m}_{t-1}(\nu_{n,(1)})}(\widetilde{Z}^{(1)}), \mathfrak{m}_{t:s-1}(\nu_{n,(1)}), \mathfrak{m}_s^{\tilde\lambda}, \mathfrak{m}_{s+1:T-1}(\nu_n), \widetilde{Z}^{(1)};\theta\big) \\
        & \quad \leq C(1+\norm{\widetilde{X}_t^{\lambda, \mathfrak{m}_{t-1}(\nu_{n,(1)})}(\widetilde{Z}^{(1)})}^2)(1+\norm{\theta}^2 + E_{\mathfrak{m}_s^{\tilde\lambda}}^{(2)})(1+E_{\mathfrak{m}_s^{\tilde\lambda}}^{(4)})\prod_{l=t}^{T-1}(1+\norm{\widetilde{Z}^{(1)}}^2) \\
        & \quad \quad \times \prod_{l=t}^{s-1}(1+E_{\mathfrak{m}_l(\nu_{n,(1)})}^{(8)})\prod_{l=s+1}^{T-1}(1+E_{\mathfrak{m}_l(\nu_n)}^{(8)}) \\
        & \quad \leq C(1+E_{\mathfrak{m}_{t-1}(\nu_{n,(1)})}^{(4)})(1+\norm{\theta}^2+E_{\mathfrak{m}_s^{\tilde\lambda}}^{(2)})(1+E_{\mathfrak{m}_s^{\tilde\lambda}}^{(4)}) P(\widetilde{Z}^{(1)})^2 \\
        & \quad \times \prod_{l=t}^{s-1}(1+E_{\mathfrak{m}_l(\nu_{n,(1)})}^{(8)})\prod_{l=s+1}^{T-1}(1+E_{\mathfrak{m}_l(\nu_n)}^{(8)}),
    \end{align*}
    so that 
    \begin{align}
        & \frac{\delta f_t^\lambda}{\delta m_s}\big(\mathfrak{m}_{t-1:s-1}(\nu_{n,(1)}), \mathfrak{m}_s^{\tilde\lambda}, \mathfrak{m}_{s+1:T-1}(\nu_n); \theta\big) g_{t}(\mathfrak{m}_{t-1}(\nu_n)) \nonumber\\
        & \quad \leq CP(\widetilde{Z}^{(1)})^3\Big(\frac{1}{n}\sum_{i=1}^{n}P(Z^{(i)})^2\Big)(1+E_{\mathfrak{m}_{t-1}(\nu_{n,(1)})}^{(4)})(1+\norm{\theta}^2+E_{\mathfrak{m}_s^{\tilde\lambda}}^{(2)})(1+E_{\mathfrak{m}_s^{\tilde\lambda}}^{(4)}) \nonumber\\
        & \quad \quad \times \prod_{l=t}^{s-1}(1+E_{\mathfrak{m}_l(\nu_{n,(1)})}^{(8)})\prod_{l=s+1}^{T-1}(1+E_{\mathfrak{m}_l(\nu_n)}^{(8)}).\nonumber
    \end{align}
    Anticipating Lemma \ref{cov}, we apply Lemma \ref{tilde gamma lemma} and \eqref{perturbed moment bound} to bound
    \begin{align}
        & \mathbb{E}_{\theta\sim \mathfrak{m}_{s}(\nu_n^{\tilde\lambda_2})}\Bigg[\Bigg(\frac{\delta f_t^\lambda}{\delta m_{s}}\big(\mathfrak{m}_{t-1:s-1}(\nu_{n,(1)}), \mathfrak{m}_s^{\tilde\lambda}, \mathfrak{m}_{s+1:T-1}(\nu_n);\theta\big)g_t(\mathfrak{m}_{t-1}(\nu_n))\Bigg)^2\Bigg]^{\frac{1}{2}} \nonumber \\
        & \leq CP(\widetilde{Z}^{(1)})^3\Big(\frac{1}{n}\sum_{i=1}^{n}P(Z^{(i)})^2\Big)(1+E_{\mathfrak{m}_{t-1}(\nu_{n,(1)})}^{(4)})(1+E_{\mathfrak{m}_s(\nu_n^{\tilde\lambda_2})}^{(4)}+E_{\mathfrak{m}_s^{\tilde\lambda}}^{(4)})(1+E_{\mathfrak{m}_s^{\tilde\lambda}}^{(4)}) \nonumber\\
        & \quad \quad \times \prod_{l=t}^{s-1}(1+E_{\mathfrak{m}_l(\nu_{n,(1)})}^{(8)})\prod_{l=s+1}^{T-1}(1+E_{\mathfrak{m}_l(\nu_n)}^{(8)}) \nonumber \\
        & \leq CP(\widetilde{Z}^{(1)})^3 \Big(\frac{1}{n}\sum_{i=1}^{n}P(Z^{(i)})^2\Big)\Big(\frac{1}{n} \sum_{i=1}^{n} P(Z^{(i)})^2 + P(\widetilde{Z}^{(1)})^2 + P(Z^{(1)})^2\Big)^{T-t+2}. \nonumber
    \end{align}
    Hence applying Lemmas \ref{cov}, \ref{S_t expectation inequality} and bounding, we find that 
    \begin{align}
        & \int_{\Theta}\sum_{t=1}^{T-1}\Bigg\{ \frac{\delta f_t^\lambda}{\delta m_s}\big(\mathfrak{m}_{t-1:s-1}(\nu_{n,(1)}), \mathfrak{m}_s^{\tilde\lambda}, \mathfrak{m}_{s+1:T-1}(\nu_n); \theta\big)g_t(\mathfrak{m}_{t-1}(\nu_n))\frac{\delta \mathfrak{m}_s}{\delta \nu}\Big(\nu_n^{\tilde\lambda_2}; Z\Big)\Bigg\}\Bigg|_{Z=\widetilde{Z}^{(1)}}^{Z=Z^{(1)}}(\mathrm{d}\theta) \nonumber \\
        & \leq \frac{2\beta^2}{\sigma^2} CP(\widetilde{Z}^{(1)})^3\Big(P(Z^{(1)})^2 + P(\widetilde{Z}^{(1)})^2\Big)\Big(\frac{1}{n}\sum_{i=1}^{n}P(Z^{(i)})^2\Big)\nonumber\\& \quad \times \Big(\frac{1}{n} \sum_{i=1}^{n} P(Z^{(i)})^2 + P(\widetilde{Z}^{(1)})^2 + P(Z^{(1)})^2\Big)^{T+1}(1+E_{\mathfrak{m}_s(\nu_n^{\tilde\lambda_2})}^{(2)})\prod_{l=s}^{T-1}(1+E_{\mathfrak{m}_l(\nu_n^{\tilde\lambda_2})}^{(4)}) \nonumber \\
        & \leq \frac{2\beta^2}{\sigma^2} CP(\widetilde{Z}^{(1)})^3\Big(P(Z^{(1)})^2 + P(\widetilde{Z}^{(1)})^2\Big)\Big(\frac{1}{n}\sum_{i=1}^{n}P(Z^{(i)})^2\Big) \nonumber \\
        & \quad \times \Big(\frac{1}{n} \sum_{i=1}^{n} P(Z^{(i)})^2 + P(\widetilde{Z}^{(1)})^2 + P(Z^{(1)})^2\Big)^{2T+1}. 
 \label{final third}
    \end{align}
    For the final term from Theorem \ref{generalisation error}, we first bound the linear functional derivative of $g_{t}$ using Assumption \ref{verifiable assumptions}(v,vi) and Lemma \ref{norm x},
    \begin{align*}
        & \frac{\delta g_{t}}{\delta m_{t-1}}(\mathfrak{m}_{t-1}^{\tilde{\lambda}};\theta) \\ & = \partial_u h_{t-1}(X_{t-1}^{\mathrm{ref}}(\widetilde{Z}^{(1)}), u_{\mathfrak{m}_{t-1}^{\tilde\lambda}}(X_{t-1}^{\mathrm{ref}}(\widetilde{Z}^{(1)})), \widetilde{Z}^{(1)}_t)\big(\phi(X_{t-1}^{\mathrm{ref}}(\widetilde{Z}^{(1)}),\theta) - \mathbb{E}_{\theta\sim\mathfrak{m}_{t-1}^{\tilde\lambda}}[\phi(X_{t-1}^{\mathrm{ref}}(\widetilde{Z}^{(1)}),\theta)]\big) \\
        & \leq CP(\widetilde{Z}^{(1)})(1+\norm{\theta}^2+E_{\mathfrak{m}_{t-1}^{\tilde\lambda}}^{(2)}),
    \end{align*}
    so that 
    \begin{align} 
    & f_t^\lambda(\mathfrak{m}_{t-1:T-1}(\nu_{n,(1)}))\frac{\delta g_{t}}{\delta m_{t-1}}(\mathfrak{m}_{t-1}^{\tilde{\lambda}};\theta)\nonumber\\
    & \quad \leq CP(\widetilde{Z}^{(1)})^2 (1+\norm{\theta}^2+E_{\mathfrak{m}_{t-1}^{\tilde\lambda}}^{(2)})(1+E_{\mathfrak{m}_{t-1}(\nu_{n,(1)})}^{(2)})\prod_{l=t}^{T-1}(1+E_{\mathfrak{m}_l(\nu_{n,(1)})}^{(4)}), \nonumber
    \end{align}
    where we bounded $f_t^\lambda$ using \eqref{f bound}.
    This then allows us to bound 
    \begin{align*}
        & \mathbb{E}_{\theta\sim \mathfrak{m}_{t-1}(\nu_n^{\tilde\lambda_2})}\Bigg[\Bigg(f_t^\lambda\big(\mathfrak{m}_{t-1:T-1}(\nu_{n,(1)})\big)\frac{\delta g_t}{\delta m_{t-1}}(\mathfrak{m}_{t-1}^{\tilde\lambda}; \theta)\Bigg)^2\Bigg]^{\frac{1}{2}} \\
        & \leq CP(\widetilde{Z}^{(1)})^2 (1+E_{\mathfrak{m}_{t-1}(\nu_n^{\tilde\lambda_2})}^{(4)}+E_{\mathfrak{m}_{t-1}^{\tilde\lambda}}^{(4)})(1+E_{\mathfrak{m}_{t-1}(\nu_{n,(1)})}^{(2)})\prod_{l=t}^{T-1}(1+E_{\mathfrak{m}_l(\nu_{n,(1)})}^{(4)}) \\
        & \leq CP(\widetilde{Z}^{(1)})^2 \Big(\frac{1}{n} \sum_{i=1}^{n} P(Z^{(i)})^2 + P(\widetilde{Z}^{(1)})^2 + P(Z^{(1)})^2\Big)^{T-t+2}.
    \end{align*}
    Applying Lemmas \ref{cov}, \ref{S_t expectation inequality}, we find that
    \begin{align}
        & \int_{\Theta}\sum_{t=1}^{T-1} \Big\{f_t^\lambda\big(\mathfrak{m}_{t-1:T-1}(\nu_{n,(1)})\big)\frac{\delta g_t}{\delta m_{t-1}}(\mathfrak{m}_{t-1}^{\tilde\lambda};\theta)\frac{\delta \mathfrak{m}_{t-1}}{\delta \nu}(\nu_n^{\tilde\lambda_2}; Z)\Big\}(\mathrm{d}\theta) \nonumber \\
        & \leq C P(\widetilde{Z}^{(1)})^2 \Big(P(Z^{(1)})^2 + P(\widetilde{Z}^{(1)})^2\Big) \nonumber \\
        & \quad \times \Big(\frac{1}{n} \sum_{i=1}^{n} P(Z^{(i)})^2 + P(\widetilde{Z}^{(1)})^2 + P(Z^{(1)})^2\Big)^{T+1}  \sum_{t=1}^{T-1}(1+E_{\mathfrak{m}_{t-1}(\nu_n^{\tilde\lambda_2})}^{(2)})\prod_{s=t-1}^{T-1}(1+E_{\mathfrak{m}_s(\nu_n^{\tilde\lambda_2})}^{(4)}) \nonumber \\
        & \leq C P(\widetilde{Z}^{(1)})^2 \Big(P(Z^{(1)})^2 + P(\widetilde{Z}^{(1)})^2\Big)\Big(\frac{1}{n} \sum_{i=1}^{n} P(Z^{(i)})^2 + P(\widetilde{Z}^{(1)})^2 + P(Z^{(1)})^2\Big)^{2T+2}. \label{fourth final}
    \end{align}
    Substituting \eqref{first final}, \eqref{second final}, \eqref{final third} and \eqref{fourth final} into the original form of the generalisation error from Theorem \ref{generalisation error}, applying \cite[Lemma D.7]{sam_ge} to simplify the final expectation, and finally using H\"{o}lder's inequality, we recover the claim. In particular, we note that the highest order moment we require of any element of the Gibbs vector $\mathfrak{m}(\nu_n)$ is $8$, hence $p \geq 8$ and $q \geq 4T + 14$ guarantee finiteness of the generalisation error upper bound, and hence the asymptotic convergence rate of $n^{-1}$.
\end{proof}
\begin{remark}
    We make some important comments on Theorem \ref{1/n theoretical}:
    \begin{itemize}
        \item 
        Note that the bound is not sharp -- in this paper we merely aim to illustrate the effects of regularisation, and the very existence of a $n^{-1}$ upper bound on the generalisation error, which itself is sharp (as this is the rate for supervised learning in one dimension).
        \item 
        The high order finite polynomial moments required of the stochastic environment are
        satisfied if the stochastic environment has finite exponential moments. This is not an unreasonable modelling assumption. 
        \item 
        The high value of $p \geq 8$ indicates that $\Gamma$ must regularise very sharply as $\norm{\theta}$ gets very large. This is due to the very weak assumptions we have made on the covariance structure of $Z$, together with the potential interactions of the controls at different times, and can be seen as a worst-case requirement.  This assumption is satisfied by 
        \[ \Gamma(\theta) := \norm{\theta}^2 + \epsilon \exp(\norm{\theta}),\]
        where $\epsilon > 0$ is very small. This example simultaneously addresses the required growth conditions on $\Gamma$ and ensures that the gradients of $\Gamma$ are unlikely to explode, which is helpful for computational purposes. In addition, this ensures regularisation remains close to quadratic regularisation, for which computationally sampling from the Gibbs vector $\mathfrak{m}(\nu_n)$ is well-understood (\cite{full_non_asymptotic}). 
    \end{itemize} 
    \par
    
\end{remark}
\subsection{Balancing Bias and Stability}\label{balancing bias + stability}
The above computations focus on the generalisation error, which explicitly ignores the bias due to regularisation of our learning problem -- it would be useful to understand how to balance the two. 
\par 
Suppose that, at each step $t$ of the minimisation procedure \eqref{entropy_minimisation_problem}, there exists some $m_t^*$ minimising the empirical risk of the (unregularised) approximate dynamic programming problem 
\[ m \mapsto \frac{1}{n}\sum_{i=1}^{n} \widehat{Q}_t(X_t^{\mathrm{ref}}(Z^{(i)}), m, m_{t+1}^*, \ldots, m_{T-1}^*, Z^{(i)}).\]
\par 
Then the measure vector $\mathbf{m}^* := (m_t^*)_{t=0}^{T-1}$ is the solution to the full empirical risk minimisation problem \eqref{ERM problem}, achieving a minimum of 
\[ \mathbb{E}_{Z\sim\nu_n}\big[\ell(\mathbf{m}^*, Z)\big] = \frac{1}{n}\sum_{i=1}^{n} \widehat{Q}_0(x_0 , m_{0:T-1}^*, Z^{(i)}).\]
Recall then that, by the dynamic programming principle, since it solves \eqref{entropy_minimisation_problem}, the Gibbs vector $\mathfrak{m}(\nu_n)$ minimises 
\[ \mathbf{m} = (m_t)_{t=0}^{T-1} \subset \mathcal{P}_2(\Theta) \mapsto \frac{1}{n}\sum_{i=1}^{n} \widehat{Q}_0(x_0, m_{0:T-1}, Z^{(i)}) + \frac{\sigma^2}{2\beta^2}\sum_{t=0}^{T-1}\mathrm{KL}(m_t||\gamma^\sigma). \]
By suboptimality of $\mathbf{m}^*$ to this problem, and nonnegativity of KL-divergence, we see that
\begin{align*}
    \mathbb{E}_{Z\sim\nu_n}[\ell(\mathfrak{m}(\nu_n), Z)] \leq \mathbb{E}_{Z\sim\nu_n}[\ell(\mathbf{m}^*, Z)] + \frac{\sigma^2}{2\beta^2}\sum_{t=0}^{T-1}\mathrm{KL}(m_t^*||\gamma^\sigma).
\end{align*}
Therefore, assuming that $\mathrm{KL}(m_t||\gamma^\sigma) < \infty$ for all $t$, we can bound the expected population risk (a measure of the bias) of the Gibbs vector by writing
\begin{align*}
    \mathbb{E}_{\mathbf{Z}_n}\mathbb{E}_{Z\sim\nu}[\ell(\mathfrak{m}(\nu_n), Z)] & = \mathrm{gen}(\mathfrak{m}(\nu_n), \nu_{\mathrm{pop}}) + \mathbb{E}_{\mathbf{Z}_n}\mathbb{E}_{Z\sim\nu_n}[\ell(\mathfrak{m}(\nu_n), Z)] \\
    & \leq \mathbb{E}_{\mathbf{Z}_n} \mathbb{E}_{Z\sim\nu_n}[\ell(\mathbf{m}^*, Z)]  + \frac{c}{n}\frac{2\beta^2}{\sigma^2}\mathbb{E}_{Z^{(1)}\sim\nu}\big[(1+\norm{Z^{(1)}})^A\big] + \frac{\sigma^2}{2\beta^2}\sum_{t=0}^{T-1}\mathrm{KL}(m_t^*||\gamma^\sigma),
\end{align*}
where we used the result from Theorem \ref{1/n theoretical}.
This demonstrates that scaling $\beta \propto n^{\frac{1}{4}}$ leads to both the generalisation error and expected population risk of the Gibbs vector being of order $n^{-\frac{1}{2}}$ -- a clear balance of the tradeoff between bias and stability. 
\begin{remark}\label{Monte Carlo comparisons}
    We note that, using our notation and bounding the $L_2$ generalisation error \eqref{L2 error} from Remark \ref{L2 generalisation remark}, it is possible to derive such a $n^{-1}$ upper bound and subsequently demonstrate scaling results by computing bounds on the expectation of the squared population risk. That is, our results would match with \cite{sam_ge}, which demonstrates that scaling $\beta \propto n^{\frac{1}{6}}$ gives an upper bound on the expectation of the squared population risk of order $n^{-\frac{2}{3}}$.
    \par
    We have omitted these computations for sake of space, but it is worth noting this possibility, as it guarantees that we may use Markov's inequality to then easily produce probabilistic bounds on the generalisation error. 
\end{remark}
\section{Computational Aspects}\label{computation}
Whilst the generalisation and in-sample properties of the Gibbs vector $\mathfrak{m}(\nu_n)$ are demonstrably desirable, it is not yet clear how one might compute such an object. In this section, we discuss one such approximation procedure, culminating in Algorithm \ref{gibbs vector algo}.
\par 
Recall the $t$-th minimisation problem from \eqref{entropy_minimisation_problem}, 
\begin{align*} \mathrm{minimise} \ m \in \mathcal{P}_p(\Theta) \mapsto \mathbb{E}_{Z\sim\nu_n}\big[\widehat{Q}_t(X_t^{\mathrm{ref}}(Z), m, Z; \mathfrak{m}_{t+1}(\nu_n), \ldots, \mathfrak{m}_{T-1}(\nu_n))\big] + \frac{\sigma^2}{2\beta^2}\mathrm{KL}(m||\gamma^\sigma).
\end{align*}
From Theorem \ref{minimisers}, a unique minimiser $\mathfrak{m}_t(\nu_n)$ exists, and is characterised by the fixed-point equation of its density given by
\begin{align*} \mathfrak{m}_t(\nu_n)(\theta) & = \frac{1}{F_{\beta,\sigma, t}}\exp\left\{-\frac{2\beta^2}{\sigma^2}\left(\mathbb{E}_{Z\sim\nu_n}\Big[\frac{\delta\widehat{Q}_t}{\delta m_t}(X_t^{\mathrm{ref}}(Z), \mathfrak{m}_t(\nu_n), Z;\theta)\Big]+\frac{1}{2\beta^2} \Gamma(\theta)\right)\right\} \\
& = \frac{1}{F_{\beta,\sigma, t}}\exp\left\{-\frac{2\beta^2}{\sigma^2}\left( \frac{\delta R_t}{\delta m_t}(\nu_n, \mathfrak{m}_{t}(\nu_n), \ldots, \mathfrak{m}_{T-1}(\nu_n);\theta)+\frac{1}{2\beta^2} \Gamma(\theta)\right)\right\},
\end{align*}
where we have used the notation $R_t$ as defined in \eqref{risk definition}. The measure $\mathfrak{m}_t(\nu_n)$ is also a stationary solution of the nonlinear Fokker--Planck equation
\[ \partial_\tau m_\tau = \nabla_\theta \Bigg(\Big(D_m R_t(\nu_n, m_\tau, \mathfrak{m}_{t+1}(\nu_n), \ldots, \mathfrak{m}_{T-1}(\nu_n); \theta)+\frac{1}{2\beta^2}\nabla_\theta U(\theta)\Big)m_\tau  + \frac{\sigma^2}{2\beta^2}\nabla_\theta m_\tau\Bigg),\]
with time indexed by $\tau$. In \cite{langevin} it is demonstrated that $m_\tau$ converges in Wasserstein-2 metric to the Gibbs vector element $\mathfrak{m}_t(\nu_n)$. We begin to see potential algorithmic connections once we note that the law of $m_\tau$ is the law of the process $\theta = (\theta_\tau)_\tau$, which is governed by a McKean--Vlasov SDE of the form
\[ \mathrm{d}\theta_\tau = -\Big(D_m R_t(\nu_n, m_\tau, \mathfrak{m}_{t+1}(\nu_n), \ldots, \mathfrak{m}_{T-1}(\nu_n); \theta_\tau)+\frac{1}{2\beta^2}\nabla U(\theta_\tau)\Big)\mathrm{d}\tau + \frac{\sigma}{\beta}\mathrm{d}W_\tau,\]
where $W = (W_\tau)_\tau$ denotes a Brownian motion. This is known as the \textit{mean-field Langevin dynamics} (MFLD). The key issue with such an equation is the explicit dependence on its own law, the very object we wish to approximate. From propagation of chaos, we may approximate $m_\tau$ by $m_\tau^r$ for some integer $r$, which denotes the empirical law of an interacting particle system, given by 
\[ \begin{cases}
    \mathrm{d}\theta_\tau^j = -\Big(D_m R_t(\nu_n, m^r_\tau, \mathfrak{m}_{t+1}(\nu_n), \ldots, \mathfrak{m}_{T-1}(\nu_n); \theta_\tau^j)+\frac{1}{2\beta^2}\nabla U(\theta_\tau^j)\Big)\mathrm{d}\tau + \frac{\sigma}{\beta}\mathrm{d}W_\tau, & j=1,\ldots,r,\\
    m_\tau^r := \frac{1}{r}\sum_{j=1}^{r} \delta_{\theta_\tau^j}.
\end{cases}\]
There has been extensive research demonstrating strong and weak convergence of $m_\tau^r$ to $m_\tau$ as $r\to\infty$ (see \cite{prop_chaos_SGD, topics_chaos}) --- particularly useful are the results where such convergence is uniform in $\tau$ (\cite{uniform-in-time-chaos}). 
\par
Note that, for any $j, \tau$, we may write
\[ \nabla_{\theta^j} R_t(\nu_n, m^r_\tau, \mathfrak{m}_{t+1}(\nu_n), \ldots, \mathfrak{m}_{T-1}(\nu_n)) = \frac{1}{r}D_m R_t(\nu_n, m^r_\tau, \mathfrak{m}_{t+1}(\nu_n), \ldots, \mathfrak{m}_{T-1}(\nu_n); \theta_\tau^j),\]
so the above system becomes 
\[ \begin{cases}
    \mathrm{d}\theta_\tau^j = -\nabla_{\theta^j}\Big(rR_t(\nu_n, m^r_\tau, \mathfrak{m}_{t+1}(\nu_n), \ldots, \mathfrak{m}_{T-1}(\nu_n))+\frac{1}{2\beta^2} U(\theta_\tau^j)\Big)\mathrm{d}\tau + \frac{\sigma}{\beta}\mathrm{d}W_\tau, & j=1,\ldots,r,\\
    m_\tau^r := \frac{1}{r}\sum_{j=1}^{r} \delta_{\theta_\tau^j}.
\end{cases}\]
Upon imposing the final layer of approximation which arises from time-discretisation, we see that this is simply a continuous version of the noisy stochastic gradient descent algorithm, with updates given by
\[ \theta^j_{\tau_{k+1}} = \theta^j_{\tau_k} - \eta\nabla_{\theta^j}\Big(rR_t(\nu_n, m^r_{\tau_k}, \mathfrak{m}_{t+1}(\nu_n), \ldots, \mathfrak{m}_{T-1}(\nu_n))+\frac{1}{2\beta^2} U(\theta_{\tau_k}^j)\Big) + \frac{\sigma}{\beta}\sqrt{\eta}\xi_k^j, \quad j = 1,\ldots, r,\]
where $\eta >0$ denotes the learning rate, and each $\xi_k^j \sim \mathcal{N}(\mathbf{0}, \mathbb{I})$ independently. 
\begin{remark}
    A full non-asymptotic understanding of how these layers of approximation (including particle approximation, and timestepping and convergence of Langevin dynamics) affect the final generalisation error, are beyond the scope of this paper, which focuses on the asymptotic guarantees discussed above. We will demonstrate empirically in Section \ref{numerics} that we can approximate the behaviour shown by the Gibbs vector sufficiently.
    \par
    To our knowledge, the only work thus far to consider the non-asymptotic error in approximating $\mathfrak{m}_t(\nu_n)$ comes from \cite{langevin_monte_carlo, full_non_asymptotic}. Our work contains added complexities through the fact that we learn $\mathfrak{m}(\nu_n)$ by backwards induction, incurring errors at each minimisation problem. Assuming continuity of the approximate Q-functions $\widehat{Q}_t$ with respect to their measure arguments guarantees that these errors propagate smoothly. 
\end{remark}

\begin{algorithm}
    \caption{Gibbs Vector Algorithm}\label{gibbs vector algo}
    \begin{algorithmic}
        \Require training data $\{Z^{(i)}\}_{i=1}^{n}$, learning rate $\eta$, terminal control time $T$, terminal algorithm time $T_\tau$, network width $r$, reference controls $\{r_t\}_{t=0}^{T-1}$, regularisation parameters $\sigma, \beta >0$.
        \For{$t = T-1,\ldots, 0$}
            \For{$k = 0,\ldots, T_\tau -1$}
                \For{$j = 1,\ldots, r$}
                \State{generate $\xi_k^j \sim \mathcal{N}(\mathbf{0},\mathbb{I})$ }
                \State $\theta^j_{k+1} \gets \theta^j_k - \eta\nabla_{\theta^j}\Big(rR_t(\nu_n, m_k^r, \mathfrak{m}_{t+1}(\nu_n), \ldots, \mathfrak{m}_{T-1}(\nu_n))+\frac{1}{2\beta^2}U(\theta^j_k)\Big)+\frac{\sigma}{\beta}\sqrt{\eta}\xi_k^j$
                \EndFor
                \State $m_{k+1}^r \gets \frac{1}{r}\sum_{j=1}^{r}\delta_{\theta_{k+1}^j}$
            \EndFor
            \State $\mathfrak{m}_t(\nu_n) \gets m_{T_\tau}^r $
        \EndFor
    \end{algorithmic}
\end{algorithm}

We explicitly state this procedure in Algorithm \ref{gibbs vector algo}. Using standard deep learning packages such as PyTorch, making computation of such a gradient simple, we see that the algorithm provides a genuinely feasible computational approach to solving \eqref{entropy_minimisation_problem}.

\section{Numerical Experiments}\label{numerics}
In the following section we present two classic control problems. The code for our numerical experiments is available at \url{https://github.com/BorisBaros13/Overlearning}.
\subsection{Portfolio Allocation: The Merton Problem}
We begin by considering a simple portfolio allocation problem in a discrete-time financial market with $d$ assets which we assume to follow Markovian dynamics and whose values are denoted by the $d$-dimensional stock price process $S_t\in\mathbb{R}^d_+, t = 0, \ldots, T.$ It is not unreasonable to assume that stock prices are unaffected by portfolio allocation
of a small investor\footnote{For a large investor with market impact, we would have to address the counterfactual estimation
problem including the effect of the strategy on the environment, as done e.g.\ in \cite{giegrich2024limit} in the context of trade execution in limit order books.}. 
Therefore, we choose to model the returns, given by $Z_{t+1} := (S_{t+1} - S_t)/S_t$, as the stochastic environment, and importantly we assume these to be Markovian -- a natural extension would be to augment $Z$ with estimates of short-run trend and volatility, providing a richer state variable. One is then free to specify a control process $\pi_t = (\pi_t^{(k)})_{k=1}^{d}\in\mathbb{R}^d$, which denotes the holdings of each asset at time $t$ (in terms of dollars invested in each asset), the remainder being invested in a risk-free bond with constant one-period interest rate $r$. 
\par
Starting with some initial wealth $y > 0$, our aim is to optimally control the self-financing wealth dynamics of the portfolio value $(Y_t^\pi)_{t=0}^{T}$, evolving according to
\[ Y_{t+1}^\pi = (1+r)Y_t^\pi + \pi_t\cdot\Big( Z_{t+1} - r\mathbf{1}\Big),\quad Y_0^\pi = y.\]
\par
In our setting, we will evaluate the performance of the control process $\pi$ using the exponential utility functional
\[ J(\pi) := 1 - \exp(-\lambda Y_T^\pi),\]
where $\lambda > 0$ denotes a risk-aversion parameter.
\par
Noting that the augmented vector process $X := (Y^{\pi}, Z)$ is a Markov process, it is enough to consider controls of feedback form, taking $X$ to represent the state for the problem. We will take $T =2$ and pre-specify a uniform initial investment of $\pi_0 = (1/d,\cdots, 1/d)$ of initial wealth $y = 1$, so that the stochastic control problem simplifies to
\[ \mathrm{maximise}\ \pi_1\in\mathcal{C} \mapsto \mathbb{E}_{Z_1,Z_2}\Big[1-\exp(-\lambda Y_2^\pi(Z))\Big].\]
In our simulations we consider an interest-free financial market ($r=0$) with stochastic environment dynamics of the form
\[ Z_1 \sim \mathcal{U}[-1,1]^d, \quad Z_2 = \zeta\eta,\]
where $\eta\in\mathbb{R}^d$ is some fixed unit vector, and $\zeta\sim\mathcal{N}(m,s)$ independently of $Z_1$ for some hyperparameters $m,s$. Since $X_1$ is really a function of $Z_1$, by the dynamic programming principle we need to maximise the Q-function
\[Q_1(z,\pi_1) = \mathbb{E}_{Z_1,Z_2}\Big[1-\exp\big(-\lambda X_2^\pi(Z)\big)\Big| Z_1 = z\Big] = 1-\exp(-\lambda x)\mathbb{E}_{\zeta}\Big[\exp\Big\{-\frac{\lambda}{d}\mathbf{1}\cdot z -\lambda \zeta \pi_1(z)\cdot \eta\Big\}\Big].  \]
Simplifying, we find that this is maximised for constant control
\[ \pi^*(z) := \pi^* = \frac{m}{\lambda s^2}\eta,\]
with resulting expected reward given by 
\[ v^* = 1-\exp\Big\{-\frac{m^2}{2s^2}\Big\}.\]
For our implementation we have chosen hyperparameter values $\lambda = 1, m = 0.18, s = 0.44, d = 10$, so that the optimal value is $-0.9297$. In order to avoid numerical integration with our learned controls, we estimate the expected generalisation error of control $\pi_1$ by the point estimator
\[ \widehat{\mathrm{gen}}(\pi_1) := \mathbb{E}_{\nu_{\mathrm{test}}}[J(\pi_1)] - \mathbb{E}_{\nu_n}[J(\pi_1)],\]
where $\nu_{\mathrm{test}}$ denotes the empirical measure of samples of the stochastic environment drawn from $\nu_{\mathrm{pop}}$, and, as before, $\nu_n$ denotes the training distribution used to construct $\pi_1$. 
\par
In Figure \ref{unreg_merton} we present the point estimates for the unregularised Merton problem. As highlighted in the overlearning result of Section \ref{overlrn}, we indeed see a lack of stability in the algorithm's performance, with explosive behaviour at all levels of network width and sample size. Further, we visually see the transition between underparametrisation and overparametrisation. As discussed in \cite{overlearn}, for sufficiently underparametrised models, we may expect better generalisation, 
however, the generalisation point estimates at a sample size of $1000$ are still unreasonably high. For a positive terminal wealth we would have $J\in (0, 1),$ so that generalisation errors of the order we see in Figure \ref{unreg_merton} are evidently poor. 

\begin{figure}[htbp]
    \centering
    \includegraphics[width=0.9\textwidth]{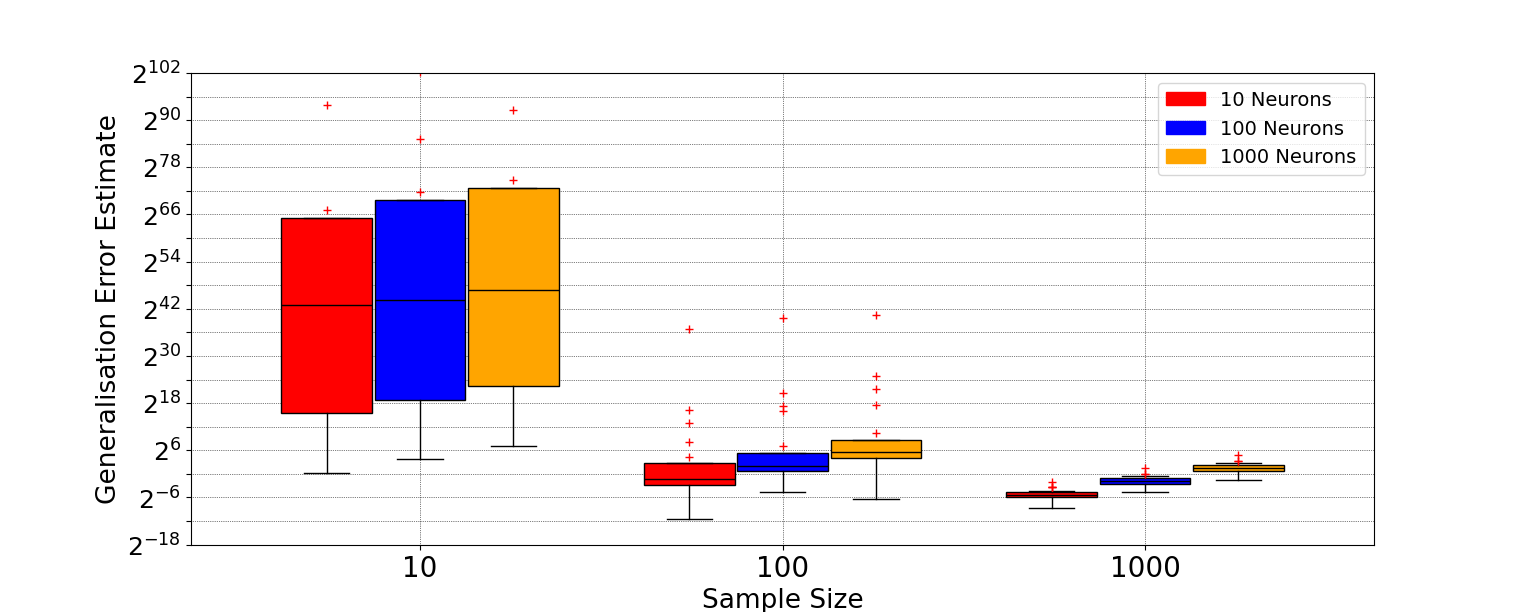}
    \vspace{-0.5cm}
    \caption{Generalisation error point estimates for unregularised Merton problem with 30 trials per setting. We omit multiple outliers for sample size of $10$, of size up to approximately $2^{200}$.}
    \label{unreg_merton}
\end{figure}
In Figure \ref{reg_merton} we present our generalisation error estimates using the Gibbs vector algorithm, where we have used a regularisation strength of $\frac{1}{2}$, coming from $\beta = \sigma = 100$. Training was done over $100,000$ epochs with a cosine annealing learning rate (as recommended in \cite{annealing}), starting from $0.1$ and finishing at $0.00001$. As demonstrated by the reference line, which represents a scale of $n^{-1}$, entropy regularisation induces a high degree of stability, even when trained with only 8 samples. It is also worth noting that similar generalisation errors are realised by neural networks over all network widths. This highlights the nature of the algorithm as a problem of statistically sampling from the Gibbs vector $\mathfrak{m}(\nu_n)$. Even with 10 hidden neurons, the mean-field neural network adequately samples from the Gibbs vector measures.
\begin{figure}[htbp]
    \centering
    \includegraphics[width=0.9\textwidth]{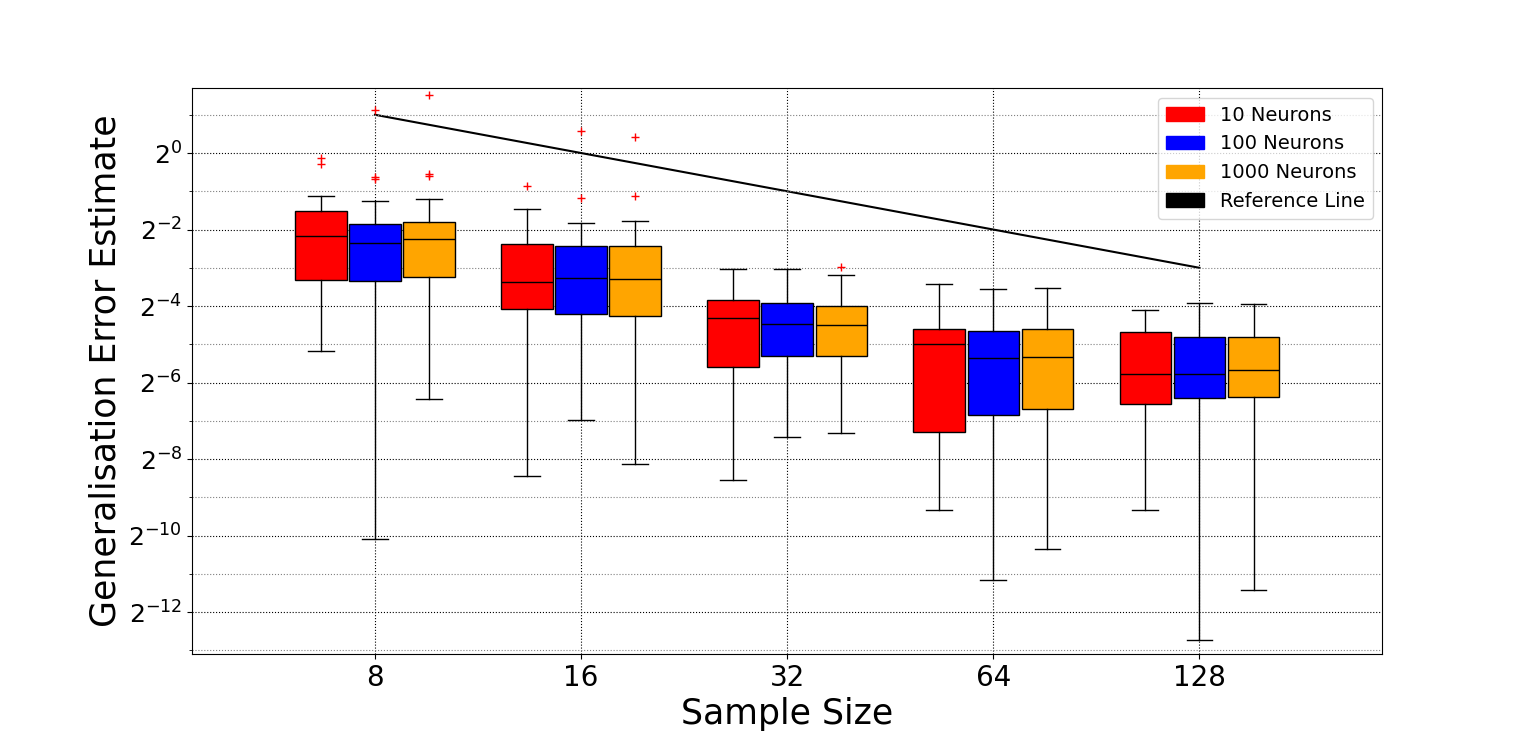}
    \vspace{-0cm}
    \caption{Generalisation error point estimates for regularised Merton problem with 50 trials per setting. The reference line is of scale $n^{-1}$ for comparison with results of Theorem \ref{1/n theoretical}.}
    \label{reg_merton}
\end{figure}
\subsection{Path Navigation: The Zermelo Problem}
We will now study a multi-period problem, to demonstrate that our bounds do not degenerate catastrophically with larger $T$. Specifically, we consider the problem of navigating a boat or plane from a starting position $X_0 = (-20, x_0),$ where $x_0 \sim \mathcal{U}[-1, 1]$, to a terminal point $(20, 0)$, aiming to avoid a circular obstacle, as displayed in Figure \ref{final zermelo}, in 50 time steps. This is a variation of the classical Zermelo navigation problem from \cite{Zermelo_navigation}. The stochastic environment $Z$ manifests in the form of a vertical wind, in this case modelled by an Ornstein--Uhlenbeck process, 
\[ dZ_t = \theta (\alpha - Z_t) \, \mathrm{d}t + \vartheta \, \mathrm{d}W_t, \quad Z_0 \sim \mathcal{U}[-1/2, 1/2],\]
where $W$ denotes a Brownian motion, and we choose hyperparameters $\theta = 1, \alpha = 0, \vartheta = 1.$ Discretising this process with time-step $\tau = 0.04$ generates a discrete-time process $Z = (Z_t)_{t=1}^{50}$.
\par
In this setting, the control process $\pi$ denotes the chosen angle with the positive x-axis, increasing anti-clockwise. Explicitly, the stochastic environment $Z$ and control process $\pi$ contribute to the realised trajectory $\mathbf{X}^{\pi}(Z)$ via the transition function
\begin{align*}
    \mathbf{X}_{t+1}^{\pi}(Z) := \begin{pmatrix}
        X_{t+1}^{\pi}(Z)\\
        Y_{t+1}^{\pi}(Z)
    \end{pmatrix}
    = 
    \mathbf{X}_{t}^{\pi}(Z)
    +
    v_s\begin{pmatrix}
        \sin(\pi_t)\\
        \cos(\pi_t)
    \end{pmatrix}
    + \begin{pmatrix}
        0 \\
        Z_{t+1}
    \end{pmatrix},
\end{align*}
where we choose the speed of the boat to be $v_s = 0.8$. 
The aim of the problem is then to choose some angle sequence $\pi$ to minimise the expectation of the loss function
\[ \ell(\mathbf{X}^\pi(Z), \pi) = \big\Vert \mathbf{X}_{50}^\pi(Z) - 20 \big\Vert^2 + M\sum_{t=0}^{50}\Bigg(1 - \frac{1}{1+\exp\big\{A\big(1-\norm{\mathbf{X}_t^\pi(Z)}^2\big)\big\}}\Bigg), \]
where we choose $M = 10, A = 2$ to indicate a soft (importantly, differentiable) version of forbidding passage through the unit circle. 

\par
Choosing our state to be the augmented vector process $(X^\pi, Y^\pi, Z)$, for our training we take 100 samples of $Z$, a regularisation strength of $1/200$ coming from $\beta = 100$ and $\sigma = \sqrt{0.1}\beta$, 100 hidden neurons, train over 20,000 epochs with a cosine annealing learning starting from 1 and decreasing to 0.00001, and use a reference control of constant heading in the positive x direction. We see in Figure \ref{final zermelo} that the in-sample performance does extremely well, demonstrating that the bias induced by entropy regularisation does not significantly deteriorate performance. It is particularly interesting to see that the control circumvents the circular obstacle by following the wind, rather than ever going against it. In Appendix \ref{zermelo} we display more images of the backwards inductive minimisation, which effectively show the algorithm's progress over the backwards time steps from purely reference-controlled states to learned actions -- in particular, we note the eventual ability to circumvent the obstacle. 
\par
In Figures \ref{red_green}, \ref{hists} we visualise the performance in- and out-of-sample for 1,000 unseen samples of $Z$. Not only does the algorithm perform well in-sample, but we see that areas around the obstacle unvisited in-sample are still well-traversed and eventually directed close to the target. Therefore, we can deduce an effective balance of bias and stability of the Gibbs vector algorithm, and importantly empirically show that the generalisation bounds are likely much tighter than those we demonstrate in Theorem \ref{generalisation error}.
\par
This is in contrast to Figure \ref{red_green_unreg}, where we visualise the in-sample and out-of-sample performance for  unregularised learning, for which we use the same training parameters as above. Although the out-of-sample behaviour is not as markedly catastrophic compared to that for the Merton problem in Figure \ref{unreg_merton}, we still see collision with the obstacle for an out-of-sample wind trajectory. Note that for our training we incorporated early stopping, which acts as an implicit regulariser -- this stopped the unregularised model from overlearning too harshly. In reality, we would train for longer, so we should really expect out-of-sample performance to look much worse for unregularised learning.
\begin{figure}[H]
    \centering
    \includegraphics[width = 0.9\textwidth]{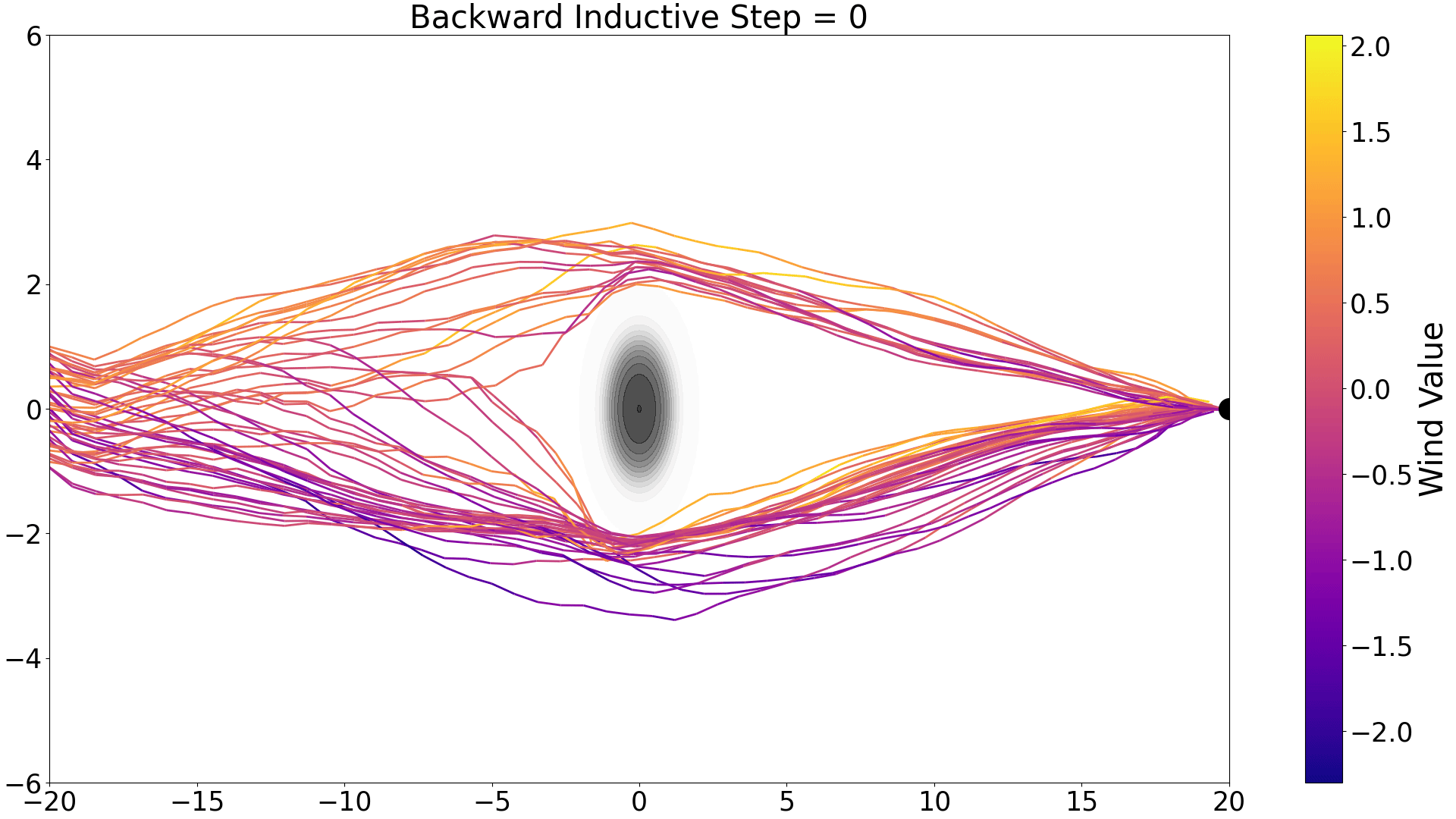}
    \vspace{-0.5cm}
    \caption{Final in-sample performance over first 50 training samples, coloured by wind.}
    \label{final zermelo}
\end{figure}
\begin{figure}[H]
    \centering
    \includegraphics[width=0.92\textwidth, height = 8cm]{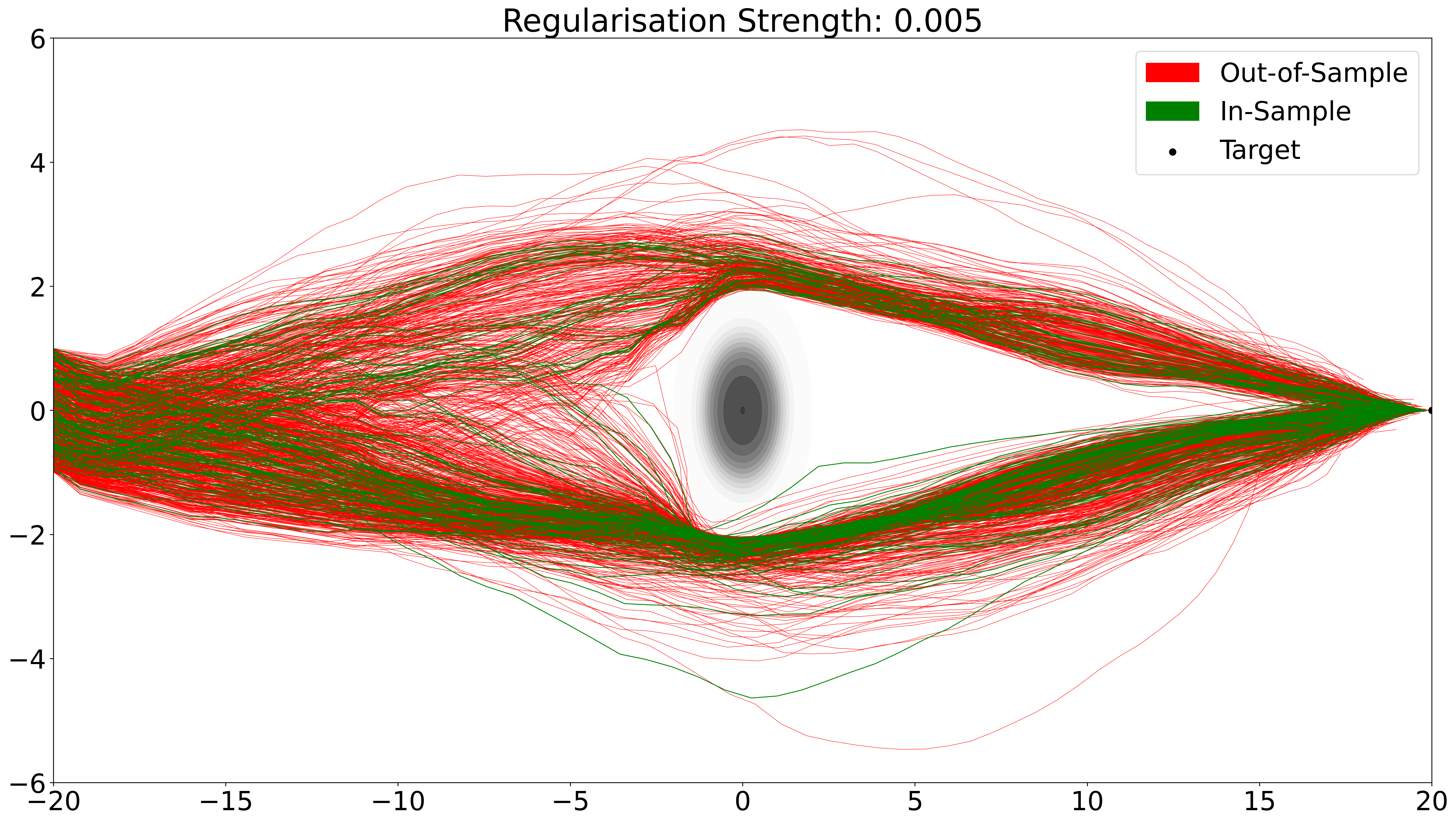}
    \vspace{-0.5cm}
    \caption{In-sample and out-of-sample performance for regularised learning over 100 training samples and 1000 testing samples.}
    \label{red_green}
\end{figure}
\begin{figure}[H]
    \centering
    \includegraphics[width=0.9\textwidth]{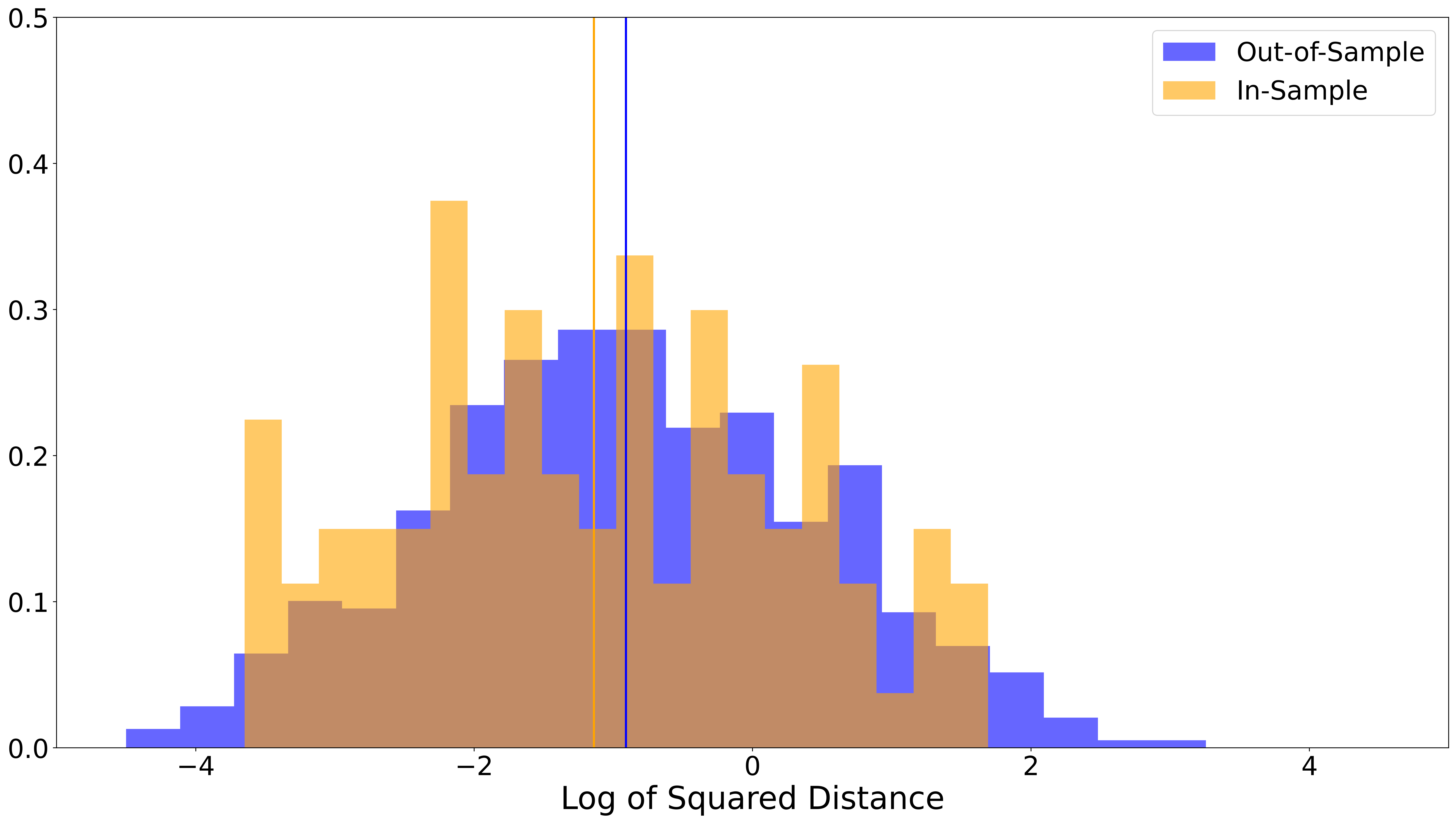}
    \vspace{-0.5cm}
    \caption{In-sample and out-of-sample performance (the logarithm of the terminal loss, which is the squared distance from the target) over 100 training samples and 1000 testing samples. Means denoted by vertical lines.}
    \label{hists}
\end{figure}
\begin{figure}[H]
    \centering
    \includegraphics[width=0.92\textwidth, height = 8cm]{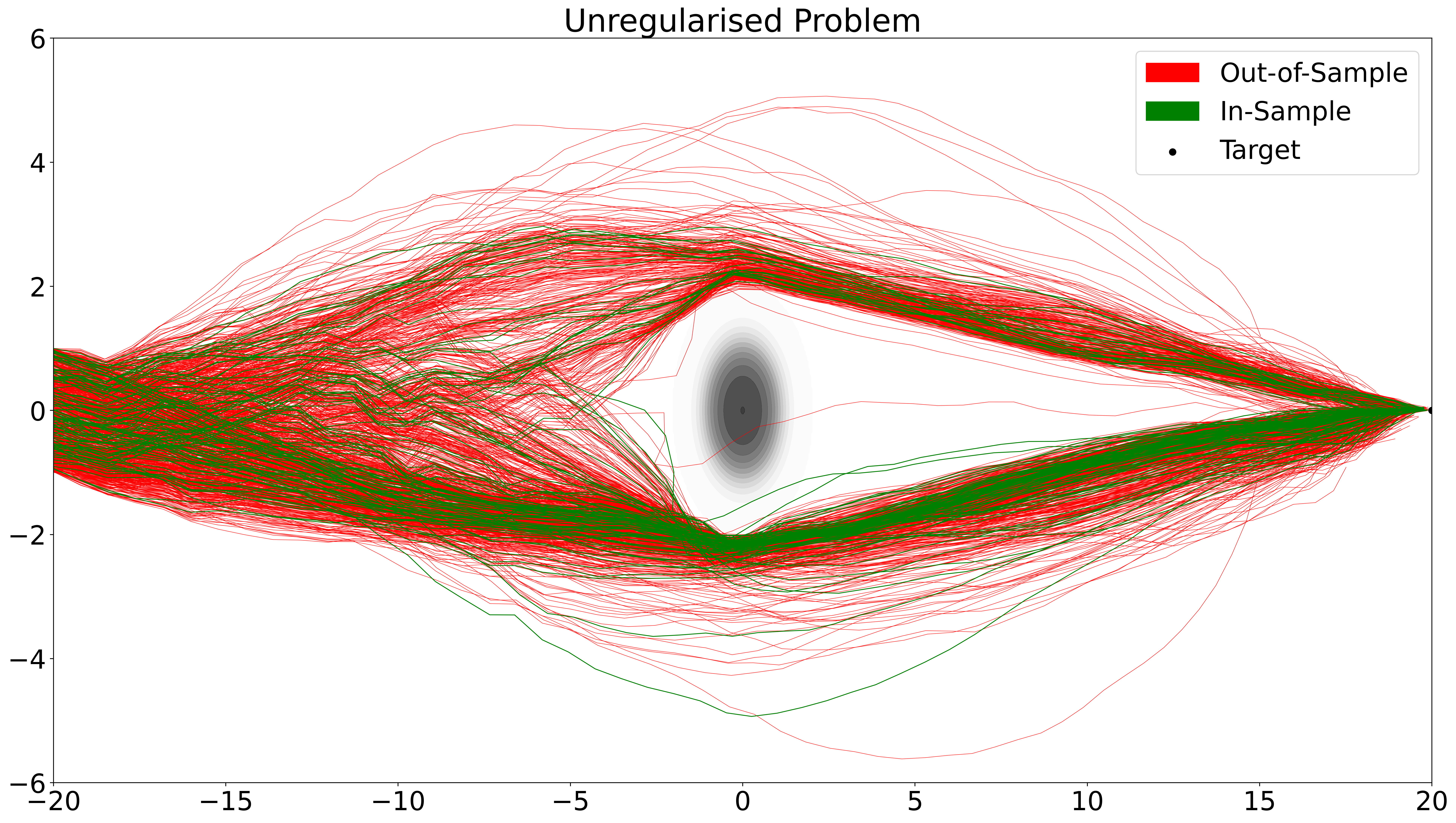}
    \vspace{-0.5cm}
    \caption{In-sample and out-of-sample performance for unregularised learning over 100 training samples and 1000 testing samples.}
    \label{red_green_unreg}
\end{figure}
\bibliography{bibs}

\begin{thebibliography}{48}
\providecommand{\natexlab}[1]{#1}
\providecommand{\url}[1]{\texttt{#1}}
\expandafter\ifx\csname urlstyle\endcsname\relax
  \providecommand{\doi}[1]{doi: #1}\else
  \providecommand{\doi}{doi: \begingroup \urlstyle{rm}\Url}\fi

\bibitem[Aminian et~al.(2023)Aminian, Cohen, and Szpruch]{sam_ge}
Gholamali Aminian, Samuel Cohen, and {\L}ukasz Szpruch.
\newblock Mean-Field Analysis of Generalization Errors.
\newblock 2023.

\bibitem[Bartlett and Mendelson(2002)]{rad_gauss}
Peter~L. Bartlett and Shahar Mendelson.
\newblock Rademacher and Gaussian Complexities: Risk Bounds and Structural Results.
\newblock \emph{Journal of Machine Learning Research}, 3:\penalty0 463–--482, 2002.

\bibitem[Belkin et~al.(2019)Belkin, Hsu, Ma, and Mandal]{bias_variance}
Mikhail Belkin, Daniel Hsu, Siyuan Ma, and Soumik Mandal.
\newblock Reconciling Modern Machine Learning Practice and the Bias-Variance Trade-Off.
\newblock \emph{Proceeding of the National Academy of Sciences of the United States of America}, 116\penalty0 (32):\penalty0 15849--15854, 2019.

\bibitem[Bertsekas and Tsitsiklis(1996)]{ndp}
Dimitri Bertsekas and John~N. Tsitsiklis.
\newblock \emph{Neuro-Dynamic Programming}.
\newblock Athena Scientific, 1996.

\bibitem[Bertsekas and Shreve(1996)]{bertsekasShreve}
Dimitri Bertsekas and Steven~E Shreve.
\newblock \emph{Stochastic Optimal Control: The Discrete-Time Case}, volume~5.
\newblock Athena Scientific, 1996.

\bibitem[Bertsekas(2019)]{RL_book}
Dimitri Bertsekas.
\newblock \emph{Reinforcement Learning and Optimal Control}.
\newblock Athena Scientific, 2019.

\bibitem[Bortoli et~al.(2020)Bortoli, Durmus, Fontaine, and Simsekli]{prop_chaos_SGD}
Valentin~De Bortoli, Alain Durmus, Xavier Fontaine, and Umut Simsekli.
\newblock Quantitative Propagation of Chaos for SGD in Wide Neural Networks.
\newblock \emph{NIPS'20: Proceedings of the 34th International Conference on Neural Information Processing Systems}, \penalty0 (24):\penalty0 278 -- 288, 2020.

\bibitem[Bousquet and Elisseeff(2002)]{leave_one_out}
Olivier Bousquet and Andr\'{e} Elisseeff.
\newblock Stability and Generalization.
\newblock \emph{Journal of Machine Learning Research}, 2:\penalty0 499--–526, 2002.

\bibitem[Buehler et~al.(2019)Buehler, Gonon, Teichmann, and Wood]{buehler2019deep}
Hans Buehler, Louis Gonon, Josef Teichmann, and Ben Wood.
\newblock Deep hedging.
\newblock \emph{Quantitative Finance}, 19\penalty0 (8):\penalty0 1271--1291, 2019.

\bibitem[Cardaliaguet et~al.(2015)Cardaliaguet, Delarue, Lasry, and Lions]{cardaliaguet_normalisation}
Pierre Cardaliaguet, Fran\c{c}ois Delarue, Jean-Michel Lasry, and Pierre-Louis Lions.
\newblock \emph{The Master Equation and the Convergence Problem in Mean Field Games}.
\newblock Princeton University Press, 2015.

\bibitem[Carmona and Delarue(2015)]{total_derivative_fact}
René Carmona and François Delarue.
\newblock Forward-Backward Stochastic Differential Equations and Controlled McKean Vlasov Dynamics.
\newblock \emph{Annals of Probability}, 43\penalty0 (5):\penalty0 2647--2700, 2015.

\bibitem[Carmona and Delarue(2018)]{mean_field_games}
René Carmona and François Delarue.
\newblock \emph{Probabilistic Theory of Mean Field Games with Applications I}.
\newblock Probability Theory and Stochastic Modelling. Springer, 2018.

\bibitem[Chen et~al.(2023)Chen, Ren, and Wang]{uniform-in-time-chaos}
Fan Chen, Zhenjie Ren, and Songbo Wang.
\newblock Uniform-in-Time Propagation of Chaos for Mean field Langevin Dynamics, 2023.

\bibitem[Chizat(2022)]{annealing}
Lénaïc Chizat.
\newblock Mean-Field Langevin Dynamics: Exponential Convergence and Annealing, 2022.

\bibitem[Dupuis and Ellis(1997)]{weak_compactness}
Paul Dupuis and Richard~S. Ellis.
\newblock \emph{A Weak Convergence Approach to the Theory of Large Deviations}.
\newblock Wiley series in probability and statistics. Probability and statistics. Wiley, 1997.

\bibitem[Fu et~al.(2024)Fu, Chen, Zhang, Yang, Ma, and Yang]{synthetic}
Fanzhe Fu, Junru Chen, Jing Zhang, Carl Yang, Lvbin Ma, and Yang Yang.
\newblock Are Synthetic Time-series Data Really not as Good as Real Data?, 2024.

\bibitem[Giegrich et~al.(2024)Giegrich, Oomen, and Reisinger]{giegrich2024limit}
Michael Giegrich, Roel Oomen, and Christoph Reisinger.
\newblock Limit Order Book Simulation and Trade Evaluation with $K$-Nearest-Neighbor Resampling.
\newblock \emph{arXiv preprint arXiv:2409.06514}, 2024.

\bibitem[Golse(2022)]{mean_field_OG}
François Golse.
\newblock Mean-Field Limits in Statistical Dynamics, 2022.

\bibitem[Géron(2019)]{exploding_gradients}
Aurélien Géron.
\newblock \emph{Hands-On Machine Learning with Scikit-Learn, Keras, and TensorFlow}.
\newblock O'Reilly Media, 2019.

\bibitem[Han and Weinan(2016)]{DL-SC}
Jiequn Han and E~Weinan.
\newblock Deep Learning Approximation for Stochastic Control Problems, 2016.

\bibitem[Hernández and Possamaï(2021)]{me_myself_I}
Camilo Hernández and Dylan Possamaï.
\newblock Me, Myself and I: A General Theory of Non-Markovian Time-Inconsistent Stochastic Control for Sophisticated Agents, 2021.

\bibitem[Hu et~al.(2019)Hu, Ren, \v{S}i\v{s}ka, and Szpruch]{langevin}
Kaitong Hu, Zhenjie Ren, David \v{S}i\v{s}ka, and {\L}ukasz Szpruch.
\newblock Mean-Field Langevin Dynamics and Energy Landscape of Neural Networks.
\newblock \emph{Annales de l'Institut Henri Poincar\'{e}, Probabilit\'{e}s et Statistiques}, 2019.

\bibitem[Hu and Laurière(2024)]{ML_SC}
Ruimeng Hu and Mathieu Laurière.
\newblock Recent Developments in Machine Learning Methods for Stochastic Control and Games.
\newblock \emph{Numerical Algebra, Control and Optimization}, 14\penalty0 (3):\penalty0 435--525, 2024.

\bibitem[Huang and Cheng(2022)]{fuel_models}
Chenyu Huang and Xiaoyue Cheng.
\newblock Estimation of Aircraft Fuel Consumption by Modeling Flight Data from Avionics Systems.
\newblock \emph{Journal of Air Transport Management}, 99, 2022.

\bibitem[Huré et~al.(2021{\natexlab{a}})Huré, Pham, Bachouch, and Langrené]{algo_paper_1}
Côme Huré, Huyên Pham, Achref Bachouch, and Nicolas Langrené.
\newblock Deep Neural Networks Algorithms for Stochastic Control Problems on Finite Horizon: Convergence Analysis.
\newblock \emph{SIAM Journal on Numerical Analysis}, 59\penalty0 (1):\penalty0 525–--557, 2021{\natexlab{a}}.

\bibitem[Huré et~al.(2021{\natexlab{b}})Huré, Pham, Bachouch, and Langrené]{algo_paper_2}
Côme Huré, Huyên Pham, Achref Bachouch, and Nicolas Langrené.
\newblock Deep Neural Networks Algorithms for Stochastic Control Problems on Finite Horizon: Numerical Applications.
\newblock \emph{Methodology and Computing in Applied Probability}, 24\penalty0 (1):\penalty0 143--–178, 2021{\natexlab{b}}.

\bibitem[Jacot et~al.(2018)Jacot, Gabriel, and Hongler]{NTK}
Arthur Jacot, Franck Gabriel, and Clément Hongler.
\newblock Neural Tangent Kernel: Convergence and Generalization in Neural Networks.
\newblock \emph{NIPS'18: Proceedings of the 32nd International Conference on Neural Information Processing Systems}, pages 8580 -- 8589, 2018.

\bibitem[Kerimkulov et~al.(2022)Kerimkulov, Leahy, Šiška, and Szpruch]{policy_entropy}
Bekzhan Kerimkulov, James-Michael Leahy, David Šiška, and Lukasz Szpruch.
\newblock Convergence of Policy Gradient for Entropy Regularized MDPs with Neural Network Approximation in the Mean-Field Regime.
\newblock \emph{Proceedings of the 39th International Conference on Machine Learning}, 2022.

\bibitem[Kou et~al.(2016)Kou, Peng, and Xu]{em}
Steven Kou, Xianhua Peng, and Xingbo Xu.
\newblock EM Algorithm and Stochastic Control in Economics, 2016.

\bibitem[Lascu and Majka(2025)]{nonconvex}
Razvan-Andrei Lascu and Mateusz~B. Majka.
\newblock Non-Convex Entropic Mean-Field Optimization via Best Response Flow, 2025.

\bibitem[Mei et~al.(2018)Mei, Montanari, and Nguyen]{mean_field_mei}
Song Mei, Andrea Montanari, and Phan-Minh Nguyen.
\newblock A Mean Field View of the Landscape of Two-Layer Neural Networks.
\newblock \emph{Proceedings of the National Academy of Sciences of the United States of America}, 115\penalty0 (33), 2018.

\bibitem[Meunier et~al.(2025)Meunier, Reisinger, and Zhang]{ent_christoph}
Matthieu Meunier, Christoph Reisinger, and Yufei Zhang.
\newblock Efficient Learning for Entropy-Regularized Markov Decision Processes via Multilevel Monte Carlo, 2025.

\bibitem[Mousavi-Hosseini et~al.(2023)Mousavi-Hosseini, Farghly, He, Balasubramanian, and Erdogdu]{langevin_monte_carlo}
Alireza Mousavi-Hosseini, Tyler Farghly, Ye~He, Krishnakumar Balasubramanian, and Murat~A. Erdogdu.
\newblock Towards a Complete Analysis of Langevin Monte Carlo: Beyond Poincaré Inequality.
\newblock \emph{Proceedings of Machine Learning Research}, 195:\penalty0 1--35, 2023.

\bibitem[Nakkiran et~al.(2021)Nakkiran, Kaplun, Bansal, Yang, Barak, and Sutskever]{double_descent}
Preetum Nakkiran, Gal Kaplun, Yamini Bansal, Tristan Yang, Boaz Barak, and Ilya Sutskever.
\newblock Deep Double Descent: Where Bigger Models and More Data Hurt.
\newblock \emph{Journal of Statistical Mechanics: Theory and Experiment}, 2021, 2021.

\bibitem[Nitanda et~al.(2021)Nitanda, Wu, and Suzuki]{NitandaWuSuzuki2021ParticleDualAveraging}
Atsushi Nitanda, Denny Wu, and Taiji Suzuki.
\newblock Particle Dual Averaging: Optimization of Mean Field Neural Network with Global Convergence Rate Analysis.
\newblock \emph{Advances in Neural Information Processing Systems}, 34:\penalty0 19608--19621, 2021.

\bibitem[Pham(2005)]{stoch_control_survey}
Huy\^{e}n Pham.
\newblock On Some Recent Aspects of Stochastic Control and their Applications.
\newblock \emph{Probability Surveys}, 2:\penalty0 506--549, 2005.

\bibitem[Pham(2009)]{control_pham}
Huy\^{e}n Pham.
\newblock \emph{Continuous-Time Stochastic Control and Optimization with Financial Applications}.
\newblock Springer, 2009.

\bibitem[Rahimi and Recht(2008)]{rand_features}
Ali Rahimi and Benjamin Recht.
\newblock Uniform Approximation of Functions with Random Bases.
\newblock \emph{46th Annual Allerton Conference on Communication, Control, and Computing}, pages 555--561, 2008.

\bibitem[Reisinger and Zhang(2021)]{ent_christoph_2}
Christoph Reisinger and Yufei Zhang.
\newblock Regularity and Stability of Feedback Relaxed Controls.
\newblock \emph{SIAM Journal on Control and Optimization}, 59\penalty0 (5), 2021.

\bibitem[Reppen and Soner(2022)]{overlearn}
Anders~Max Reppen and Halil~Mete Soner.
\newblock Deep Empirical Risk Minimization in Finance: Looking Into the Future.
\newblock \emph{Mathematical Finance}, 33:\penalty0 116--145, 2022.

\bibitem[Reppen et~al.(2023)Reppen, Soner, and Tissot-Daguette]{DL-finance}
Anders~Max Reppen, Halil~Mete Soner, and Valentin Tissot-Daguette.
\newblock Deep Stochastic Optimization in Finance.
\newblock \emph{Digital Finance}, 5, 2023.

\bibitem[Sirignano and Spiliopoulos(2020)]{mean-field-depth}
Justin Sirignano and Konstantinos Spiliopoulos.
\newblock Mean Field Analysis of Deep Neural Networks: A Law of Large Numbers.
\newblock \emph{SIAM Journal on Applied Mathematics}, 80\penalty0 (2):\penalty0 725 -- 752, 2020.

\bibitem[Suzuki et~al.(2023)Suzuki, Wu, and Nitanda]{full_non_asymptotic}
Taiji Suzuki, Denny Wu, and Atsushi Nitanda.
\newblock Convergence of Mean-Field Langevin Dynamics: Time and Space Discretization, Stochastic Gradient, and Variance Reduction, 2023.

\bibitem[Sznitman(1989)]{topics_chaos}
Alain-Sol Sznitman.
\newblock Topics in Propagation of Chaos.
\newblock \emph{Ecole d'Et\'e de Probabilit\'es de Saint-Flour}, 19:\penalty0 165--251, 1989.

\bibitem[Vapnik and Chervonenkis(1971)]{vapnik}
Vladimir~N. Vapnik and Alexey~Ya. Chervonenkis.
\newblock On the Uniform Convergence of Relative Frequencies of Events to Their Probabilities.
\newblock \emph{Theory of Probability \& Its Applications}, 16\penalty0 (2):\penalty0 264--280, 1971.

\bibitem[Wang et~al.(2020)Wang, Zariphopoulou, and Zhou]{Wang2020RLContinuousTime}
Haoran Wang, Thaleia Zariphopoulou, and Xun~Yu Zhou.
\newblock Reinforcement Learning in Continuous Time and Space: A Stochastic Control Approach.
\newblock \emph{Journal of Machine Learning Research}, 21\penalty0 (198):\penalty0 1--34, 2020.

\bibitem[Zermelo(1931)]{Zermelo_navigation}
Ernst Zermelo.
\newblock {{\"U}ber das Navigationsproblem bei ruhender oder ver{\"a}nderlicher Windverteilung}.
\newblock \emph{Zeitschrift Angewandte Mathematik und Mechanik}, 11\penalty0 (2):\penalty0 114--124, 1931.

\bibitem[Zhang et~al.(2021)Zhang, Bengio, Hardt, Recht, and Vinyals]{empirical_ge}
Chiyuan Zhang, Samy Bengio, Moritz Hardt, Benjamin Recht, and Oriol Vinyals.
\newblock Understanding Deep Learning Requires Rethinking Generalisation.
\newblock \emph{Communications of the ACM}, 64:\penalty0 107--115, 2021.

\end{thebibliography}
\begin{appendix}
    \section{Some Prerequisites from Calculus on the Space of Probability Measures}\label{prereq}
We briefly present some standard definitions and concepts from calculus on the space of probability measures, which will become crucial in our analysis of the generalisation error. The below definitions may be found in \cite{cardaliaguet_normalisation, total_derivative_fact}.

By $\mathcal{P}(\Theta)$ we denote the space of probability measures on $\Theta$, and by $\mathcal{P}_p(\Theta)$ the subspace in which measures have finite $p$-moments for $p\geq 1$. Taking $\Theta = \mathbb{R}^d$, we denote the \textit{Wasserstein-p metric} on $\mathcal{P}_p(\mathbb{R}^d)$ by
\[ \mathcal{W}_p(\mu,\nu) := \inf\left\{\left(\int_{\mathbb{R}^d\times\mathbb{R}^d}\abs{x-y}^p\pi(\textrm{d}x,\textrm{d}y)\right)^{\frac{1}{p}}:\pi \textrm{ is a coupling of $\mu$ and $\nu$}\right\}\]
for $\mu,\nu\in\mathcal{P}_p(\mathbb{R}^d)$. By a coupling $\pi$ of $\mu$ and $\nu$, we mean a probability measure $\pi \in \mathcal{P}_p(\mathbb{R}^d\times\mathbb{R}^d)$ such that the marginals satisfy $\pi(A\times\mathbb{R}^d) = \mu(A)$ and $\pi(\mathbb{R}^d\times A) = \nu(A)$ for some $A\in\mathcal{B}(\mathbb{R}^d)$. 
Without proof we present the standard results (also stated explicitly in \cite{sam_ge})
\begin{enumerate}
    \item 
    $(\mathcal{P}_p(\mathbb{R}^d),\mathcal{W}_p)$ is a Polish space;
    \item 
    $\mathcal{W}_p(\mu_n,\nu)\to 0$ if and only if $\mu_n$ weakly converges to $\mu$ and $\int_{\mathbb{R}^d}\abs{x}^p\mu_n(\textrm{d}x)\to\int_{\mathbb{R}^d}\abs{x}^p\mu(\textrm{d}x)$;
    \item 
    For all $q>p$ the set $\{\mu\in\mathcal{P}_p(\mathbb{R}^d):\int_{\mathbb{R}^d}\abs{x}^q\mu\textrm{d}x\leq C \}$ is $\mathcal{W}_p$-compact.
\end{enumerate}
In order to analyse minimisation problems in the space of probability measures, it is important to develop some notion of derivative, in this case known as the linear functional derivative. 
\begin{definition}\label{frechet definition}
    For a function $F:\mathcal{P}(\mathbb{R}^d)\times\mathbb{R}^k\to\mathbb{R}$, we say the map $m\mapsto F(m,x)$ is in $C^1$ if there exists a map $\frac{\delta F}{\delta m}:\mathcal{P}(\mathbb{R}^d)\times\mathbb{R}^k\times\mathbb{R}^d \to \mathbb{R}$ such that
    \begin{enumerate}
        \item 
        $\frac{\delta F}{\delta m}$ is measurable with respect to $x,a,$ and continuous with respect to $m$;
        \item 
        For every bounded $B\subset\mathcal{P}_2(\mathbb{R}^d)\times\mathbb{R}^k$, there exists a constant $C>0$ such that $\abs{\frac{\delta F}{\delta m}(m,x,a)}\leq C(1+\abs{a}^2)$ for all $(m,x)\in B$;
        \item 
        For all $m,m'\in\mathcal{P}_2(\mathbb{R}^d)$,
        \[F(m',x)-F(m,x) = \int_0^1\int_{\mathbb{R}^d} \frac{\delta F}{\delta m}(m+\lambda(m'-m),x;a)(m'-m)(\textrm{d}a)\textrm{d}\lambda.\]
    \end{enumerate}
Noting that the functional linear derivative is only defined up to some constant, we impose the normalisation condition $\int \frac{\delta F}{\delta m} (m,x;a)m(\textrm{d}a)=0$. Further to the above, we say $F$ is $\mathcal{C}^2$ if both $F$ and $\frac{\delta F}{\delta m}$ are $\mathcal{C}^1$, and so on for higher derivatives. 
\par
In the case that $F:\mathcal{P}(\mathbb{R}^d)\to \mathbb{R},$ we give the map $m\mapsto F(m)$ all of the definitions above, simply excluding the $x$ variable. 
\par
If, further to the above definition, the map $(m,x,a,)\mapsto \frac{\delta F}{\delta m}(m,x,a)$ is continuously differentiable in $x$, then the intrinsic derivative is given by
\[ D_m F(m,x;a) := \nabla\left(\frac{\delta F}{\delta m}(m,x;a)\right),\]
where the gradient $\nabla$ is taken over $x\in\mathbb{R}^k$. 
\end{definition}

    \section{Supplementary Inequalities}\label{appA}
In the following section we explicitly demonstrate upper bounds which will repeatedly become useful throughout our main proofs. Where products are undefined, we take them to be $1$, and where sums are undefined we take them to be $0$. We also perform computations in one dimension, though the multivariate setting follows equivalently. 

\begin{lemma}\label{norm x}
    Under Assumption \ref{verifiable assumptions}, for $\mathbf{m} = (m_{s})_{s=0}^{T-1}$ and some $t$, there exists $C > 0$ such that
    \[ (1+\norm{X_t^{\mathbf{m}}(Z)}) \leq C(1+\norm{x_0})\prod_{s=0}^{t-1}(1+E_{m_s}^{(2)})(1+\norm{Z_{s+1}}).\]
\end{lemma}
\begin{proof}
    Begin by noting, from Assumption \ref{verifiable assumptions}(iv), that
    \[ 1+\norm{X_t^{\mathbf{m}}(Z)} \leq C\Big(1+\norm{X_{t-1}^{\mathbf{m}}(Z)} + \norm{u_{m_t-1}(X_{t-1}^{\mathbf{m}}(Z))} + \norm{Z_{t}}\Big).\]
    From Assumption \ref{verifiable assumptions}(vi), $ \norm{u_{m_{t-1}}(X_{t-1}^\mathbf{m}(Z))} \leq C(1+\norm{X_{t-1}^{\mathbf{m}}(Z)})(1+E_{m_{t-1}}^{(2)})$,
    so we can form the recursion
    \begin{align*}
        1+\norm{X_t^{\mathbf{m}}(Z)} & \leq C(1+\norm{X_{t-1}^{\mathbf{m}}(Z)})(1+E_{m_{t-1}}^{(2)})(1+\norm{Z_t}) \\
        & \leq \cdots \\
        & \leq C(1+\norm{x_0})\prod_{s=0}^{t-1}(1+E_{m_s}^{(2)})(1+\norm{Z_{s+1}}).
    \end{align*}
\end{proof}
\begin{lemma}\label{u derivs}
    Under Assumption \ref{verifiable assumptions}, there exists $C>0$ such that 
    \begin{align*}
        \partial_x u_m(x) & \leq C(1+E_m^{(2)}), \\
        \partial_x^2 u_m(x) & \leq C(1+E_m^{(2)})(1+E_m^{(4)}).
    \end{align*}
\end{lemma}
\begin{proof}
    We see that 
    \begin{align*}
        \partial_x u_m(x) & = \partial_x \int_\Theta a\sigma(\omega x + b)m(\mathrm{d}\theta) = \int_\Theta a\omega\sigma'(\omega x + b) m(\mathrm{d}\theta) \leq C\int_\Theta a\omega m(\mathrm{d}\theta) \\ & \leq C\int_\Theta (a^2 + \omega^2) m(\mathrm{d}\theta) \leq C\int_\Theta \norm{\theta}^2 m(\mathrm{d}\theta) \leq C(1+E_m^{(2)}).  
    \end{align*}
    Similarly, 
    \begin{align*}
        \partial_x^2 u_m(x) & = \int_\Theta a\omega^2 \sigma''(\omega x + b)m(\mathrm{d}\theta) \leq C\int_\Theta a\omega^2 m(\mathrm{d}\theta) \leq C\sqrt{\int_\Theta a^2 m(\mathrm{d}\theta)\int_\Theta (a\omega)^2 m(\mathrm{d}\theta)} \\
        & \leq C\sqrt{(1+E_m^{(2)})(1+E_m^{(4)})} \leq C(1+E_m^{(2)})(1+E_m^{(4)}).
    \end{align*}
\end{proof}
\begin{lemma}\label{frechet_state}
    There exists $C >0$ such that, for $s >t$, $\mathbf{m} := (m_l)_{l=0}^{T-1}$, $x\in\mathcal{X}$,
    \[ \frac{\delta}{\delta m_t} X_s^{t, x, \mathbf{m}}(Z) \leq C(1+\norm{x})(1+\norm{\theta}^2+E_{m_t}^{(2)})\prod_{l=t+1}^{s-1}(1+E_{m_l}^{(2)}),\]
    where $X^{t,x,\mathbf{m}}(Z)$ denotes the state process with initial value $X_t^{t,x,\mathbf{m}}(Z) = x$, controlled by measures from $\mathbf{m}$ thereafter. 
    Similarly, for some $l > s$, we have
    \[ \frac{\delta}{\delta m_s}X_l^{t, x, \mathbf{m}}(Z) \leq C (1+\norm{X_s^{t,x,\mathbf{m}}(Z)})(1+\norm{\theta}^2 + E^{(2)}_{m_s})\prod_{q = s+1}^{l-1}(1+E^{(2)}_{m_q}).\]
\end{lemma}
\begin{proof}
    \begin{align*}
    \frac{\delta}{\delta m_t}X_s^{t,x,\mathbf{m}}(Z) & = \frac{\delta}{\delta m_t}h_{s-1}\Big(X_{s-1}^{t,x,\mathbf{m}}(Z), u_{m_{s-1}}(X_{s-1}^{t,x,\mathbf{m}}(Z)), Z_s\Big) \\
    & = \Big(\partial_x h_{s-1} + (\partial_u h_{s-1})(\partial_x u_{m_{s-1}})\Big)\frac{\delta}{\delta m_t}X_{s-1}^{t,x,\mathbf{m}}(Z) \\
    & =: D_{s-1}\frac{\delta}{\delta m_t}X_{s-1}^{t,x,\mathbf{m}}(Z) \\
    & = D_{s-1}\cdots D_{t+1}\frac{\delta}{\delta m_t} h_t(x, u_{m_t}(x), Z_{t+1}) \\
    & = D_{s-1}\cdots D_{t+1} (\partial_u h_t) \Big(\phi(x,\theta) - \mathbb{E}_{\theta\sim m_t}[\phi(x, \theta)]\Big).
    \end{align*}
    Applying Assumption \ref{verifiable assumptions} yields the upper bound. A similar computation yields the second result.
\end{proof}
\begin{lemma}\label{x state derivs}
    There exists $C > 0$ such that, for $s >t, \mathbf{m} := (m_l)_{l=0}^{T-1}, x\in\mathcal{X}$,
    \[ \frac{\partial}{\partial x}X_s^{t,x,\mathbf{m}}(Z) \leq C \prod_{l=t}^{s-1}(1+E_{m_l}^{(2)}),\]
    and 
    \[ \frac{\partial^2}{\partial x^2}X_s^{t,x,\mathbf{m}}(Z) \leq C\prod_{l=t}^{s-1}(1+E_{m_l}^{(4)})(1+E_{m_l}^{(2)}).\]
\end{lemma}
\begin{proof}
    Directly,
    \begin{align*}
        \frac{\partial}{\partial x}X_s^{t,x ,\mathbf{m}}(Z) & = \big(\partial_x h_{s-1} + (\partial_u h_{s-1})(\partial_x u_{m_{s-1}})\big)\frac{\partial}{\partial x}X_{s-1}^{t,x,\mathbf{m}}(Z) \\
        & = \cdots = D_{s-1}\cdots D_{t+1}\frac{\partial}{\partial x}h_t(x, u_{m_t}(x), Z_{t+1}) \\
        & = D_{s-1}\cdots D_{t+1}D_t \leq C\prod_{l=t}^{s-1}(1+E^{(2)}_{m_l}).
    \end{align*}
    Defining $G_l := \partial_x^2 h_l + 2(\partial^2_{ux} h_l)( \partial_x u_{m_l}) + (\partial_u^2 h_l )(\partial_x u_{m_l})^2 + (\partial_u h_l)( \partial_x^2 u_{m_l}),$ for the second derivative we compute
    \begin{align*}
        \frac{\partial^2}{\partial x^2}X_s^{t,x,\mathbf{m}}(Z) & = G_{s-1}D_{s-2}^2\cdots D_t^2 + D_{s-1}G_{s-2}D_{s-3}^2\cdots D_t^2 \\
        & \quad + \cdots + D_{s-1}\cdots D_{t+2}G_{t+1}D_t^2 + D_{s-1}\cdots D_{t+1}G_t \\
        & \leq C\prod_{l=t}^{s-1}(1+E^{(4)}_{m_l})(1+E^{(2)}_{m_l}).
    \end{align*}
\end{proof}
We may now use these elementary inequalities to prove some more involved bounds, which will prove useful in bounding the terms demonstrated in Theorem \ref{generalisation error}.

\begin{lemma}\label{Q deriv x}
    There exists $C > 0$ such that, for $\mathbf{m} = (m_s)_{s=0}^{T-1}, x \in \mathcal{X}$,
    \[ \frac{\partial \widehat{Q}_t}{\partial x}(x, m_t, m_{t+1}, \ldots, m_{T-1}, Z) \leq C(1+\norm{x})\prod_{s=t}^{T-1}(1+E_{m_s}^{(4)})(1+\norm{Z_{s+1}}).\]
\end{lemma}
\begin{proof}
    Directly computing, we have
    \begin{align*}
        \frac{\partial \widehat{Q}_t}{\partial x}(x, m_t, m_{t+1}, \ldots, m_{T-1}, Z) & = \sum_{s \geq t}\big(\partial_x c_s^* + (\partial_u c_s^*)( \partial_x u_{m_s})\big)\frac{\partial}{\partial x} X_s^{t,x,\mathbf{m}}(Z) \\
        & \leq \sum_{s\geq t}(1+E_{m_s}^{(4)})(1+\norm{X_s^{t,x,\mathbf{m}}(Z)})\frac{\partial}{\partial x}X_s^{t,x,\mathbf{m}}(Z) \\
        & \leq (1+E_{m_t}^{(4)})(1+\norm{x}) \\ 
        & \quad + C\sum_{s>t}(1+E_{m_s}^{(4)})(1+\norm{X_s^{t,x,\mathbf{m}}(Z)})\prod_{l=t}^{s-1}(1+E_{m_l}^{(2)}) \\
        & \leq (1+E_{m_t}^{(4)})(1+\norm{x}) \\
        & \quad + C(1+\norm{x})\sum_{s >t }\prod_{l=t}^{s-1}(1+E_{m_l}^{(4)})(1+\norm{Z_{l+1}}) \\
        & \leq C(1+\norm{x})\prod_{s=t}^{T-1}(1+E_{m_s}^{(4)})(1+\norm{Z_{s+1}}).
    \end{align*}
\end{proof}
\begin{lemma}\label{Q frechet deriv}
    There exists $C >0$ such that, for $s \geq t, \mathbf{m} = (m_l)_{l=0}^{T-1}, x \in\mathcal{X}$, 
    \begin{align*}
        & \frac{\delta}{\delta m_s}\frac{\partial \widehat{Q}_t}{\partial x}(x, m_t, m_{t+1}, \ldots, m_{T-1}, Z; \theta) \\
        & \quad \leq C(1+\norm{x}^2)(1+\norm{\theta}^2+E_{m_s}^{(2)})(1+E_{m_s}^{(4)})\prod_{l=t}^{T-1}(1+\norm{Z_{l+1}}^2)\prod_{l=t, l\neq s}^{T-1}(1+E_{m_l}^{(8)}).
    \end{align*}
\end{lemma}
\begin{proof}
    Immediately, 
    \begin{align*}
        \frac{\delta}{\delta m_s}\frac{\partial \widehat{Q}_t}{\partial x}(x, m_t,\ldots, m_{T-1}, Z; \theta) & = \frac{\delta }{\delta m_s}\sum_{l\geq t}\big(\partial_x c_l^* + (\partial_u c_l^*)( \partial_x u_{m_l})\big)\frac{\partial}{\partial x}X_l^{t,x,\mathbf{m}}(Z) \\ 
        & =\frac{\delta}{\delta m_s}\sum_{l\geq s}\big(\partial_x c_l^* + (\partial_u c_l^*)(\partial_x u_{m_l})\big)\frac{\partial}{\partial x}X_l^{t,x,\mathbf{m}}(Z) \\
        & = \frac{\delta}{\delta m_s}\sum_{l\geq s}\big(\partial_xc_l^* + (\partial_u c_l^*)( \partial_x u_{m_l})\big)D_{l-1}\cdots D_t.
    \end{align*}
    We begin by considering the first term, 
    \begin{align*}
        & \frac{\delta}{\delta m_s}(\partial_x c_s^* + (\partial_u c_s^*)( \partial_x u_{m_s}))D_{s-1}\cdots D_t \\ & \quad = \Big(\partial^2_{xu} c_s^* + (\partial_u^2 c_s^*)( \partial_x u_{m_s}) + (\partial_uc_s^*) \partial_x\Big) (\phi(X_s^{t,x, \mathbf{m}}(Z),\theta)-\mathbb{E}_{\theta\sim m_s}[\phi(X_s^{t,x,\mathbf{m}}(Z),\theta)])D_{s-1}\cdots D_t \\
        & \quad \leq C(1+\norm{X_s^{t,x,\mathbf{m}}(Z)}^2)(1+E_{m_s}^{(2)})(1+\norm{\theta}^2 + E_{m_s}^{(2)})\prod_{l=t}^{s-1}(1+E_{m_l}^{(2)}) \\
        & \quad \leq C(1+\norm{x}^2)(1+E_{m_s}^{(2)})(1+\norm{\theta}^2+E_{m_s}^{(2)})\prod_{l=t}^{s-1}(1+E_{m_l}^{(4)})(1+\norm{Z_{l+1}}^2).
    \end{align*}
    Likewise, for $l > s$, we compute
    \begin{align*}
        & \frac{\delta}{\delta m_s}\Big(\big(\partial_x c_l^* + (\partial_u c_l^*)( \partial_x u_{m_l})\big)D_{l-1}\cdots D_t\Big) \\ 
        & \quad = \big( \partial_x^2 c_l^* + 2(\partial^2_{ux} c_l^*)( \partial_x u_{m_l}) + (\partial_u^2 c_l^*) (\partial_x u_{m_l})^2 + (\partial_u c_l^*) (\partial_x^2 u_{m_l})\big)D_{l-1}\cdots D_t \frac{\delta }{\delta m_s}X_l^{t,x, \mathbf{m}}(Z) \\
        & \quad \quad + \big(\partial_x c_l^* + (\partial_u c_l^*)( \partial_x u_{m_l})\big)\sum_{q=t}^{l-1} G_q \frac{\delta}{\delta m_s}X_q^{t,x,\mathbf{m}}(Z)\prod_{n=t, n\neq q}^{l-1} D_n.
    \end{align*}
    For the first term we find
    \begin{align*}
        & \big( \partial_x^2 c_l^* + 2(\partial^2_{ux} c_l^*) (\partial_x u_{m_l}) + (\partial_u^2 c_l^* )(\partial_x u_{m_l})^2 + (\partial_u c_l^*) (\partial_x^2 u_{m_l})\big)D_{l-1}\cdots D_t \frac{\delta }{\delta m_s}X_l^{t,x, \mathbf{m}}(Z) \\
        & \quad \leq C(1+E_{m_l}^{(8)})(1+\norm{X_l^{t,x,\mathbf{m}}(Z)})(1+E_{m_{l-1}}^{(2)})\cdots(1+E_{m_t}^{(2)}) \\
        & \quad \quad \times (1+\norm{X_s^{t,x,\mathbf{m}}(Z)})(1+\norm{\theta}^2+E_{m_s}^{(2)})(1+E_{m_{l-1}}^{(2)})\cdots (1+E_{m_{s+1}}^{(2)}) \\
        & \quad \leq C(1+E_{m_l}^{(8)})(1+\norm{x})(1+E_{m_{l-1}}^{(4)})\cdots(1+E_{m_t}^{(4)})(1+\norm{Z_{l}})\cdots(1+\norm{Z_{t+1}}) \\
        & \quad \quad \times (1+E_{m_{l-1}}^{(2)})\cdots(1+E_{m_{s+1}}^{(2)})(1+\norm{\theta}^2 + E_{m_s}^{(2)})(1+E_{m_{s-1}}^{(2)})\cdots(1+E_{m_t}^{(2)}) \\
        & \quad \quad \times (1+\norm{Z_s})\cdots(1+\norm{Z_{t+1}}) \\
        & \quad \leq C(1+\norm{x})(1+E_{m_l}^{(8)})(1+\norm{\theta}^2+E_{m_s}^{(2)})\prod_{q=t}^{l-1}(1+\norm{Z_{q+1}}^2)\prod_{q = s+1}^{l-1}(1+E_{m_{q}}^{(4)})(1+E_{m_q}^{(2)})\prod_{q=t}^{s-1}(1+E_{m_q}^{(2)}).
    \end{align*}
    Handling the second term now,
    \begin{align*}
        & \big(\partial_x c_l^* + (\partial_u c_l^*)(\partial_x u_{m_l})\big) \sum_{q=t}^{l-1} G_q \frac{\delta }{\delta m_s} X_q^{t,x,\mathbf{m}}(Z)\prod_{n=t, n\neq q}^{l-1}D_n \\
        & \quad \leq C(1+\norm{X_l^{t,x,\mathbf{m}}(Z)})(1+E_{m_l}^{(4)})\sum_{q=s+1}^{l-1}(1+E_{m_q}^{(4)})(1+E_{m_q}^{(2)})(1+\norm{X_s^{t,x,\mathbf{m}}(Z)}) \\
        & \quad \quad \times \prod_{n=s+1}^{q-1}(1+E_{m_n}^{(2)})\prod_{n=t, n\neq q}^{l-1}(1+E_{m_n}^{(2)}) \\
        & \quad \leq C(1+\norm{x}^2)(1+E_{m_l}^{(4)})\prod_{q=t}^{l-1}(1+E_{m_q}^{(2)})(1+\norm{Z_{q+1}}^2) \\
        & \quad \quad \times \sum_{q=s+1}^{l-1}(1+E_{m_q}^{(4)})(1+E_{m_q}^{(2)})(1+\norm{\theta}^2+E_{m_s}^{(2)})\prod_{n=t, n\neq s}^{q-1}(1+E_{m_n}^{(2)})\prod_{n=t, n\neq q}^{l-1}(1+E_{m_n}^{(2)}) \\
        & \quad \leq C(1+\norm{x}^2)(1+E_{m_l}^{(4)})(1+\norm{\theta}^2+E_{m_s}^{(2)})(1+E_{m_s}^{(4)})\prod_{q=t}^{l-1}(1+\norm{Z_{q+1}}^2)\prod_{q=t,q\neq s}^{l-1}(1+E_{m_q}^{(8)}).
    \end{align*}
    Bounding uniformly over both terms and all $l \geq s$ gives the claim.
\end{proof}
\begin{lemma}\label{Q double deriv x}
    There exists $C > 0$ such that, for some $t, \mathbf{m} = (m_s)_{s=0}^{T-1}, x\in\mathcal{X}$,
    \[ \frac{\partial^2 \widehat{Q}_t}{\partial x^2}(x, m_t, \ldots, m_{T-1}, Z) \leq C(1+\norm{x})\prod_{s=t}^{T-1}(1+E_{m_s}^{(8)})(1+\norm{Z_{s+1}}).\]
\end{lemma}
\begin{proof}
    Beginning directly, 
    \begin{align*}
        & \frac{\partial^2 \widehat{Q}_t}{\partial x^2}(x, m_t,\ldots, m_{T-1}, Z) \\ & \quad = \frac{\partial}{\partial x}\sum_{s\geq t}\big(\partial_x c_s^* + (\partial_u c_s^*)(\partial_x u_{m_s})\big)\frac{\partial}{\partial x}X_s^{t,x,\mathbf{m}}(Z) \\
        & \quad = \sum_{s\geq t}\Big\{\big(\partial_x^2 c_s^* + 2(\partial^2_{ux} c_s^*)( \partial_x u_{m_s}) + (\partial_u^2 c_s^*)(\partial_x u_{m_s})^2 + (\partial_u c_s^* )(\partial_x^2 u_{m_s})\big)\Big(\frac{\partial}{\partial x}X_s^{t,x,\mathbf{m}}(Z)\Big)^2 \\
        & \quad \quad + \big(\partial_x c_s^* + (\partial_u c_s^*)( \partial_x u_{m_s})\big)\frac{\partial^2}{\partial x^2}X_s^{t,x,\mathbf{m}}(Z)\Big\}.
    \end{align*}
    From Lemmas \ref{u derivs} and \ref{x state derivs} we see that 
    \begin{align*}
        & \big(\partial_x^2 c_s^* + 2(\partial^2_{ux} c_s^*)( \partial_x u_{m_s}) + (\partial_u^2 c_s^*)(\partial_x u_{m_s})^2 + (\partial_u c_s^*)( \partial_x^2 u_{m_s})\big)\Big(\frac{\partial}{\partial x}X_s^{t,x,\mathbf{m}}(Z)\Big)^2 \\
        & \quad \leq C(1+\norm{X_s^{t,x,\mathbf{m}}(Z)})(1+E_{m_s}^{(8)})\prod_{l=t}^{s-1}(1+E_{m_l}^{(4)}) \\
        & \quad \leq C(1+\norm{x})(1+E_{m_s}^{(8)})\prod_{l=t}^{s-1}(1+E_{m_l}^{(4)})(1+E_{m_l}^{(2)})(1+\norm{Z_{l+1}}),
    \end{align*}
    and 
    \begin{align*}
        & \big(\partial_x c_s^* + (\partial_u c_s^*)( \partial_x u_{m_s})\big)\frac{\partial^2}{\partial x^2}X_s^{t,x,\mathbf{m}}(Z) \\
        & \quad \leq C(1+\norm{X_s^{t,x,\mathbf{m}}(Z)})(1+E_{m_s}^{(4)})\prod_{l=t}^{s-1}(1+E_{m_l}^{(4)})(1+E_{m_l}^{(2)}) \\
        & \quad \leq C(1+\norm{x})(1+E_{m_s}^{(4)})\prod_{l=t}^{s-1}(1+E_{m_l}^{(8)})(1+\norm{Z_{l+1}}).
    \end{align*}
    Bounding uniformly over both terms and all $s \geq t$ gives
    \begin{align*}
        \frac{\partial^2 \widehat{Q}_t}{\partial x^2}(x, m_t, \ldots, m_{T-1}, Z) \leq C(1+\norm{x})\prod_{s=t}^{T-1}(1+E_{m_s}^{(8)})(1+\norm{Z_{s+1}}).
    \end{align*}
\end{proof}

    \section{Results for Mean-Field Neural Networks}\label{appB}

We here state auxiliary results, adapted slightly from \cite{sam_ge}. The proofs are almost identical, and so are omitted. Where we mention some $\nu$, we take $\nu\in\mathcal{P}_q(\mathcal{Z}^T)$, with $q$ as in Theorem \ref{minimisers}.
\begin{lemma}\label{set-valued map}(\cite[Lemma D.1]{sam_ge})
    For $t= 0, \ldots, T-1,$ define the set-valued map
    \begin{align*}
        B_t(\nu) := & \Big\{m_t\in\mathcal{P}_p(\Theta):\frac{\sigma^2}{2\beta^2}\mathrm{KL}(m_t||\gamma^{\sigma})\leq\mathbb{E}_{Z\sim\nu}[\widehat{Q}_t(X_t^{\mathrm{ref}}(Z),\gamma^{\sigma}, Z)]\\
        & \quad \mathrm{and}\ \int_{\Theta}\norm{\theta}^p m_t(\mathrm{d}\theta) \leq \mathbb{E}_{Z\sim\nu}[\widehat{Q}_t(X_t^{\mathrm{ref}}(Z),\widetilde{\gamma}^{\sigma}_p, Z)] + \int_{\Theta}\norm{\theta}^p \widetilde{\gamma}^{\sigma}_p(\theta)\mathrm{d}\theta \Big\}.
    \end{align*}
    Then, for all $t = 0, \ldots, T-1$, the relevant component of the Gibbs vector satisfies $\mathfrak{m}_t(\nu)\in B_t(\nu).$
\end{lemma}
\begin{lemma}(\cite[Lemma D.2]{sam_ge})\label{Hilbert-Schmidt}
    For $t = 0,\ldots, T-1$, the linear maps $\mathcal{C}^t_m : L^2(m,\Theta) \to L^2(m,\Theta)$ defined by
    \[ \mathcal{C}_m^t f(\theta) := \mathbb{E}_{\theta'\sim m}\Big[\int_{\mathcal{Z}^T}\frac{\delta^2 \widehat{Q}_t}{\delta m^2}(X_t^{\mathrm{ref}}(Z), m, Z;\theta, \theta')\nu(\mathrm{d}Z)f(\theta')\Big], \quad m \in B_t(\nu),\]
    are positive in the sense that $\langle f, \mathcal{C}^t_m f\rangle_{L^2(m,\Theta)} \geq 0$. 
    \par
    In particular, each $\mathcal{C}_m^t$ is a Hilbert-Schmidt operator with a discrete spectrum \[\sigma(\mathcal{C}_m^t) = \{\lambda_i^t\}_{i\geq 0} \subset [0,\infty).\]
\end{lemma}
\begin{lemma}(\cite[Lemma D.3]{sam_ge})
    We define 
    \[ S_t(\nu,\theta):=\int_{\mathcal{Z}^{T}}\frac{\delta \widehat{Q}_t}{\delta m}(X_t^{\mathrm{ref}}(Z),\mathfrak{m}_t(\nu),Z;\theta)\nu(\mathrm{d}Z).\]
    Under Assumption \ref{verifiable assumptions}, each $S_t$ is differentiable with respect to $\nu$. In particular, explicitly
    \begin{align*}
        \frac{\delta S_t}{\delta \nu}(\nu,\theta; Z) & = \frac{\delta \widehat{Q}_t}{\delta m}(X_t^{\mathrm{ref}}(Z), \mathfrak{m}_t(\nu),Z;\theta) - \int_{\mathcal{Z}^T} \frac{\delta \widehat{Q}_t}{\delta m}(X_t^{\mathrm{ref}}(Z'), \mathfrak{m}_t(\nu),Z';\theta)\nu(\mathrm{d}Z') \\
        & \quad - \frac{2\beta^2}{\sigma^2}\mathbb{C}\mathrm{ov}_{\theta'\sim \mathfrak{m}_t(\nu)}\Big[\int_{\mathcal{Z}^{T}} \frac{\delta^2 \widehat{Q}_t}{\delta m^2}(X_t^{\mathrm{ref}}(Z'), \mathfrak{m}_t(\nu),Z';\theta,\theta')\nu(\mathrm{d}Z'), \frac{\delta S_t}{\delta \nu}(\nu, \theta'; Z)\Big].
    \end{align*}
    Even further, we have the inequality 
    \begin{align*}
    & \int_{\Theta}\Big(\frac{\delta S_t}{\delta \nu}(\nu,\theta; Z\Big)^2 \mathfrak{m}_t(\nu)(\mathrm{d}\theta) \\
    & \quad \leq \int_{\Theta}\Big(\frac{\delta \widehat{Q}_t}{\delta m}(X_t^{\mathrm{ref}}(Z), \mathfrak{m}_t(\nu),Z;\theta) - \int_{\mathcal{Z}^T} \frac{\delta \widehat{Q}_t}{\delta m}(X_t^{\mathrm{ref}}(Z'), \mathfrak{m}_t(\nu),Z';\theta)\nu(\mathrm{d}Z')\Big)^2 \mathfrak{m}_t(\nu)(\mathrm{d}\theta).
    \end{align*}
\end{lemma}

\begin{lemma}(\cite[Lemma D.4]{sam_ge})\label{cov}
    Under Assumption \ref{verifiable assumptions}, the densities of the components of the Gibbs vector have derivatives
    \[ \frac{\delta \mathfrak{m}_t}{\delta \nu}(\nu, \theta; Z) = - \frac{2\beta^2}{\sigma^2}\mathfrak{m}_t(\nu)(\theta)\Big(\frac{\delta S_t}{\delta \nu}(\nu,\theta; Z) - \int_{\Theta}\mathfrak{m}_t(\nu)(\theta')\frac{\delta S_t}{\delta \nu}(\nu, \theta'; Z)\mathrm{d}\theta'\Big).\]
    In particular, for any $f \in L^2(\mathrm{d}\theta)$, we have the inner product representation
    \[ \int_{\Theta} f(\theta)\frac{\delta \mathfrak{m}_t}{\delta \nu}(\nu,\theta;Z)\mathrm{d}\theta = - \frac{2\beta^2}{\sigma^2}\mathbb{C}\mathrm{ov}_{\theta\sim\mathfrak{m}_t(\nu)}\Big[f(\theta), \frac{\delta S_t}{\delta \nu}(\nu,\theta;Z)\Big].\]
\end{lemma}
\begin{lemma}\label{S_t expectation inequality}
    For general $t, m$, we have 
    \begin{align*}& \Bigg(\int_\Theta \Big(\frac{\delta S_t}{\delta \nu}\big(\nu, \theta; Z\big)\Big|^{Z=Z^{(1)}}_{Z=\widetilde{Z}^{(1)}}\Big)^2 \mathfrak{m}_t(\nu; \mathrm{d}\theta)\Bigg)^{\frac{1}{2}} \\ & \quad \leq C (1+\norm{x_0}^2) \Big(P(Z^{(1)})^2 + P(\widetilde{Z}^{(1)})^2\Big)(1+E_{\mathfrak{m}_t(\nu)}^{(2)})\prod_{s=t}^{T-1}(1+E_{\mathfrak{m}_s(\nu)}^{(4)})
    \end{align*}
\end{lemma}
\begin{proof}
    As demonstrated in Lemma \ref{Hilbert-Schmidt}, we may write 
    \begin{align*}
        \frac{\delta S_t}{\delta \nu}\big(\nu, \theta; Z\big)\Big|^{Z=Z^{(1)}}_{Z=\widetilde{Z}^{(1)}} & = \Big(\mathrm{id}+\frac{2\beta^2}{\sigma^2}\mathcal{C}^t_{\mathfrak{m}_t(\nu)}\Big)^{-1}\frac{\delta \widehat{Q}_t}{\delta m_t}\big(X_t^{\mathrm{ref}}(Z), \mathfrak{m}_{t:T-1}(\nu), Z; \theta\big)\Big|_{Z=\widetilde{Z}^{(1)}}^{Z=Z^{(1)}}, 
    \end{align*}
    so then 
    \begin{align*}
    \int_\Theta \Big(\frac{\delta S_t}{\delta \nu}\big(\nu, \theta; Z\big)\Big|^{Z=Z^{(1)}}_{Z=\widetilde{Z}^{(1)}}\Big)^2\mathfrak{m}_t(\nu; \mathrm{d}\theta) & \leq 2\int_\Theta \Bigg(\frac{\delta \widehat{Q}_t}{\delta m_t}\big(X_t^{\mathrm{ref}}(Z^{(1)}), \mathfrak{m}_{t:T-1}(\nu),Z^{(1)};\theta\big)^2 \\
    & \quad \quad + \frac{\delta \widehat{Q}_t}{\delta m_t}\big(X_t^{\mathrm{ref}}(\widetilde{Z}^{(1)}), \mathfrak{m}_{t:T-1}(\nu),\widetilde{Z}^{(1)};\theta\big)^2\Bigg)\mathfrak{m}_t(\nu; \mathrm{d}\theta).
    \end{align*}
    Denoting $\mathbf{m} = (m_s)_{s=0}^{T-1}$, we now bound
    \begin{align*}
        & \frac{\delta \widehat{Q}_t}{\delta m_t}\big(X_t^{\mathrm{ref}}(Z), m_{t:T-1},Z;\theta\big) \\ & = \frac{\delta}{\delta m_t}\sum_{s\geq t} c_s^*(X_s^{t, \mathbf{m}}(Z), m_s) \\
        & = (\partial_u c_t^*) \big(\phi(X_t^{\mathrm{ref}}(Z),\theta)-\mathbb{E}_{\theta\sim m_t}[\phi(X_t^{\mathrm{ref}}(Z), \theta)]\big) \\
        & \quad + \sum_{s=t+1}^{T-1}\big(\partial_x c_s^* + \partial_u c_s^* \partial_x u_{m_s} \big)\frac{\delta}{\delta m_t}X_s^{t,\mathbf{m}}(Z) \\
        & \leq C(1+\norm{X_t^{\mathrm{ref}}(Z)}^2)(1+E_{m_t}^{(2)})(1+\norm{\theta}^2 + E_{m_t}^{(2)}) \\
        & \quad + C\sum_{s=t+1}^{T-1}(1+\norm{X_s^{t,\mathbf{m}}(Z)})(1+E_{m_s}^{(4)})(1+\norm{X_t^{\mathrm{ref}}(Z)})(1+\norm{\theta}^2 + E_{m_t}^{(2)})\prod_{l=t+1}^{s-1} (1+E_{m_l}^{(2)}) \\
        & \leq C(1+\norm{x_0}^2)P(Z)^2(1+E_{m_t}^{(2)})(1+\norm{\theta}^2+E_{m_t}^{(2)})\prod_{s= t+1}^{T-1}(1+E_{m_s}^{(4)}). \\
    \end{align*}
    Substituting this into the above, taking the square root and simplifying gives
    \begin{align*}
    & \Bigg(\int_\Theta \Big(\frac{\delta S_t}{\delta \nu}\big(\nu, \theta; Z\big)\Big|^{Z=Z^{(1)}}_{Z=\widetilde{Z}^{(1)}}\Big)^2\mathfrak{m}_t(\nu; \mathrm{d}\theta)\Bigg)^{\frac{1}{2}} \\
    & \leq C (1+\norm{x_0}^2)\Big(P(Z^{(1)})^2 + P(\widetilde{Z}^{(1)})^2\Big)(1+E_{\mathfrak{m}_t(\nu)}^{(2)})\prod_{s=t}^{T-1}(1+E_{\mathfrak{m}_s(\nu)}^{(4)}).
    \end{align*}
    
\end{proof}

\newpage
\section{Further Zermelo Results}\label{zermelo}
Below we display the training over time for the Zermelo problem discussed in Section \ref{numerics}. Over the 50 time steps of the problem, we display the results for times $t = 50, 44, 39, 34, 29, 24, 19, 14, 9, 4, 0$. 
\par\vspace{1em}
\noindent\begin{minipage}[b]{0.5\textwidth}
    \centering
    \includegraphics[width=\linewidth]{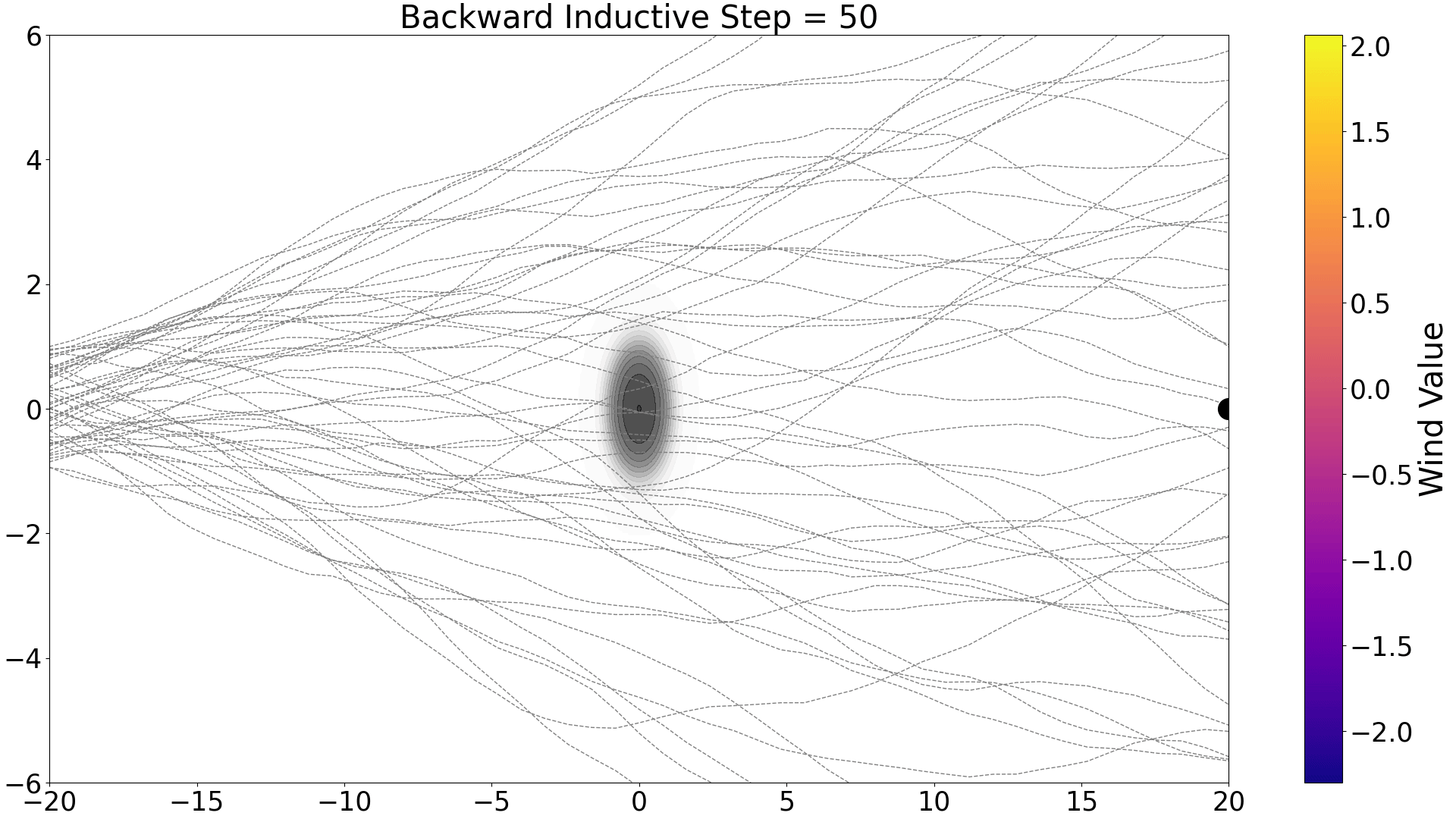}
\end{minipage}%
\begin{minipage}[b]{0.5\textwidth}
    \centering
    \includegraphics[width=\linewidth]{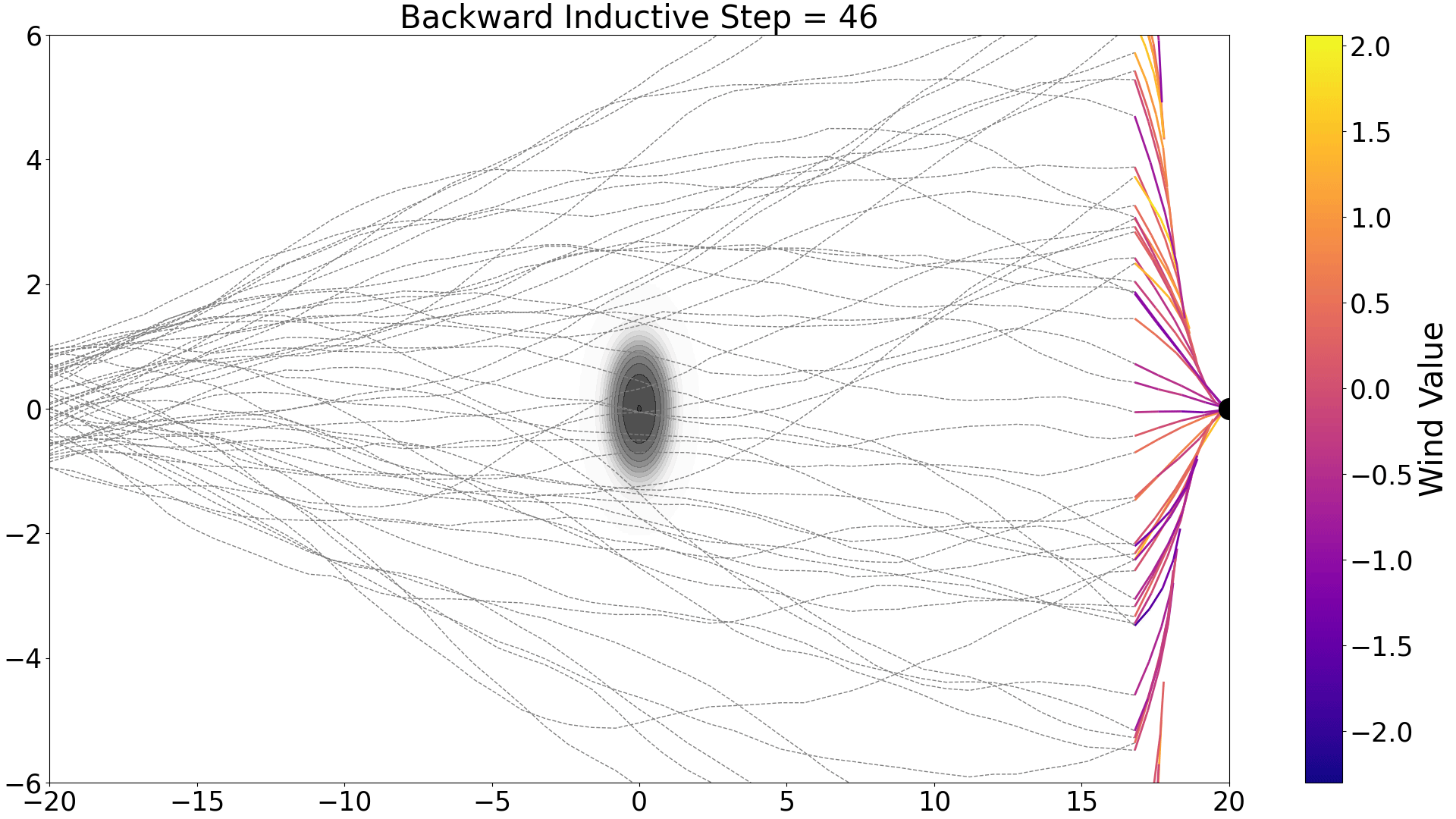}
\end{minipage}

\vspace{0.5em} 

\noindent\begin{minipage}[b]{0.5\textwidth}
    \centering
    \includegraphics[width=\linewidth]{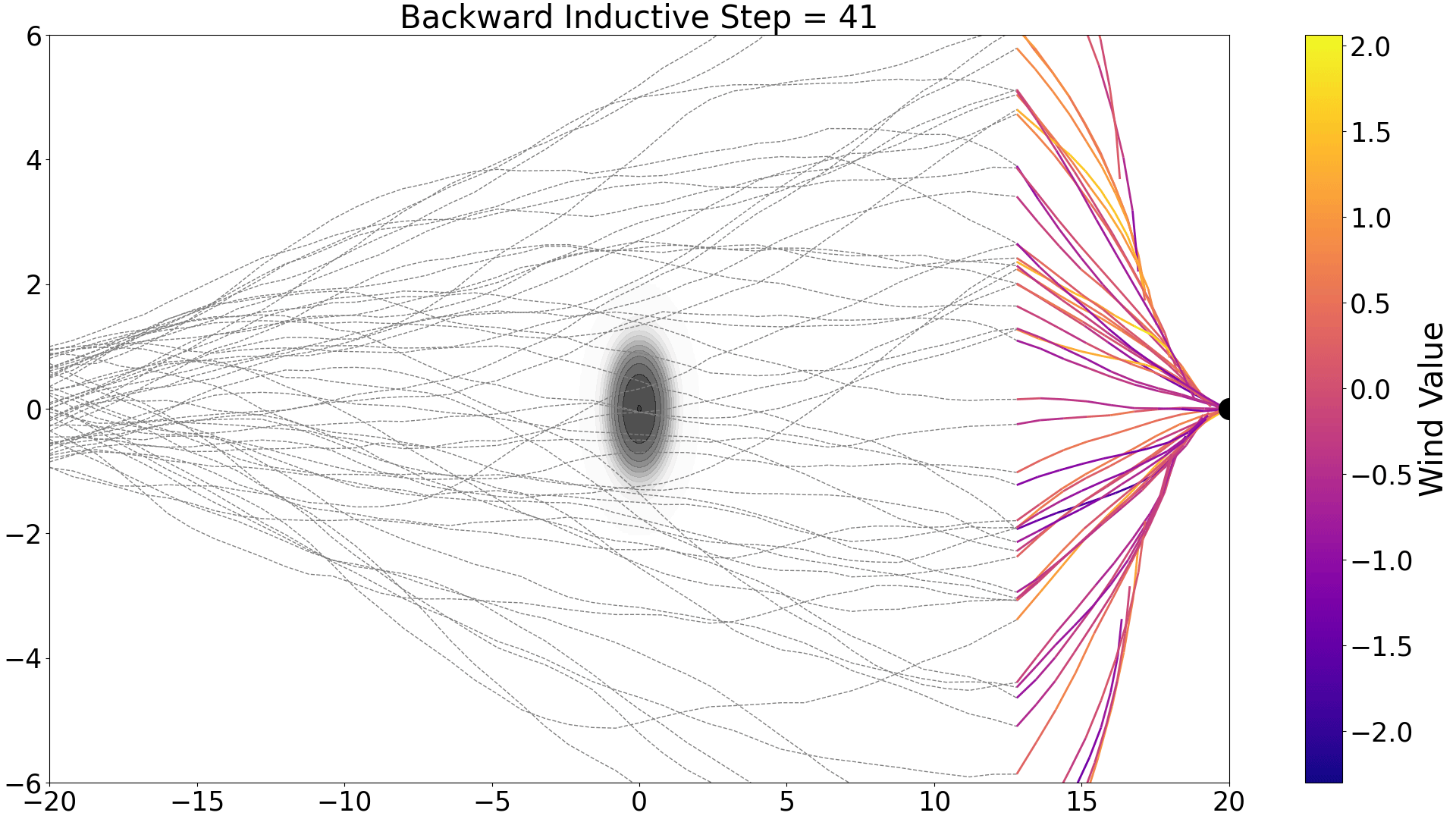}
\end{minipage}%
\begin{minipage}[b]{0.5\textwidth}
    \centering
    \includegraphics[width=\linewidth]{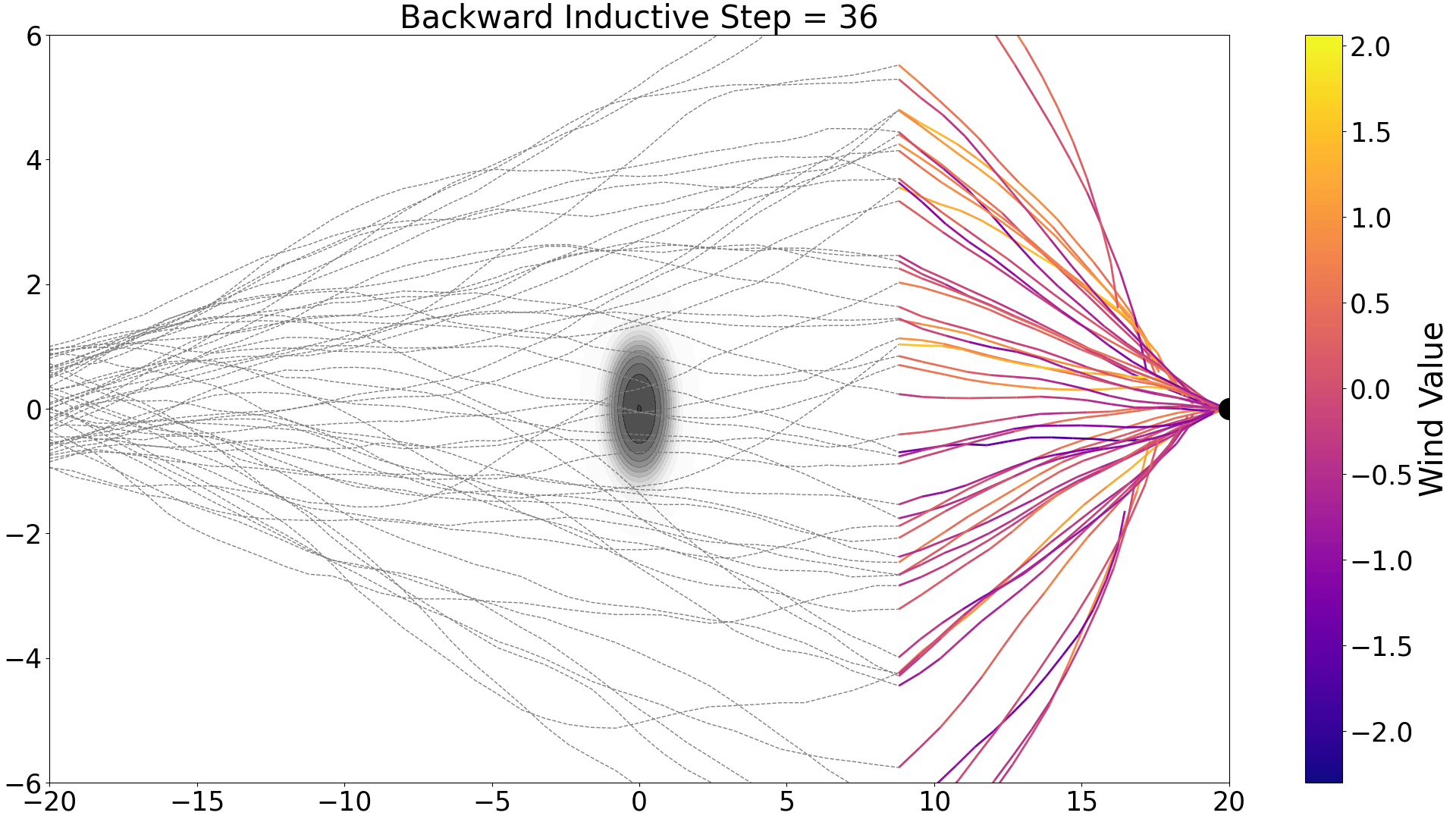}
\end{minipage}

\vspace{0.5em}

\noindent\begin{minipage}[b]{0.5\textwidth}
    \centering
    \includegraphics[width=\linewidth]{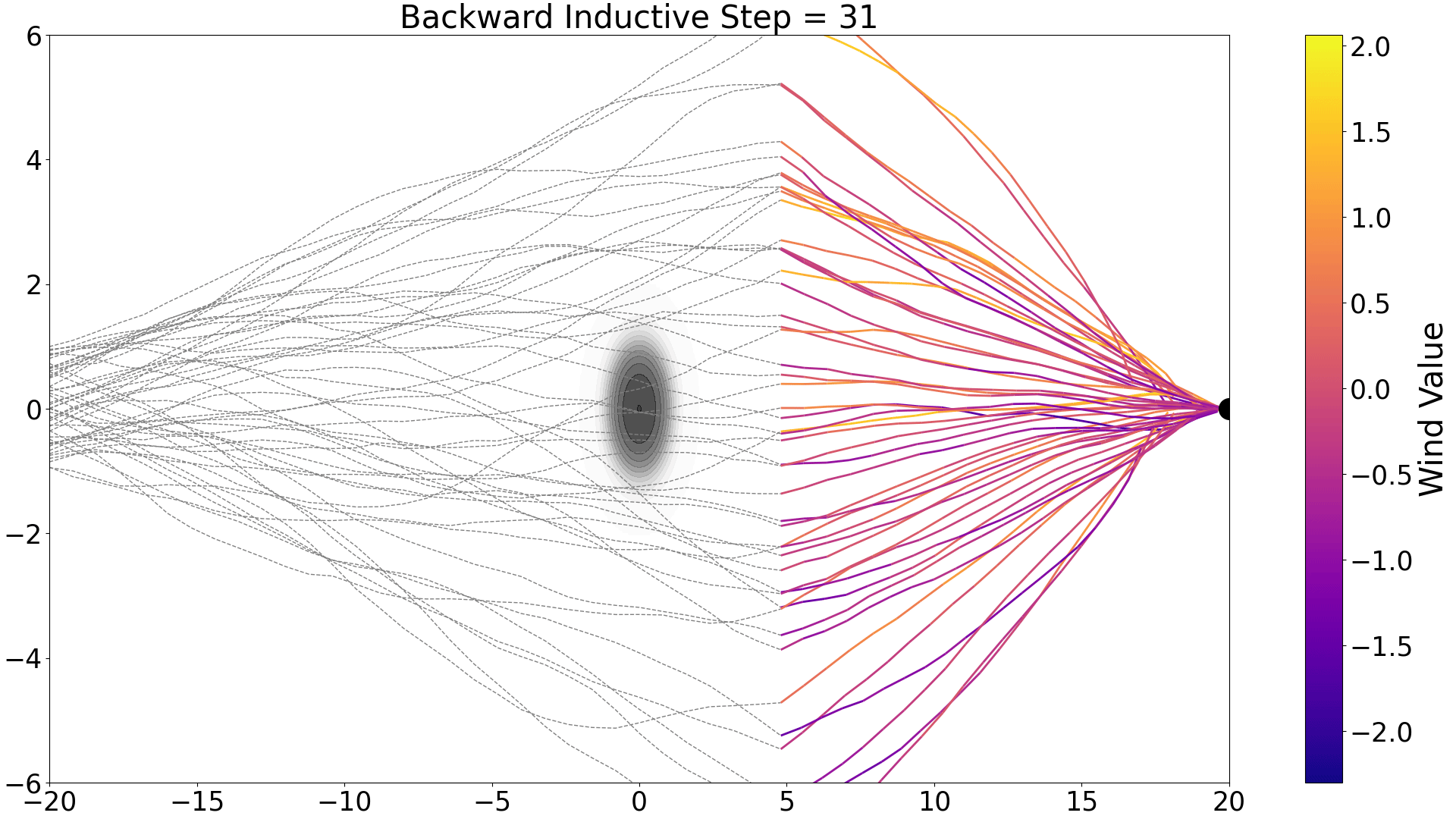}
\end{minipage}%
\begin{minipage}[b]{0.5\textwidth}
    \centering
    \includegraphics[width=\linewidth]{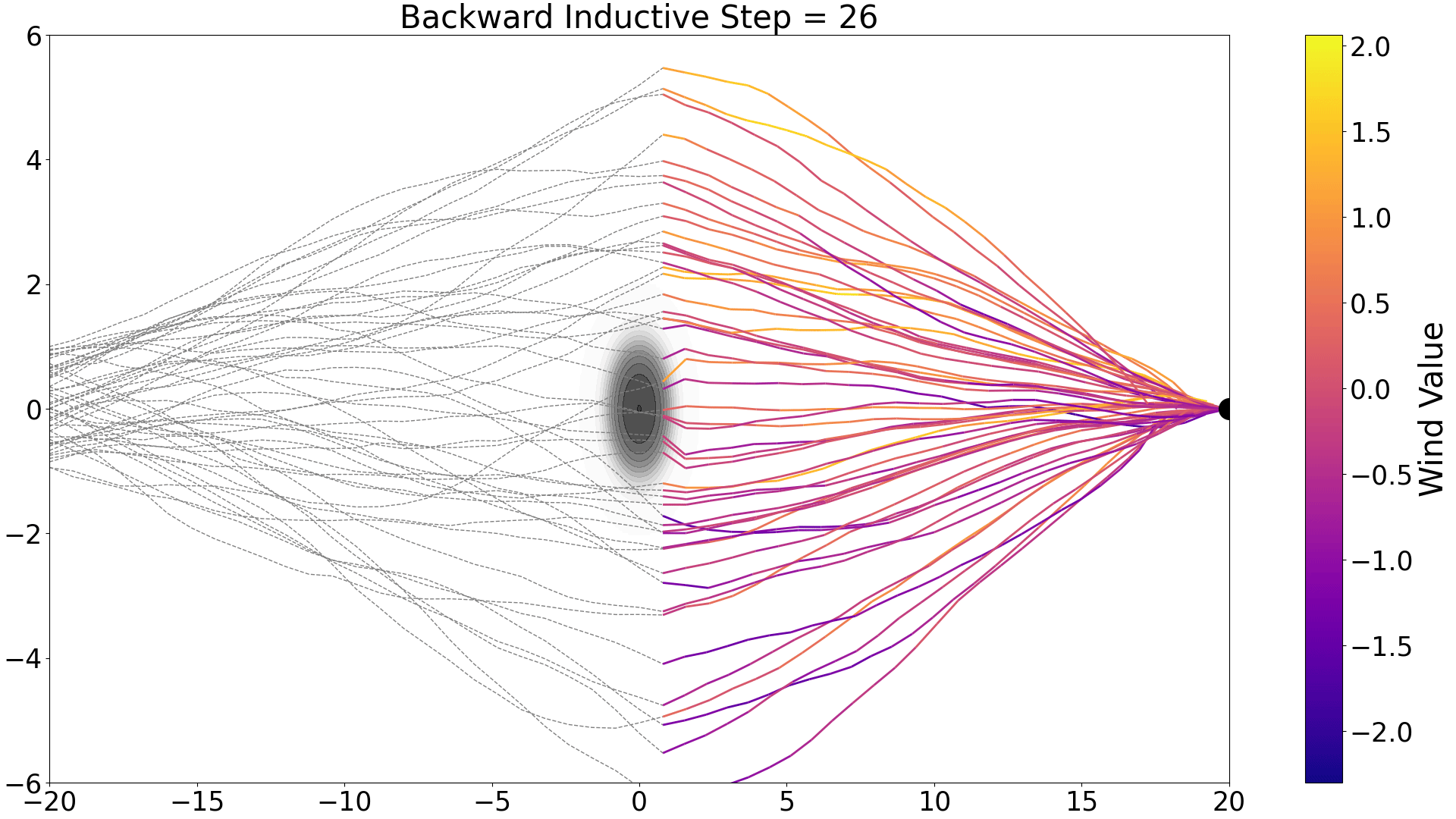}
\end{minipage}

\vspace{0.5em}

\noindent\begin{minipage}[b]{0.5\textwidth}
    \centering
    \includegraphics[width=\linewidth]{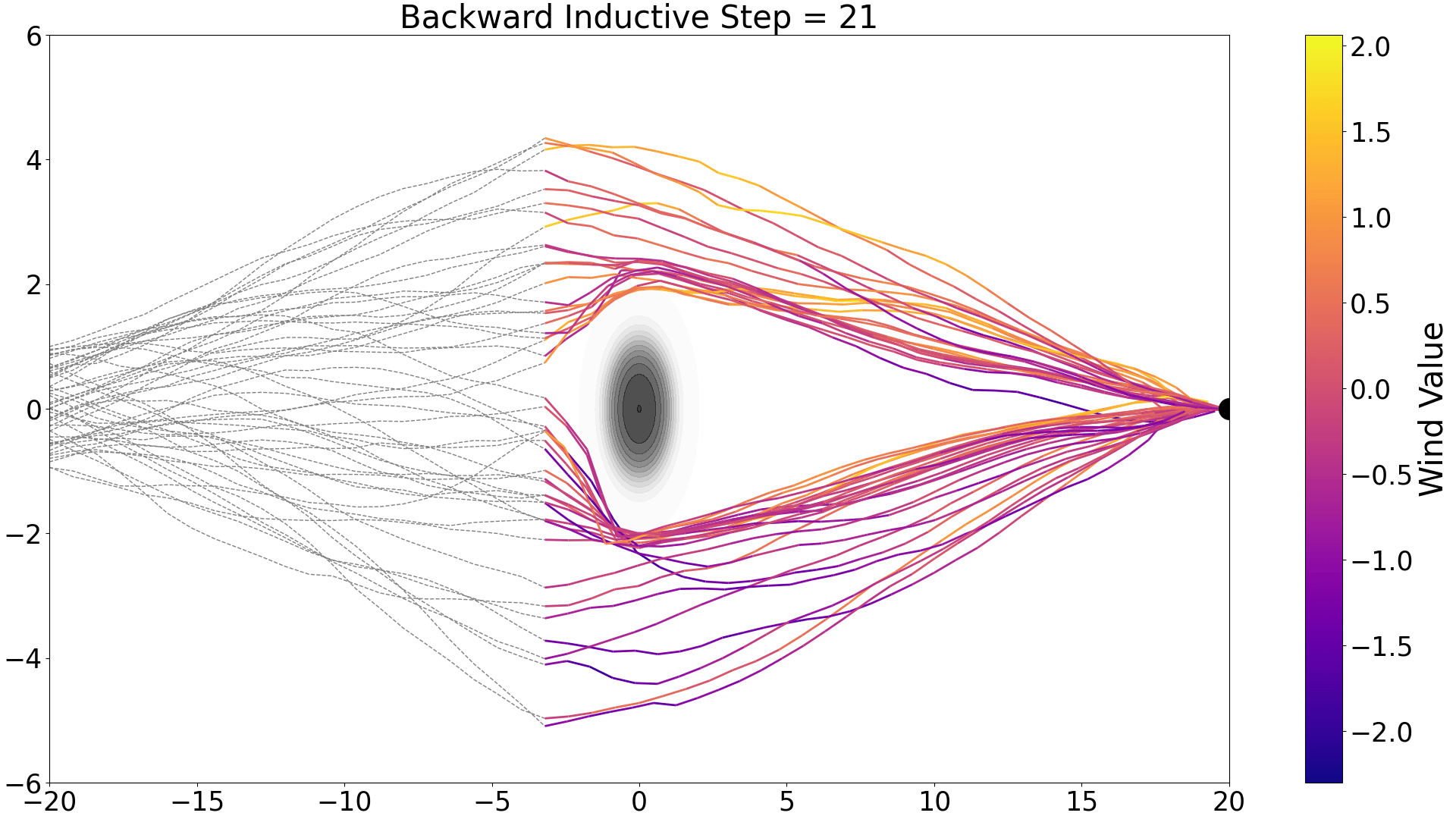}
\end{minipage}%
\begin{minipage}[b]{0.5\textwidth}
    \centering
    \includegraphics[width=\linewidth]{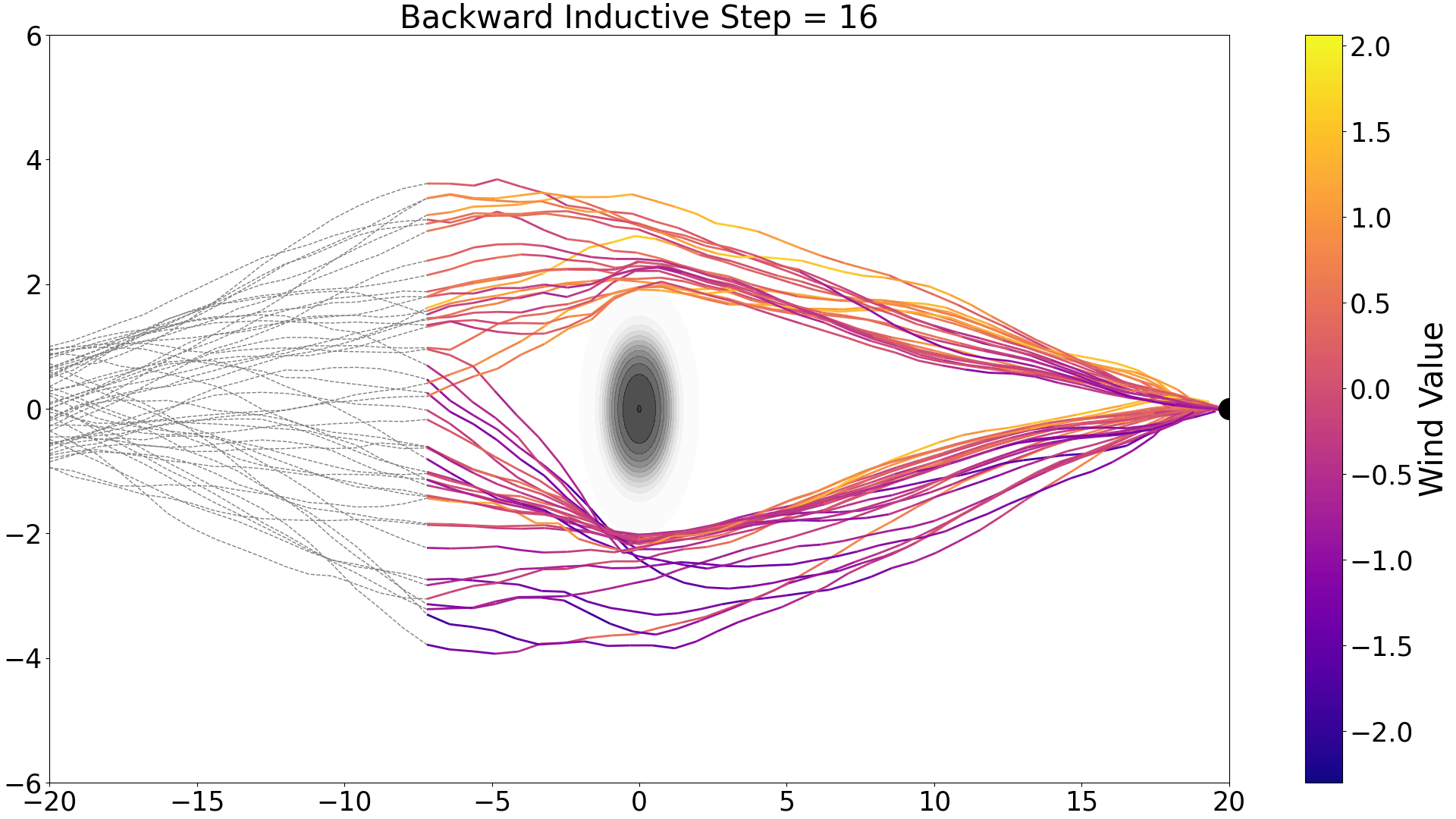}
\end{minipage}

\vspace{0.5em}

\noindent\begin{minipage}[b]{0.5\textwidth}
    \centering
    \includegraphics[width=\linewidth]{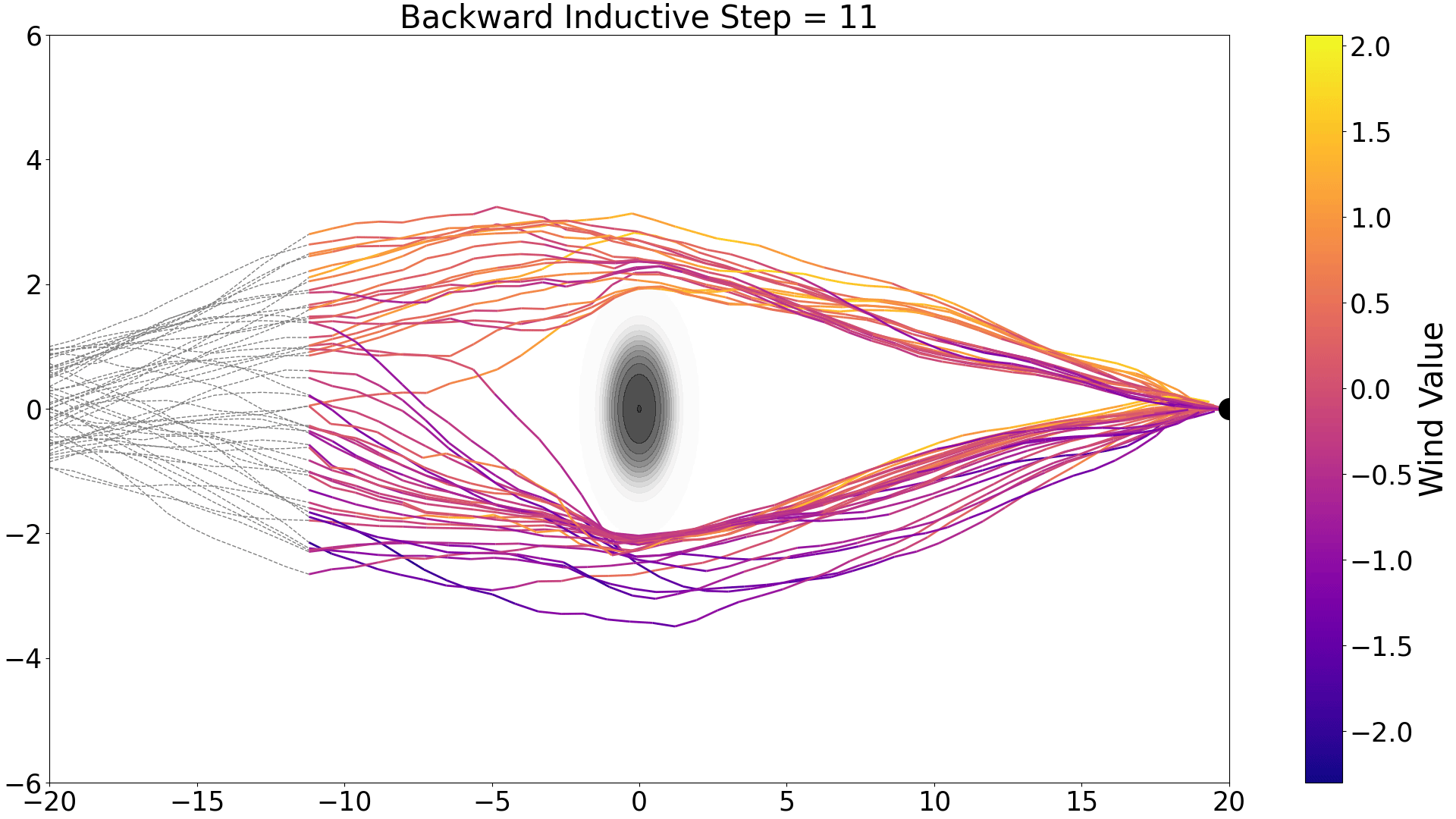}
\end{minipage}%
\begin{minipage}[b]{0.5\textwidth}
    \centering
    \includegraphics[width=\linewidth]{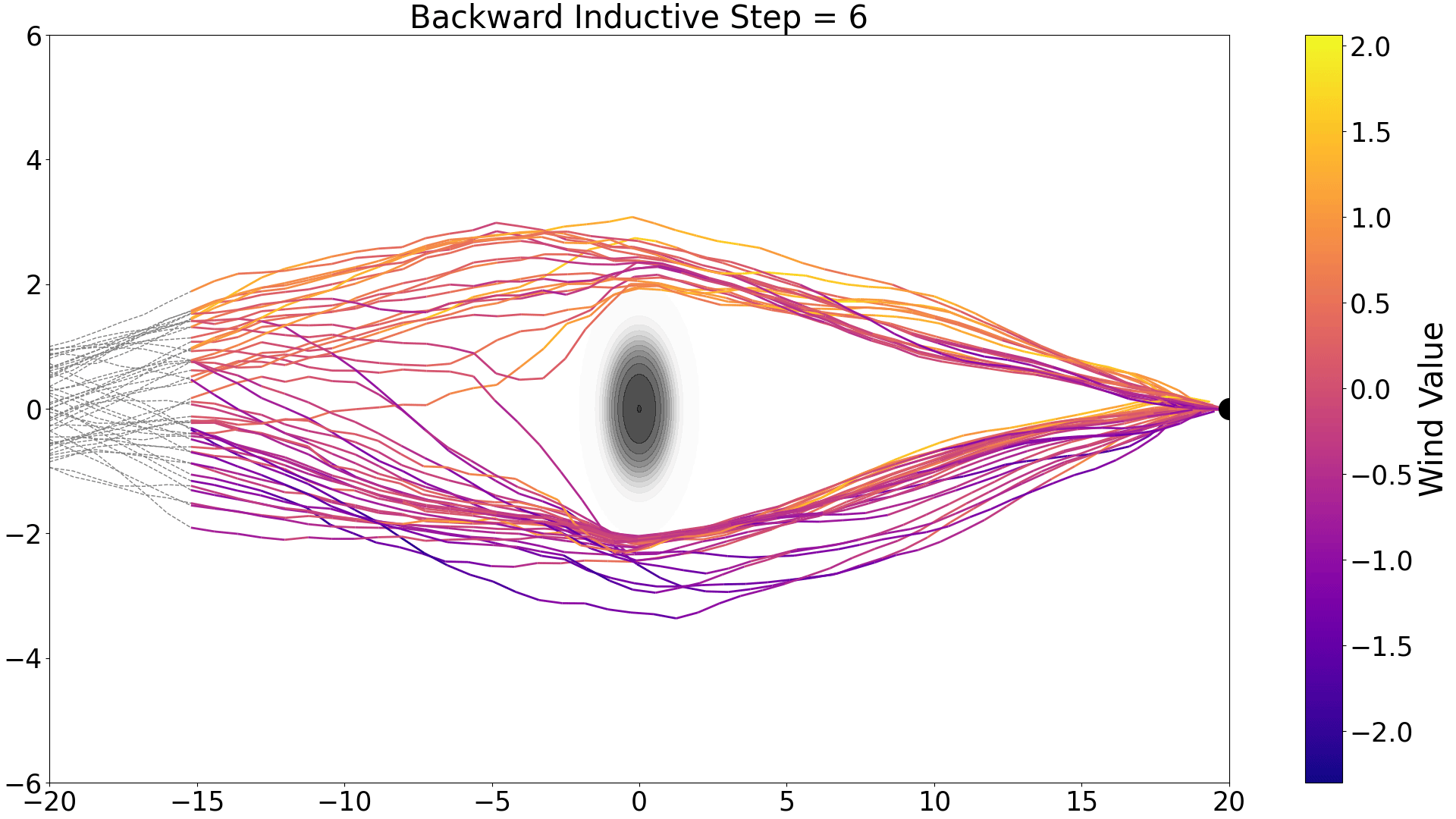}
\end{minipage}
\par\vspace{1em}
\noindent\begin{minipage}[b]{\textwidth}
    \centering
    \includegraphics[width=\linewidth]{images/gif_split/frame_50_delay-0.1s.png}
\end{minipage}%

\begin{center}
  \parbox{0.9\textwidth}{
    Figure 6: Progress of the Gibbs vector algorithm on the Zermelo navigation problem, 
    displayed for the first 50 training samples out of 100.}
\end{center}

\end{appendix}
\end{document}